\documentclass[12pt,english,sumlimits,reqno,oneside]{amsart}

\usepackage[T1]{fontenc}
\usepackage[utf8]{inputenc}
\usepackage{lmodern}

\usepackage{babel}

\usepackage[a4paper,margin=2.5cm]{geometry}

\usepackage{xifthen}%provides ifthenelse command
\usepackage{ifthenx}%complements xifthen. MUST be loaded after xifthen

\usepackage{xargs}

\usepackage{graphicx}
\usepackage[table,svgnames,x11names]{xcolor}

\usepackage[disable,
colorinlistoftodos,prependcaption,textsize=tiny,textwidth=3.5cm]{todonotes}
\presetkeys{todonotes}{fancyline}{}
\newcommandx{\todoin}[2][1=]{\todo[inline, caption={todo}, #1]{%
    \begin{minipage}{\textwidth-20pt}#2\end{minipage}}}

\newcommandx{\remove}[2][1=]{\todo[linecolor=Plum,backgroundcolor=Plum!25,bordercolor=Plum,#1]{#2}}
\newcommandx{\removein}[2][1=]{\remove[inline, caption={remove}, #1]{%
    \begin{minipage}{\textwidth-20pt}#2\end{minipage}}}

\newcommandx{\remark}[2][1=]{\todo[linecolor=red,backgroundcolor=red!25,bordercolor=red,#1]{#2}}
\newcommandx{\remarkin}[2][1=]{\remark[inline, caption={remark}, #1]{%
    \begin{minipage}{\textwidth-20pt}#2\end{minipage}}}

\newcommandx{\suggestion}[2][1=]{\todo[linecolor=yellow,backgroundcolor=yellow!25,bordercolor=yellow,#1]{#2}}
\newcommandx{\suggestionin}[2][1=]{\suggestion[inline, caption={suggestion}, #1]{%
    \begin{minipage}{\textwidth-20pt}#2\end{minipage}}}

\usepackage{amssymb}
\usepackage{amsmath}
\usepackage{amsthm}
\usepackage{mathtools}

\usepackage{exscale}%to scale mathematical symbols
\usepackage{relsize}%to relatively adjust size (also works with math symbols)
\usepackage{bbm}%blackboard symbols
\usepackage{bm}
\usepackage{amsbsy}
\usepackage{mathdots}%provides \iddots

\usepackage{mathrsfs}%provides \mathscr --> eulerscript

\usepackage{stmaryrd}%provides double brackets \llbracket

\usepackage{mdframed}

\usepackage[all]{xy}
\usepackage{fp}

\usepackage[section]{placeins}
\usepackage{float}

\usepackage{enumerate}
\newcounter{proof}
\newenvironment{myproof}%
{\stepcounter{proof}\begin{proof}}%
{\end{proof}}%
\newcounter{proofstep}[proof]
{\refstepcounter{proofstep}\bigskip\par\noindent%
  \ifthenelse{\isempty{#1}}%if
    {\textit{Step \theproofstep. }}%then
    {\textit{#1.}}%else
  \noindent}%
{\par}%
\newcounter{proofstep-noskip}[proof]
\newenvironment{proofstep-noskip}[1][]%
{\refstepcounter{proofstep-noskip}
  \ifnum\the\value{proofstep-noskip}>1
    \bigskip\par\noindent
  \fi
  \ifthenelse{\isempty{#1}}%if
    {\textit{Step \theproofstep-noskip. }}%then
    {\textit{#1.}}%else
  \noindent}%
{\par}%
\newcounter{proofcase}[proof]
{\refstepcounter{proofcase}\bigskip\par\noindent%
  \ifthenelse{\isempty{#1}}%if
    {\textit{Case \theproofcase. }}%then
    {\textit{#1.}}%else
  \noindent}%
{\par}%

\newcounter{proofcase-noskip}[proof]
\newenvironment{proofcase-noskip}[1][]%
{\refstepcounter{proofcase-noskip}
  \ifnum\the\value{proofcase-noskip}>1
    \bigskip\par\noindent
  \fi
  \ifthenelse{\isempty{#1}}%if
    {\textit{Case \theproofcase-noskip. }}%then
    {\textit{#1.}}%else
  \noindent}%
{\par}%

\usepackage{varioref}
\usepackage{hyperref}
\hypersetup{
  colorlinks,
  allcolors=DarkBlue
}
\usepackage[nameinlink]{cleveref}%must be loaded after all other packages

\theoremstyle{plain}
\newtheorem{thm}{Theorem}[section]
\newtheorem*{thm*}{Theorem}
\newtheorem{pro}[thm]{Proposition}
\newtheorem{cor}[thm]{Corollary}
\newtheorem{lem}[thm]{Lemma}

\theoremstyle{definition}
\newtheorem{dfn}[thm]{Definition}
\theoremstyle{remark}
\newtheorem{ntn}[thm]{Notation}
\newtheorem{rem}[thm]{Remark}

\numberwithin{equation}{section}

\newcommandx{\textref}[2][1=]{\hyperref[#2]{#1\ref*{#2}}}
\newcommandx{\textrefp}[2][1=]{(\hyperref[#2]{#1\ref*{#2}})}

\newcommand{\dif}{\ensuremath{\, \mathrm d}}

\DeclareMathOperator{\spn}{span}

\DeclareMathOperator*{\ave}{ave}
\DeclareMathOperator{\cond}{\mathbb{E}}

\newcommand{\vvvert}{{\vert\kern-0.25ex\vert\kern-0.25ex\vert}}

\begin{document}

\title[Factorization in Haar system Hardy spaces]%
{Factorization in Haar system Hardy spaces}%

\author[R.~Lechner]{Richard Lechner}%
\address{Richard Lechner, Institute of Analysis, Johannes Kepler University Linz, Altenberger
  Strasse 69, A-4040 Linz, Austria}%
\email{Richard.Lechner@jku.at}%

\author[T.~Speckhofer]{Thomas Speckhofer}%
\address{Thomas Speckhofer, Institute of Analysis, Johannes Kepler University Linz, Altenberger
  Strasse 69, A-4040 Linz, Austria}%
\email{Thomas.Speckhofer@jku.at}%

\date{\today}%



\thanks{The first author was supported by the Austrian Science Foundation~(FWF) project no.~P32728.\\
\indent The second author was supported by the Austrian Science Foundation~(FWF) project no.~I5231.}

\begin{abstract}
  A Haar system Hardy space is the completion of the linear span of the Haar system~$(h_I)_I$,
  either under a rearrangement-invariant norm~$\|\cdot \|$ or under the associated square function norm
  \begin{equation*}
    \Bigl\| \sum_Ia_Ih_I \Bigr\|_{*} = \Bigl\| \Bigl( \sum_I a_I^2 h_I^2 \Bigr)^{1/2} \Bigr\|.
  \end{equation*}
  Apart from $L^p$, $1\le p<\infty$, the class of these spaces includes all separable
  rearrangement-invariant function spaces on~$[0,1]$ and also the dyadic Hardy space~$H^1$.  Using a
  unified and systematic approach, we prove that a Haar system Hardy space~$Y$ with $Y\ne C(\Delta)$
  ($C(\Delta)$ denotes the continuous functions on the Cantor set) has the following properties, which
  are closely related to the primariness of~$Y$: For every bounded linear operator~$T$ on~$Y$, the
  identity~$I_Y$ factors either through~$T$ or through $I_Y - T$, and if $T$ has large diagonal with
  respect to the Haar system, then the identity factors through~$T$. In particular, we obtain that
  \begin{equation*}
    \mathcal{M}_Y = \{ T\in \mathcal{B}(Y) : I_Y \ne ATB\text{ for all } A, B\in \mathcal{B}(Y) \}
  \end{equation*}
  is the unique maximal ideal of the algebra~$\mathcal{B}(Y)$ of bounded linear operators on~$Y$.  Moreover,
  we prove similar factorization results for the spaces $\ell^p(Y)$, $1\le p \leq \infty$, and use them to show
  that they are primary.
\end{abstract}

\maketitle

\tableofcontents

%auto-ignore

\section{Introduction}
\label{sec:introduction}

In~1970, J.~Lindenstrauss gave a talk on decompositions of Banach spaces (see~\cite{MR0405074} for
the abstract), which spurred a research program consisting of the following two main research
directions:
\begin{itemize}
  \item Constructing infinite-dimensional Banach spaces $E$ which are indecomposable, i.e., for any
        decomposition of $E$ into two complemented subspaces, one of them is finite-dimensional.
  \item Identifying the Banach spaces $E$ which are primary, i.e., whenever, $E = F\oplus G$, then
        either $F$ or $G$ is isomorphic to $E$.
\end{itemize}
The foundations for creating indecomposable Banach spaces were laid by the works of
B.~S.~Tsirelson~\cite{MR0350378} (see also~\cite{MR0355537,MR0981801}) and
Th.~Schlumprecht~\cite{MR1177333}.  But it was not until~1993 that the first indecomposable Banach
space $X_{\mathrm{GM}}$ was constructed in the seminal work~\cite{MR1201238} by W.~T.~Gowers and
B.~Maurey---in fact, they even showed that $X_{\mathrm{GM}}$ is hereditarily indecomposable.  This
fruitful line of research led to solutions of long-standing problems, including the
unconditional basic sequence problem and Banach's hyperplane problem~\cite{MR1201238} (see
also~\cite{MR1315601,MR1999196}) as well as the scalar-plus-compact problem solved by S.~A.~Argyros and
R.~G.~Haydon~\cite{MR2784662}.

The study of primary Banach spaces, on the other hand, goes back to A.~Pe{\l}czy\'{n}ski who proved
in his~1960 work~\cite{MR126145} that the classical sequence spaces $c_0$ and $\ell^p$,
$1\le p<\infty$, are prime, i.e., every infinite-dimensional complemented subspace is isomorphic to
the whole space (clearly, prime spaces are primary).

We will now give a terse historical overview of developments most relevant to the present work.  The
Lebesgue spaces $L^p$, $1\le p<\infty$, were shown to be primary in~1975 by P.~Enflo and
B.~Maurey~\cite{MR0397386} (the proof for~$L^p$, $1<p<\infty$, is due to P.~Enflo, and it was
extended to~$L^1$ by B.~Maurey, see~\cite{MR0500076}). Subsequently, D.~Alspach, P.~Enflo and
E.~Odell~\cite{MR0500076} gave an alternative proof for the primariness of $L^p$, $1 < p < \infty$,
and extended this result to separable rearrangement-invariant~(r.i.) function spaces on $[0,1]$ with
non-trivial Boyd indices (see~\cite{MR540367} or~\cite{MR0527010} for full proofs).  In particular,
this constitutes an alternative proof for $L^p$, $1 < p < \infty$.  Shortly thereafter, P.~Enflo and
T.~W.~Starbird~\cite{MR0557491} obtained the primariness of $L^1$ via $E$-operators. Later, the
dyadic Hardy space~$H^1$ and its dual $\mathrm{BMO}$ were proved to be primary by P.~F.~X.~Müller
(see~\cite{MR0931037} and~\cite{MR0955660}, respectively).  M.~Capon proved the primariness of the
two-parameter spaces $L^p(L^q)$, $1 < p,q < \infty$ and $L^p(E)$, where $E$ denotes a Banach space
with a symmetric Schauder basis (see~\cite{MR0687937} and~\cite{MR0688955}).  Capon's and
Alspach-Enflo-Odell's methods were then successfully adapted by P.~F.~X.~Müller in~\cite{MR1283008}
to show that the two-parameter dyadic Hardy space $H^1(\delta^2)$ is primary.  The primariness
of its dual $\mathrm{BMO}(\delta^2)$ was later obtained in~\cite{MR3436171} by utilizing Bourgain's
localization method.

The method of P.~Enflo and B.~Maurey~\cite{MR0397386} was to construct a pointwise multiplier
$M_g\colon L^p\to L^p$, $f\mapsto g\cdot f$, which approximates a given operator $T\colon L^p\to L^p$ on a large
subspace of $L^p$:
\begin{equation*}
  \cond^{\mathcal{B}}(\chi_B\cdot (T f))
  \approx M_g f,
  \qquad f\in L^p(B,\mathcal{B}),
\end{equation*}
where $\mathcal{B}$ is a carefully constructed $\sigma$-algebra and $|B|\geq 1/2$.  In a second step, they then
stabilize the function $g$ on another large subspace:
\begin{equation*}
  \cond^{\mathcal{C}}(\chi_C\cdot (T f))
  \approx c\cdot f,
  \qquad f\in L^p(C,\mathcal{C}),
\end{equation*}
where $\mathcal{C}\subset \mathcal{B}$ and $C\subset B$, $|C| > 0$.

A variant of the approach taken by D.~Alspach, P.~Enflo and E.~Odell~\cite{MR0500076} for $L^p$,
$1<p<\infty$, is to construct a block basis $(\tilde{h}_I)_I$ of the Haar system $(h_I)_I$ which
almost diagonalizes a given operator $T\colon L^p\to L^p$. More precisely, $(\tilde{h}_I)_I$ is
equivalent to $(h_I)_I$ and spans a subspace of $L^p$ which is complemented by a projection~$P$, and
we have
\begin{equation*}
  PT \tilde{h}_I \approx d_I \tilde{h}_I,
  \qquad  I\in \mathcal{D},
\end{equation*}
where $\mathcal{D}$ denotes the set of dyadic intervals and $(d_I)_I$ is a suitable family of
scalars.  Subsequently, either the \emph{Haar multiplier} $D$ given by $D h_I = d_I h_I$ or the
operator $I_{L^p}-D$ can, by virtue of the unconditionality of the Haar system in $L^p$,
$1 < p < \infty$, be inverted on a large subspace of $L^p$.  Finally, primariness follows from
Pe{\l}czy\'{n}ski's decomposition method \cite{MR126145} (see also \cite[II.B.24]{MR1144277}).

The approach developed by D.~Alspach, P.~Enflo and E.~Odell~\cite{MR0500076} turned out to be the
more flexible one, and it was further refined by
P.~F.~X.~Müller~\cite{MR0879418,MR0920082,MR0955660,MR0931037} (see also \cite{MR2157745,MR4439252})
to be suitable for the dyadic Hardy space $H^1$ and its dual $\mathrm{BMO}$.  However, since the
Haar system fails to be unconditional in $L^1$, the method described above cannot be applied to
$L^1$ in the same manner. Only recently, the first named author, P.~Motakis, P.~F.~X.~Müller and
Th.~Schlumprecht~\cite{MR4145794} successfully extended this technique to the space~$L^1$: First,
they introduced \emph{strategically reproducible bases} as part of a general framework which allows
one to reduce factorization problems for general bounded linear operators to the case of
\emph{diagonal operators} (i.e., in $L^1$, to the case of Haar multipliers).  Then they proved that
the identity on $L^1$ factors through every bounded Haar multiplier whose entries are bounded away
from~$0$, utilizing a result by E.~M.~Semenov and S.~N.~Uksusov~\cite{MR2975943} (see also
\cite{MR3397275}), which characterizes the bounded Haar multipliers on $L^1$ (for extensions to the
vector-valued case, see~\cite{MR2270086,MR4430957,MR3639095,MR4274301}).

In this paper, we will prove factorization results for \emph{Haar system Hardy spaces}, a large
class of Banach spaces in which the Haar system is a Schauder basis.  They are constructed as the
completion of the linear span of the Haar system either under a rearrangement-invariant norm
$\|\cdot\|$ or under the associated square function norm $\|\cdot\|_{*}$ given by
\begin{equation*}
  \Bigl\| \sum_{I} a_I h_I \Bigr\|_{*}
  = \Bigl\| \Bigl(\sum_{I} a_I^2 h_I^2\Bigr)^{1/2} \Bigr\|.
\end{equation*}
This class encompasses all separable r.i.~spaces on $[0,1)$ and their square function versions,
including the classical Lebesgue spaces $L^p$ as well as the dyadic Hardy spaces $H^p$, $1\le p<\infty$.
We develop a unified approach, which can deal with all these spaces simultaneously; in particular,
we obtain a unified proof of the classical results that $L^p$, $1\le p<\infty$ and~$H^1$ are primary.  We
then extend our factorization results to infinite direct sums of Haar system Hardy spaces.

We will now give a general outline of our approach, in which we systematically reduce a given
operator~$T$ on a Haar system Hardy space $Y$ to a constant multiple of the identity operator:
First, we reduce~$T$ to a bounded Haar multiplier via the framework of strategically reproducible
bases and strategically supporting systems.  If the Haar system $(h_I)_I$ is unconditional, then
every bounded sequence of scalars $(d_I)_I$ determines a bounded Haar multiplier by
$D h_I = d_I h_I$, $I\in \mathcal{D}$. On the contrary, the result of E.~M.~Semenov and
S.~N.~Uksusov~\cite{MR2975943} shows that there are far fewer bounded Haar multipliers on $L^1$.
Thus, a unified approach must necessarily be able to deal with any Haar multiplier without
exploiting that there are only few bounded Haar multipliers.  Indeed, our approach reduces any
bounded Haar multiplier to one with very small variation, i.e., to a stable Haar multiplier.  This
reduction is achieved by using randomized block bases of the Haar system. By a perturbation
argument, we then arrive at a scalar multiple $cI_Y$ of the identity operator, completing our
reduction procedure. One notable aspect of our method is that we have a priori knowledge about the
constant~$c$ as soon as the Haar multiplier~$D$ is determined.

Before proceeding, we will introduce some terminology. Let $E$ be a Banach space and let
$S,T\colon E\to E$ be bounded linear operators. We say that \emph{$S$ factors through $T$ (with
  constant $C\ge 0$)} if there exist bounded linear operators $A,B\colon E\to E$ with $S = ATB$ (and
$\|A\|\|B\|\le C$). The first important property is the \emph{primary factorization property}.
\begin{dfn}
  We say that a Banach space $E$ has the \emph{primary factorization property} if for every bounded
  linear operator $T\colon E\to E$, the identity $I_E$ either factors through $T$ or through
  $I_E - T$.
\end{dfn}
If a Banach space $E$ has the primary factorization property and satisfies Pe{\l}czy\'{n}ski's
\emph{accordion property}, i.e., $E\sim \ell^p(E)$ for some $1\le p\le \infty$, then $E$ is
primary. This follows from Pe{\l}czy\'{n}ski's decomposition method and the following observation:
If $P$ is a bounded linear projection on $E$, then the primary factorization property of $E$ implies
that either $P(E)$ or $(I_E - P)(E)$ contains a complemented subspace that is isomorphic to $E$.

We will now describe how the primary factorization property is related to the theory of operator
ideals.  Let $E$ be a Banach space and let $\mathcal{B}(E)$ denote the algebra of bounded linear
operators on $E$. Then we define
\begin{equation*}
  \mathcal{M}_E = \{ T\in \mathcal{B}(E) : I_E \ne ATB\text{ for all } A, B\in \mathcal{B}(E) \}.
\end{equation*}
The set $\mathcal{M}_E$ is an ideal of $\mathcal{B}(E)$ if and only if it is closed under
addition. It was observed by D.~Dosev and W.~B.~Johnson \cite{MR2586976} that if $\mathcal{M}_E$ is
an ideal of $\mathcal{B}(E)$, then it is automatically the \emph{unique maximal ideal} of
$\mathcal{B}(E)$. On the other hand, the assertion that $\mathcal{M}_E$ is closed under addition is
equivalent to the primary factorization property of $E$ (see \cite[Proposition~5.1]{MR2586976}).
In~\cite{MR4477014}, T.~Kania and the first named author introduced \emph{strategically supporting}
systems in dual pairs of Banach spaces to describe sufficient conditions for $\mathcal{M}_E$ to be
the unique maximal ideal of $\mathcal{B}(E)$. Typically, this concept is used to find a large
complemented subspace where a given operator has additional properties, e.g., large diagonal: A
bounded linear operator $T$ on a Banach space $E$ with a Schauder basis $(e_j)_{j=1}^{\infty}$ and
biorthogonal functionals $(e_j^{*})_{j=1}^{\infty}$ has \emph{large diagonal} if
$\inf_{j\in \mathbb{N}} |\langle e_j^{*}, Te_j \rangle| > 0$.  The study of operators with large
diagonal goes back to A.~D.~Andrew~\cite{MR0567081}, and they were later explicitly investigated
in~\cite{MR0890420,MR3861733,MR3819715,MR3794334,MR3990955,MR3910428,MR4460217}.  The first named
author, P.~Motakis, P.~F.~X.~Müller and Th.~Schlumprecht then developed these approaches into a
systematic framework in~\cite{MR4145794,MR4299595,MR4430957}\todo{add 2D paper}.  A part of this
framework is the \emph{factorization property}, which was analyzed using strategically reproducible
Schauder bases.
\begin{dfn}
  Let $E$ be a Banach space with a Schauder basis $(e_j)_{j=1}^{\infty}$ and biorthogonal
  functionals $(e_j^{*})_{j=1}^{\infty}$.  We say that $(e_j)_{j=1}^{\infty}$ has the
  \emph{factorization property} if the identity $I_E$ factors through every bounded linear operator
  $T\colon E\to E$ which has large diagonal with respect to~$(e_j)_{j=1}^{\infty}$.
\end{dfn}

For further development of this framework and some new applications to stopping time Banach spaces,
we refer to~\cite{MR4477014}.  More recently, Kh.~V.~Navoyan~\cite{MR4635015} showed that under mild
assumptions, the Haar basis of a \emph{Haar system space}~$X$ has the factorization property,
provided that it is unconditional in $X$.  A \emph{Haar system space} is the completion of the
linear span of the Haar system under a rearrangement-invariant norm; these spaces were introduced
in~\cite{MR4430957}.  Before presenting our main results in \Cref{sec:main-results}, we will
establish necessary notation and terminology.

%%% Local Variables:
%%% mode: latex
%%% TeX-master: "main"
%%% End:
 %auto-ignore

\section{Notation and basic definitions}
\label{sec:notation}

We denote by $\mathcal{D}$ the collection of all dyadic intervals in~$[0,1)$, namely
\begin{equation*}
  \mathcal{D}
  = \Big\{\Big[\frac{i-1}{2^j},\frac{i}{2^j}\Big) :
  j\in\mathbb{N}_0,\ 1\leq i\leq 2^j\Big\}.
\end{equation*}
In addition, for each $n\in\mathbb{N}_0$, we define
\begin{equation*}
  \mathcal{D}_n
  = \{I\in\mathcal{D} : |I| = 2^{-n}\}
  \qquad
  \text{and}
  \qquad
  \mathcal{D}_{\le n}
  = \bigcup_{k=0}^n \mathcal{D}_k,
  \quad
  \mathcal{D}_{<n} = \bigcup_{k=0}^{n-1}\mathcal{D}_k.
\end{equation*}
For each dyadic interval $I\in \mathcal{D}$, let~$I^+$ denote the left half of~$I$ and~$I^-$ its
right half, i.e., $I^+$ is the largest dyadic interval $J\in\mathcal{D}$ with $J\subsetneq I$ and
$\inf J = \inf I$, and we have $I^- = I\setminus I^+$.  If we use the symbol~$\pm$ multiple times in
an equation, we mean either always~$+$ or always~$-$.  Sometimes, instead of~$I^+$ or~$I^-$, we will
write~$I^{\varepsilon}$, where $\varepsilon \in \{ \pm 1 \}$. Conversely, for
$I\in \mathcal{D}\setminus \{ [0,1) \}$, we denote by~$\pi(I)$ the dyadic predecessor of~$I$, i.e.,
the unique dyadic interval~$J\in \mathcal{D}$ with $I = J^+$ or $I = J^-$.  Finally, for any
subcollection $\mathcal{B}\subset\mathcal{D}$, we put $\mathcal{B}^* = \bigcup_{I\in \mathcal{B}}I$.

Next, we define the bijective function $\iota\colon\mathcal{D}\to\mathbb{N}$ by
\begin{equation*}
  \Big[\frac{i}{2^j},\frac{i+1}{2^j}\Big)
  \overset{\iota}{\mapsto} 2^j + i.
\end{equation*}
The function $\iota $ defines a linear order on $\mathcal{D}$, and we will frequently consider
sequences indexed by dyadic intervals, identifying $\mathcal{D}$ with $\mathbb{N}$.  The Haar system
$(h_I)_{I\in\mathcal{D}}$ is defined by
\begin{equation*}
  h_I
  = \chi_{I^+} - \chi_{I^-},
  \qquad I\in\mathcal{D},
\end{equation*}
where $\chi_A$ denotes the characteristic function of a subset $A\subset [0,1)$.  We additionally
define $h_\varnothing = \chi_{[0,1)}$ and put $\mathcal{D}^+ = \mathcal{D}\cup\{\varnothing\}$.  We
also put $\iota(\varnothing) = 0$.  We will usually write $I\leq J$ if $\iota(I)\leq\iota(J)$,
$I,J\in\mathcal{D}^+$.  Henceforth, whenever we write $\sum_{I\in\mathcal{D}^+}$, we will always
mean that the sum is taken with this linear order $\iota$.  Recall that the Haar system
$(h_I)_{I\in \mathcal{D}^+}$, in the linear order defined by $\iota$, is a monotone Schauder basis
of $L^p$, $1\le p<\infty$ (and unconditional if $1<p<\infty$). For
$x = \sum_{I\in \mathcal{D}^+} a_Ih_I\in L^1$, we define the \emph{Haar support} of $x$ to be the
set of all $I\in \mathcal{D}^+$ with $a_I\ne 0$. More generally, if $(e_j)_{j=1}^{\infty}$ is a
Schauder basis of a Banach space $E$, then for $x = \sum_{j=1}^{\infty}a_je_j\in E$, we define
$\operatorname{supp}x = \{ j\in \mathbb{N} : a_j \ne 0 \}$.

We only consider real Banach spaces. If $(x_n)_{n=1}^{\infty}$ is a sequence in a Banach space~$E$,
then we denote its closed linear span in~$E$ by $[x_n]_{n\in \mathbb{N}}$. If~$E$ and~$F$ are Banach
spaces and $C\ge 1$, we say that~$E$ and~$F$ are $C$-isomorphic, i.e., $E \overset{C}{\sim} F$, if
there exists an isomorphism~$T$ from~$E$ onto~$F$ with $\|T\|\|T^{-1}\|\le C$. Moreover, if
$C\ge 0$, we say that a closed subspace~$F$ of~$E$ is $C$-complemented in~$E$ if there exists a
linear projection $P\colon E\to E$ with $P(E) = F$ and $\|P\|\le C$. Finally, we say that two
measurable functions $x, y\colon [0,1)\to \mathbb{R}$ are \emph{equimeasurable} if their absolute
values~$|x|$ and~$|y|$ have the same distribution.

\subsection{Haar system Hardy spaces}
\label{sec:haar-system-hardy-spaces}

The class of \emph{Haar system Hardy spaces} will be defined as an extension of the class of
\emph{Haar system spaces}, which was introduced in~\cite{MR4430957} as follows.
\begin{dfn}\label{dfn:HS-1d}
  A \emph{Haar system space $X$} is the completion of
  $H := \spn\{ h_I : I\in\mathcal{D}^+\} = \spn\{\chi_I : I\in\mathcal{D}\}$ under a norm
  $\|\cdot\|_X$ that satisfies the following properties:
  \begin{enumerate}[(i)]
    \item\label{dfn:HS-1d:1} If $x$, $y$ are in $H$ and $|x|$, $|y|$ have the same distribution,
          then $\|x\|_X = \|y\|_X$.
    \item\label{dfn:HS-1d:2} $\|\chi_{[0,1)}\|_X = 1$.
  \end{enumerate}
  We denote the class of Haar system spaces by $\mathcal{H}(\delta)$.  Moreover, given
  $X\in\mathcal{H}(\delta)$, we define the closed subspace $X_0$ of $X$ as the closure of
  $H_0 := \spn\{ h_I : I\in \mathcal{D} \}$ in $X$.  We denote the class of these subspaces by
  $\mathcal{H}_0(\delta)$.
\end{dfn}
Note that if a norm on $H$ satisfies Property~\eqref{dfn:HS-1d:1}, then we can always scale it so
that it satisfies Property~\eqref{dfn:HS-1d:2}. One can show that the Haar system
$(h_I)_{I\in \mathcal{D}^+}$, in the linear order defined by $\iota$, is a monotone Schauder basis
of any Haar system space $X$ (see \Cref{pro:HS-1d}).

Besides the spaces $L^p$, $1\le p<\infty$, and the closure of $H$ in $L^{\infty}$, the class
$\mathcal{H}(\delta)$ includes all rearrangement-invariant function spaces on $[0,1)$ (e.g.,
Orlicz function spaces) in which the span of the Haar system $(h_I)_{I\in \mathcal{D}^+}$ is
dense. According to \cite[Proposition~2.c.1]{MR540367}, this is true for all separable
rearrangement-invariant function spaces.

By $(r_n)_{n=0}^{\infty}$, we denote the sequence of standard Rademacher functions, i.e.,
\begin{equation*}
  r_n = \sum_{I\in \mathcal{D}_n} h_I,\qquad n\in \mathbb{N}_0.
\end{equation*}
We will now introduce the class of \emph{Haar system Hardy spaces}. To this end, we first define the set
$\mathcal{R}$ as
\begin{equation*}
  \mathcal{R} = \{ (r_{\iota(I)})_{I\in \mathcal{D}^+}, (r_0)_{I\in \mathcal{D}^+} \}.
\end{equation*}
Hence, if $\mathbf{r} = (r_I)_{I\in \mathcal{D}^+}\in \mathcal{R}$, then $\mathbf{r}$ is either an
independent sequence of $\pm 1$-valued random variables (indexed by dyadic intervals) or a constant
sequence. Starting with a Haar system space $X$ and a sequence
$\mathbf{r} = (r_I)_{I\in \mathcal{D}^+}\in \mathcal{R}$, we obtain a Haar system Hardy space
by taking the completion of $H$ under a new norm:

\begin{dfn}\label{dfn:haar-system-hardy-space}
  \todo{cite 2d-haar-system-hardy spaces paper!} Given $X\in\mathcal{H}(\delta)$ and
  $\mathbf{r} = (r_I)_{I\in\mathcal{D}^+}\in\mathcal{R}$, we define the \emph{(one-parameter) Haar
    system Hardy space $X(\mathbf{r})$} as the completion of
  $H = \spn\{ h_I : I\in \mathcal{D}^+ \}$ under the norm $\|\cdot\|_{X(\mathbf{r})}$ given by
  \begin{equation*}
    \Bigl\| \sum_{I\in\mathcal{D}^+} a_I h_I\Bigr\|_{X(\mathbf{r})}
    = \Bigl\|
    s\mapsto \int_0^1 \Bigl|
    \sum_{I\in\mathcal{D}^+} r_I(u) a_I h_I(s)
    \Bigr| \dif u
    \Bigr\|_X.
  \end{equation*}
  We will denote the class of one-parameter Haar system Hardy spaces by $\mathcal{HH}(\delta)$.
  Moreover, given $X(\mathbf{r})\in\mathcal{HH}(\delta)$, we define the closed subspace
  $X_0(\mathbf{r}) = [h_I]_{I\in \mathcal{D}}\subset X(\mathbf{r})$.  For notational convenience, we
  will also refer to the subspaces $X_0(\mathbf{r})$ as Haar system Hardy spaces.  We denote the
  class of these subspaces by $\mathcal{H}\mathcal{H}_0(\delta)$.
\end{dfn}

Clearly, if $\mathbf{r} = (r_I)_{I\in\mathcal{D}^+}$ is an independent sequence, then
$(h_I)_{I\in\mathcal{D}^+}$ is a $1$-unconditional Schauder basis of $X(\mathbf{r})$. If, on the
other hand, $r_I = r_0$ for all $I\in\mathcal{D}^+$, then we have
$\|\cdot \|_{X(\mathbf{r})} = \|\cdot \|_X$ and thus $X(\mathbf{r}) = X$.  We already know that in
this case, $(h_I)_{I\in\mathcal{D}^+}$ is a monotone Schauder basis of $X(\mathbf{r})$ (but it need
not be unconditional).  Finally, note that if $x$ is a finite linear combination of disjointly
supported Haar functions, then we always have $\|x\|_{X(\mathbf{r})} = \|x\|_X$.

\begin{rem}\label{rem:classical-examples}
  Let $X\in \mathcal{H}(\delta)$, suppose that $\mathbf{r}\in \mathcal{R}$ is independent, and let
  $(a_I)_{I\in \mathcal{D}}$ be a scalar sequence with $a_I\ne 0$ for at most finitely many
  $I$. Then, using Khintchine's inequality and the fact that $|x|\le |y|$ pointwise implies
  $\|x\|_X\le \|y\|_X$ for all $x,y\in H$ (see \Cref{pro:HS-1d}~\eqref{pro:HS-1d:v}), we obtain
  \begin{equation}\label{eq:khintchine-norm-square-function}
    \Bigl\| \sum_{I\in \mathcal{D}^+} a_Ih_I \Bigr\|_{X(\mathbf{r})}
    \sim \Bigl\| \Bigl( \sum_{I\in \mathcal{D}^+} |a_I|^2h_I(s)^2 \Bigr)^{1/2} \Bigr\|_X.
  \end{equation}
  Thus, for $X = L^1$ and $\mathbf{r}\in \mathcal{R}$ independent, we see that $X(\mathbf{r})$ is
  isomorphic to the dyadic Hardy space $H^1$, and for $X = L^p$, $1 < p < \infty$, we have
  $X(\mathbf{r})\sim H^p\sim L^p$, where $H^p$ denotes the dyadic Hardy space with parameter $p$. In
  fact, $\|\cdot\|_X\sim \|\cdot\|_{X(\mathbf{r})}$ holds whenever $X$ is a separable r.i.~function
  space with non-trivial Boyd indices (according to the remark
  following~\cite[Proposition~2.d.8]{MR540367}).  In all these cases, the identity operator provides
  an isomorphism. Finally, if $X$ is the closure of $H$ in $L^{\infty}$ and
  $\mathbf{r}\in \mathcal{R}$ is independent, then by~\eqref{eq:khintchine-norm-square-function},
  the space $X_0(\mathbf{r})$ is isomorphic to the closure of~$H_0$ in the non-separable space
  $SL^{\infty}$ (see~\cite{MR2060198,MR3819715}).
\end{rem}

\subsection{Additional definitions}
\label{sec:additional-definitions}

We begin this subsection by summarizing and extending the definitions related to factorization of
operators that were given in \Cref{sec:introduction}. First, we introduce the following additional
factorization modes.

\begin{dfn}\label{dfn:factorization-modes}
  Let $E$ denote a Banach space. Let $S, T\colon E\to E$ denote bounded linear operators, and let
  $C, \eta \ge 0$.
  \begin{enumerate}[(i)]
    \item\label{enu:dfn:factorization-modes:a} We say that \emph{$S$ factors through $T$ with
          constant $C$ and error $\eta$} if there exist linear operators $A,B\colon E\to E$ with
          $\|A\|\|B\|\leq C$ such that $\|S - ATB\|\leq \eta$.
    \item\label{enu:dfn:factorization-modes:a2} If~\eqref{enu:dfn:factorization-modes:a} holds and
          we additionally have $AB = I_E$, then we say that $S$ \emph{projectionally} factors
          through $T$ with constant $C$ and error $\eta$.
    \item\label{enu:dfn:factorization-modes:b} We say that $S$ (projectionally) factors through $T$
          \emph{with constant $C^+$} and error $\eta$ if for every $\gamma > 0$, the operator $S$
          (projectionally) factors through $T$ with constant $C + \gamma$ and error $\eta$.
  \end{enumerate}
  If we omit the phrase ``with error $\eta$''
  in~\eqref{enu:dfn:factorization-modes:a},~\eqref{enu:dfn:factorization-modes:a2}
  or~\eqref{enu:dfn:factorization-modes:b}, then we take that to mean that the error is~$0$.
\end{dfn}

Next, we make some elementary observations that will be useful later.

\begin{rem}\label{rem:factors}
  Let $R,S,T\colon E\to E$ denote bounded linear operators and suppose that $R$ (projectionally)
  factors through $S$ with constant $C_1$ and error $\eta_1$, and that $S$ (projectionally) factors
  through $T$ with constant $C_2$ and error $\eta_2$.  In \cite[Proposition~2.3]{MR4430957}, it was
  observed that $R$ (projectionally) factors through $T$ with constant $C_1 C_2$ and error
  $\eta_1 + C_1\eta_2$.
\end{rem}

\begin{rem}\label{rem:factors:iso}
  Let $S, T\colon E\to E$ denote bounded linear operators and suppose that $S$ is an isomorphism.
  If $S$ factors through $T$ with constant $C\geq 0$ and error $\eta\geq 0$ and
  $\eta \|S^{-1}\| < 1$, then $S$ factors through $T$ with constant $\frac{C}{1-\eta \|S^{-1}\|}$
  (and error $0$). Indeed, let $A,B\colon E\to E$ be bounded linear operators with
  $\|S - ATB\|\le \eta$ and $\|A\|\|B\|\le C$. Then we have
  \begin{equation*}
    \|I_E - S^{-1}ATB\|\le \eta\|S^{-1}\| < 1,
  \end{equation*}
  so $Q := (S^{-1}ATB)^{-1}$ exists and satisfies $\|Q\|\le 1/(1 - \eta\|S^{-1}\|)$. The statement
  follows since we have $S = ATBQ$.
\end{rem}

\begin{rem}\label{rem:factors:proj}
  Let $S,T\colon E\to E$ denote bounded linear operators and suppose that $S$ projectionally factors
  through $T$ with constant $C\geq 1$ and error $\eta > 0$.  Then $I_E - S$ projectionally factors
  through $I_E - T$ with constant $C\geq 1$ and error $\eta > 0$.
\end{rem}

The next definition includes quantitative and uniform versions of the (primary) factorization
property as well as some more variations of these concepts.

\begin{dfn}\label{dfn:factorization-properties}
  Let $E$ denote a Banach space with a Schauder basis $(e_j)_{j=1}^{\infty}$ and biorthogonal
  functionals $(e_j^{*})_{j=1}^{\infty}$, and let $C \ge 0$.
  \begin{enumerate}[(i)]
    \item Let $\delta > 0$, and let $T\colon E\to E$ be a bounded linear operator. We say that
          \begin{itemize}
            \item $T$ is \emph{diagonal (with respect to $(e_j)_{j=1}^{\infty}$)} if
                  $\langle e_k^{*}, Te_j \rangle = 0$ for all $k\ne j$. Diagonal operators with
                  respect to the Haar system are called \emph{Haar multipliers.}
            \item $T$ has \emph{$\delta$-large diagonal (with respect to $(e_j)_{j=1}^{\infty}$)} if
                  $|\langle e_j^{*}, Te_j \rangle| \ge \delta$ for all $j\in \mathbb{N}$.
            \item $T$ has $\delta$-large \emph{positive} diagonal if
                  $\langle e_j^{*}, Te_j \rangle \ge \delta$ for all $j\in \mathbb{N}$.
            \item $T$ has $\delta$-large \emph{negative} diagonal if
                  $\langle e_j^{*}, Te_j \rangle \le -\delta$ for all $j\in \mathbb{N}$.
          \end{itemize}
    \item\label{enu:dfn:factorization-modes:d} We say that $E$ has the \emph{$C$-primary (diagonal)
          factorization property (with respect to $(e_j)_{j=1}^{\infty}$)} if for any bounded linear
          operator $T\colon E\to E$ (which is diagonal with respect to $(e_j)_{j=1}^{\infty}$), the
          identity $I_E$ either factors through $T$ or through $I_E - T$ with constant~$C^+$.
    \item\label{enu:dfn:factorization-modes:e} Let $K\colon (0,\infty)\to (0,\infty)$.  We say that
          $(e_j)_{j=1}^{\infty}$ has the \emph{$K(\delta)$-(diagonal) factorization property} if for
          every bounded linear (diagonal) operator $T\colon E\to E$ with $\delta$-large diagonal
          with respect to $(e_j)_{j=1}^{\infty}$, the identity $I_E$ factors through $T$ with
          constant $K(\delta)^+$.
    \item If~\eqref{enu:dfn:factorization-modes:e} holds with ``$\delta$-large
          diagonal'' replaced by ``$\delta$-large \emph{positive} diagonal'', then we say that
          $(e_j)_{j=1}^{\infty}$ has the $K(\delta)$-\emph{positive} (diagonal) factorization
          property.
  \end{enumerate}
\end{dfn}

\begin{rem}\label{rem:enumeration-of-basis}
  Let $\lambda \ge 1$ and suppose that for each $k\in \mathbb{N}$, $E_k$ is a Banach space with a Schauder basis
  $(e_{k,j})_{j=1}^{\infty}$ whose basis constant is bounded by $\lambda$.  Let $1\le p\le \infty$. We identify each
  space $E_k$ with the subspace of $\ell^p((E_k)_{k=1}^{\infty})$ consisting of those sequences for which
  all coordinates, except the $k$th one, are equal to zero.

  Now let~$(\tilde{e}_m)_{m=1}^{\infty}$ be an enumeration of~$(e_{k,j})_{k,j=1}^{\infty}$ with the property
  that whenever we have $e_{k,i} = \tilde{e}_l$ and $e_{k,j} = \tilde{e}_m$ for some
  $k,i,j,l,m\in\mathbb{N}$, then the inequality $i < j$ implies $l < m$. Then for every
  $1\le p<\infty$, $(\tilde{e}_m)_{m=1}^{\infty}$ is a Schauder basis of
  $\ell^p((E_k)_{k=1}^{\infty})$ with basis constant at most $\lambda$. The associated biorthogonal functionals
  are given by $\tilde{e}_m^{*} = e_{k,j}^{*}$ (for $\tilde{e}_m = e_{k,j}$), where $e_{k,j}^{*}$
  acts on the $k$th component of $\ell^p((E_k)_{k=1}^{\infty})$.

  Clearly, for $p = \infty$, $(\tilde{e}_m)_{m=1}^{\infty}$ is not a Schauder basis of
  $\ell^p((E_k)_{k=1}^{\infty})$, but we can still define the notion of a large diagonal: Let
  $\delta > 0$. A bounded linear operator
  $T\colon \ell^{\infty}((E_k)_{k=1}^{\infty})\to \ell^{\infty}((E_k)_{k=1}^{\infty})$ has
  $\delta$-large diagonal with respect to $(\tilde{e}_m)_{m=1}^{\infty}$ if
  $|\langle \tilde{e}_m^*, T\tilde{e}_m \rangle| \ge \delta$ holds for all $m \in \mathbb{N}$. The
  $K(\delta)$-factorization property of $(\tilde{e}_m)_{m=1}^{\infty}$ in
  $\ell^{\infty}((E_k)_{k=1}^{\infty})$ is then defined like above.
\end{rem}

In order to prove the (primary) factorization property for $\ell^{\infty}$-sums of Haar system Hardy spaces,
we need to assume our spaces form a sequence that is \emph{uniformly asymptotically curved} with
respect to the array consisting of their Haar bases. This property was introduced
in~\cite{MR4299595}.

\begin{dfn}\label{dfn:asympt-curved}
  Let $E$ be a Banach space with a Schauder basis $(e_j)_{j=1}^{\infty}$.  Moreover, let
  $(E_k)_{k=1}^{\infty}$ be a sequence of Banach spaces, and for each $k\in \mathbb{N}$, let
  $(e_{k,j})_{j=1}^{\infty}$ denote a Schauder basis of $E_k$.  By an \emph{array}, we mean an
  indexed family $(x_{k,j})_{k,j=1}^{\infty}$ with $x_{k,j}\in E_k$ for all $k,j\in \mathbb{N}$.
  \begin{itemize}
    \item We say that $E$ is \emph{asymptotically curved with respect to $(e_j)_{j=1}^{\infty}$} if
          for every bounded block basis $(x_j)_{j=1}^{\infty}$ of $(e_j)_{j=1}^{\infty}$, we have
          \begin{equation*}
            \lim_{n\to \infty} \frac{1}{n} \Big\|\sum_{j=1}^n x_j\Big\|_E
            = 0.
          \end{equation*}
    \item We say that the sequence $(E_k)_{k=1}^{\infty}$ \emph{uniformly asymptotically curved with
          respect to the array $(e_{k,j})_{k,j=1}^{\infty}$} if the following holds: For every bounded array
          $(x_{k,j})_{k,j=1}^{\infty}$ with the property that $(x_{k,j})_{j=1}^{\infty}$ is a block basis of
          $(e_{k,j})_{j=1}^{\infty}$ for all $k\in \mathbb{N}$, we have
          \begin{equation*}
            \lim_{n\to \infty} \sup_{k\in \mathbb{N}}\frac{1}{n}\Big\|\sum_{j=1}^n x_{k,j} \Big\|_{E_k}
            = 0.
          \end{equation*}
  \end{itemize}
\end{dfn}
We refer to Section \Cref{sec:curved} for a discussion of these concepts.

%%% Local Variables:
%%% mode: latex
%%% TeX-master: "main"
%%% End:
 %auto-ignore

\section{Main results}
\label{sec:main-results}

In order to avoid having to deal with the constant function $h_{\varnothing}$ separately in the
proofs, we will state our results for the subspaces $X_0(\mathbf{r})$ instead of $X(\mathbf{r})$. We
will explain in \Cref{rem:constant-function} how to obtain the corresponding versions of
\Cref{thm:main-result:A} and \Cref{thm:main-result:B} for the spaces~$X(\mathbf{r})$.  In the
following, we always assume that $Y, Y_k\in \mathcal{HH}_0(\delta)$, $k\in \mathbb{N}$. The Haar
basis of~$Y$ is denoted by~$(h_I)_{I\in \mathcal{D}}$, and for every $k\in \mathbb{N}$, the Haar
basis of~$Y_k$ is denoted by $(h_{k,I})_{I\in \mathcal{D}}$.  Given $1\le p\le \infty$, we identify
each space $Y_k$ with the subspace of $\ell^p((Y_k)_{k=1}^{\infty})$ consisting of sequences
supported in the $k$th component, and we enumerate $(h_{k,I})_{k\in \mathbb{N}, I\in \mathcal{D}}$
according to \Cref{rem:enumeration-of-basis}, thus obtaining a monotone Schauder basis of
$\ell^p((Y_k)_{k=1}^{\infty})$ for $1\le p<\infty$.

\begin{thm}\label{thm:main-result:A}
  Suppose that the sequence of Rademacher functions $(r_n)_{n=0}^{\infty}$ is weakly null in~$Y$,
  and let~$E$ denote one of the following Banach spaces:
  \begin{enumerate}[(i)]
    \item\label{thm:main-result:A:i} $E = Y$
    \item\label{thm:main-result:A:ii} $E = \ell^p(Y)$ for some $1\le p<\infty$
    \item\label{thm:main-result:A:iii} $E = \ell^{\infty}(Y)$ if $Y$ is asymptotically curved with
          respect to $(h_I)_{I\in \mathcal{D}}$.
  \end{enumerate}
  Then~$E$ has the $4$-primary factorization property, and hence, $\mathcal{M}_E$ is the unique
  maximal ideal of $\mathcal{B}(E)$. In particular, the spaces in~\eqref{thm:main-result:A:ii}
  and~\eqref{thm:main-result:A:iii} are primary.
\end{thm}

We will prove \Cref{thm:main-result:A} in \Cref{sec:proofs}. Next, we state the main
results which involve the factorization property.

\begin{thm}\label{thm:main-result:B}
  Suppose that the sequence of Rademacher functions $(r_n)_{n=0}^{\infty}$ is weakly null in~$Y$ and
  in each~$Y_k$, $k\in \mathbb{N}$, and let $E$ and $(e_m)_{m=1}^{\infty}$ denote one of the following pairs of Banach spaces and sequences:
  \begin{enumerate}[(i)]
    \item\label{thm:main-result:B:i} $E = Y$ and $(e_m)_{m=1}^{\infty} = (h_I)_{I\in \mathcal{D}}$
    \item\label{thm:main-result:B:ii} $E = \ell^p((Y_k)_{k=1}^{\infty})$ for some $1\le p<\infty$
          and $(e_m)_{m=1}^{\infty} = (h_{k,I})_{k\in \mathbb{N},I\in \mathcal{D}}$
    \item\label{thm:main-result:B:iii} $E = \ell^{\infty}((Y_k)_{k=1}^{\infty})$ and
          $(e_m)_{m=1}^{\infty} = (h_{k,I})_{k\in \mathbb{N},I\in \mathcal{D}}$ if
          $(Y_k)_{k=1}^{\infty}$ is uniformly asymptotically curved with respect to
          $(h_{k,I})_{k\in \mathbb{N},I\in \mathcal{D}}$.
  \end{enumerate}
  Then~$(e_m)_{m=1}^{\infty}$ has the $4/\delta$-factorization property in~$E$.
\end{thm}
The proof of \Cref{thm:main-result:B} can be found in \Cref{sec:proofs}.  We will now discuss the
hypotheses of \Cref{thm:main-result:A} and \Cref{thm:main-result:B}. In \Cref{sec:curved}, we will
provide sufficient conditions for (uniform) asymptotic curvedness in Haar system Hardy spaces (see
\Cref{pro:hshs-curved-examples}, \Cref{rem:hshs-curved-general} and
\Cref{rem:hshs-uniformly-curved-general}).

The requirement that the (standard) sequence of Rademacher functions $(r_n)_{n=0}^{\infty}$ is
weakly null in a Haar system Hardy space $X(\mathbf{r})\in \mathcal{HH}(\delta)$ or, equivalently,
in $Y = X_0(\mathbf{r})$ is not a strong limitation: If $\mathbf{r}$ is independent, then this
condition is always satisfied because using~\eqref{eq:khintchine-norm-square-function}, we see that
the sequence $(r_n)_{n=0}^{\infty}$ in $X(\mathbf{r})$ is equivalent to the unit vector basis of
$\ell^2$.  On the other hand, if $\mathbf{r}$ is constant, recall that $X(\mathbf{r}) = X$.  Then
the proof of \cite[Proposition~2.c.10]{MR540367}, which is a result by V.~A.~Rodin and E.~M.~Semenov
\cite{MR388068}, shows that the following conditions are equivalent:
\begin{enumerate}[(i)]
  \item\label{rem:rademacher-weakly-null:i} The sequence of standard Rademacher functions
        $(r_n)_{n=0}^{\infty}$ is weakly null in $X$.
  \item\label{rem:rademacher-weakly-null:ii} The sequence $(r_n)_{n=0}^{\infty}$ in $X$ is not
        equivalent to the unit vector basis of $\ell^1$.
  \item\label{rem:rademacher-weakly-null:iii} The norm $\|\cdot \|_X$ is not equivalent to the
        $L^{\infty}$-norm on $H$ (i.e., $X\ne C(\Delta)$).
  \item\label{rem:rademacher-weakly-null:iv} We have
        $\lim_{n\to \infty}\|\chi_{[0,2^{-n})}\|_X = 0$.
\end{enumerate}
The implication
\eqref{rem:rademacher-weakly-null:ii}$\Longrightarrow$\eqref{rem:rademacher-weakly-null:i} also
follows from Rosenthal's $\ell^1$~Theorem (see \cite[Remark~2.15]{MR4430957}).

\begin{rem}\label{rem:constant-function}
  Later on, we will show that for every $X(\mathbf{r})\in \mathcal{HH}(\delta)$, the closed subspace
  $Y = X_0(\mathbf{r})$ is isomorphic to $X(\mathbf{r})$ and that there exists an isomorphism
  $S\colon X(\mathbf{r})\to X_0(\mathbf{r})$ with $\|S\|\le 9$ and $\|S^{-1}\|\le 18$ which maps
  $(h_I)_{I\in \mathcal{D}^+}$ bijectively onto a permutation of $(h_I)_{I\in \mathcal{D}}$ (see
  \Cref{pro:hsp:hyperplane} and~\Cref{rem:hsp:hyperplane}).  Since the isomorphism $S$ is a
  rearrangement of the Haar system, it preserves large diagonals of bounded linear operators, i.e.,
  if $\delta > 0$ and $T\colon X(\mathbf{r})\to X(\mathbf{r})$ has $\delta$-large diagonal with
  respect to $(h_I)_{I\in \mathcal{D}^+}$, then $STS^{-1}\in \mathcal{B}(Y)$ has $\delta$-large
  diagonal with respect to $(h_I)_{I\in \mathcal{D}}$.  We thus see that the results of
  \Cref{thm:main-result:A} and \Cref{thm:main-result:B} concerning the primary factorization
  property and the factorization property respectively, carry over from~$Y$ and
  $(h_I)_{I\in \mathcal{D}}$ to~$X(\mathbf{r})$ and $(h_I)_{I\in \mathcal{D}^+}$, albeit at the
  price of increasing the factorization constant.  The same is true for the results in
  \Cref{thm:main-result:B}~\eqref{thm:main-result:B:ii} and~\eqref{thm:main-result:B:iii}.
\end{rem}

We will now introduce the \emph{characteristic set} of a bounded Haar multiplier, which contains a
priori information on the factorization appearing in \Cref{thm:main-result:3}.

\begin{dfn}\label{dfn:characteristic}
  Given a bounded Haar multiplier $D\colon Y\to Y$, we define the \emph{characteristic set
    $\Lambda(D)$} by
  \begin{equation*}
    \Lambda(D)
    := \{ c\in \mathbb{R} : \text{$c$ is a cluster point of $(\langle r_n, D r_n \rangle)_{n=0}^{\infty}$}\}.
  \end{equation*}
\end{dfn}

\begin{rem}\label{rem:characteristic}
  Since for each $n\in \mathbb{N}$, the Rademacher function~$r_n$ has norm~$1$ both in~$Y$ and
  in~$Y^{*}$ (see \Cref{cor:norm-in-dual-space}), the boundedness of~$D$ implies that~$\Lambda(D)$
  is non-empty and bounded: We have $|c|\le \|D\|$ for all $c\in \Lambda(D)$.  Moreover, note that
  $\Lambda(I_Y - D) = \{1 - c : c\in \Lambda(D)\}$.  Finally, if $D$ has $\delta$-large positive
  diagonal for some $\delta > 0$, i.e., $\langle h_I, Dh_I \rangle \ge \delta |I|$ for all
  $I\in \mathcal{D}$, then we have $\inf \Lambda(D)\geq \delta$.
\end{rem}

\begin{thm}\label{thm:main-result:3}
  Let $D\colon Y\to Y$ be a bounded Haar multiplier, and let $c\in \Lambda(D)$.  Then the following
  assertions are true:
  \begin{enumerate}[(i)]
    \item\label{thm:main-result:3:i} For every $\eta > 0$, $cI_Y$ projectionally factors through $D$
          with constant~$1$ and error~$\eta$.
    \item\label{thm:main-result:3:ii} If $c\neq 0$, then the identity $I_Y$ factors through $D$ with
          constant $(1/|c|)^+$.
    \item\label{thm:main-result:3:iii} The identity $I_Y$ either factors through $D$ or through
          $I_Y - D$ with constant $2^+$; more precisely, with constant $\min(1/|c|,1/|1-c|)^+$,
          where we define $1/0 = \infty$.
  \end{enumerate}
\end{thm}

For the proof of \Cref{thm:main-result:3}, see \Cref{sec:stabilization}. Note that
\Cref{thm:main-result:3}~\eqref{thm:main-result:3:i} does not simply state that \emph{some} multiple
of the identity $cI_Y$ projectionally factors through $D$ (with some constant and error), but it
also provides a priori knowledge about the constant $c$: We may choose $c$ to be any element of the
characteristic set $\Lambda(D)$, and we also have some knowledge about what this set can look like
(see \Cref{rem:characteristic}).

%%% Local Variables:
%%% mode: latex
%%% TeX-master: "main"
%%% End:
 %auto-ignore

\section{Properties of Haar system Hardy spaces}
\label{sec:properties-of-hshs}

Before discussing the properties of Haar system Hardy spaces, we recall the following basic
results on Haar system spaces.

\begin{pro}\label{pro:HS-1d}
  Let $X\in\mathcal{H}(\delta)$.  Then following assertions are true.
  \begin{enumerate}[(i)]
    \item\label{pro:HS-1d:i} For every $f\in H = \spn\{\chi_I : I\in\mathcal{D}\}$, we have
          $\|f\|_{L^1}\leq \|f\|_X\leq \|f\|_{L^\infty}$.  Therefore, $X$ can be naturally
          identified with a space of measurable scalar valued functions on $[0,1)$ and
          ${\overline H^{\|\cdot\|_{L^\infty}}} \subset X \subset{L^1}$.
    \item\label{pro:HS-1d:ii} The Haar system $(h_I)_{I\in \mathcal{D}^+}$, in the usual linear
          order, is a monotone Schauder basis of~$X$.
    \item \label{pro:HS-1d:iii} $H$ naturally coincides with a subspace of $X^{*}$, and its closure
          $\overline{H}$ in $X^{*}$ is also a Haar system space.
    \item\label{pro:HS-1d:v} For all $f,g\in H$ with $|f|\leq |g|$, we have $\|f\|_X \leq \|g\|_X$.
  \end{enumerate}
\end{pro}
\begin{proof}
  We refer to~\cite[Proposition~2.13]{MR4430957} for a proof
  of~\eqref{pro:HS-1d:i}--\eqref{pro:HS-1d:iii}; assertion~\eqref{pro:HS-1d:v} follows from the
  observation that for each $n\in \mathbb{N}$, the family $(\chi_I : I\in \mathcal{D}_n)$ is
  $1$-unconditional in $X$.
\end{proof}

\begin{rem}
  In \Cref{pro:HS-1d}~\eqref{pro:HS-1d:iii}, we identify each $g\in H$ with the bounded linear
  functional $x_g^{*}\in X^{*}$ defined as the continuous extension of $f\mapsto \int_0^1fg$,
  $f\in H$.
\end{rem}

Next, we show that if $X$ is a Haar system space, then conditional expectations with respect to
certain finite $\sigma$-algebras are contractions on $X$.

\begin{lem}\label{lem:conditional-expectation}
  Let $X\in\mathcal{H}(\delta)$, and let $\mathcal{F}$ denote a $\sigma$-algebra generated by a partition
  $(A_i : 1\le i\le m)$ of $[0,1)$, where each set $A_i$ is a finite union of dyadic intervals.  Then
  \begin{equation*}
    \|\mathbb{E}^{\mathcal{F}} x\|_X
    \leq \|x\|_X,
    \qquad x\in H.
  \end{equation*}
\end{lem}

\begin{proof}
  Let $M\in\mathbb{N}$ and suppose that $x = \sum_{I\in\mathcal{D}_M} a_I \chi_I$.  Pick $N > M$ and
  sets of pairwise disjoint dyadic intervals
  $\mathcal{A}_i = \{ K_{i,k} : 1\leq k\leq n_i\}\subset\mathcal{D}_N$ such that
  $\bigcup_{k=1}^{n_i} K_{i,k} = A_i$, $1\leq i\leq m$.  Put
  \begin{equation*}
    R_N
    = \{
    \rho\colon\mathcal{D}_N\to\mathcal{D}_N :
    \text{$\rho$ is bijective and $\rho(\mathcal{A}_i) = \mathcal{A}_i$ for all $1\leq i\leq m$}
    \}
  \end{equation*}
  and let $\ave_{\rho\in R_N}$ denote the average over all $\rho$ in $R_N$.  Observe that since the
  intervals $K_{i,k}$, $1\leq i\leq m$, $1\leq k\leq n_i$ are pairwise disjoint, using
  \Cref{dfn:HS-1d}~\eqref{dfn:HS-1d:1}, we obtain
  \begin{align*}
    \|x\|_X
    &= \Bigl\|
      \sum_{i=1}^m\sum_{k=1}^{n_i}
      \sum_{\substack{I\in\mathcal{D}_M\\I\supset K_{i,k}}} a_I
    \chi_{K_{i,k}}
    \Bigr\|_X
    = \ave_{\rho\in R_N} \Bigl\|
    \sum_{i=1}^m\sum_{k=1}^{n_i}
    \sum_{\substack{I\in\mathcal{D}_M\\I\supset K_{i,k}}} a_I
    \chi_{\rho(K_{i,k})}
    \Bigr\|_X\\
    &\geq \Bigl\|
      \sum_{i=1}^m\sum_{k=1}^{n_i}
      \sum_{\substack{I\in\mathcal{D}_M\\I\supset K_{i,k}}} a_I
    \ave_{\rho\in R_N} \chi_{\rho(K_{i,k})}
    \Bigr\|_X.
  \end{align*}
  Note that for fixed $1\leq i\leq m$, $1\leq k\leq n_i$ we have
  \begin{equation*}
    \ave_{\rho\in R_N} \chi_{\rho(K_{i,k})}
    = \frac{1}{|R_N|} \sum_{\rho\in R_N} \chi_{\rho(K_{i,k})}
    = \sum_{K\in\mathcal{A}_i}
    \frac{1}{|R_N|}\sum_{\substack{\rho\in R_N\\\rho(K_{i,k}) = K}} \chi_K
    = \sum_{K\in\mathcal{A}_i} \frac{1}{|\mathcal{A}_i|} \chi_K
    = \frac{|K_{i,k}|}{|A_i|} \chi_{A_i}.
  \end{equation*}
  Inserting this identity in the above inequality yields
  \begin{align*}
    \|x\|_X
    &\geq \Bigl\|
      \sum_{i=1}^m \sum_{I\in\mathcal{D}_M} a_I
      \sum_{\substack{1\leq k\leq n_i\\K_{i,k}\subset I}} \frac{|K_{i,k}|}{|A_i|} \chi_{A_i}
    \Bigr\|_X
    = \Bigl\|
    \sum_{i=1}^m \sum_{I\in\mathcal{D}_M} a_I
    \frac{|A_i\cap I|}{|A_i|} \chi_{A_i}
    \Bigr\|_X\\
    &= \Bigl\|
      \sum_{I\in\mathcal{D}_M} a_I \mathbb{E}^{\mathcal{F}} \chi_I
      \Bigr\|_X
      = \| \mathbb{E}^{\mathcal{F}} x\|_X,
  \end{align*}
  as claimed.
\end{proof}

If $X(\mathbf{r})\in \mathcal{HH}(\delta)$ is a Haar system Hardy space, then just like in the case
of Haar system spaces, we can identify $H = \operatorname{span}\{ h_I : I\in \mathcal{D}^+ \}$ with
a subspace of the dual space of $X(\mathbf{r})$.  In the next lemma, we compute the norm of a Haar
function~$h_I$, viewed as an element of~$X(\mathbf{r})^{*}$.

\begin{lem}\label{lem:product-of-norms}
  Let $X(\mathbf{r})\in \mathcal{HH}(\delta)$.  Then for every $I\in \mathcal{D}$, we have
  $\|h_I\|_{X(\mathbf{r})}\|h_I\|_{X(\mathbf{r})^{*}} = |I|$ and
  $\|h_I\|_{X_0(\mathbf{r})}\|h_I\|_{X_0(\mathbf{r})^{*}} = |I|$.
\end{lem}
\begin{proof}
  We only prove the second equality since the proof of the first one is analogous.  Fix
  $I\in \mathcal{D}$ and let $x = \sum_{J\in \mathcal{D}} a_Jh_J\in H_0$.  Now observe that
  \begin{equation*}
    \|a_I h_I\|_{X(\mathbf{r})}
    \le \frac{1}{2} \Bigl\| a_Ih_I + \sum_{\substack{J\in \mathcal{D}\\ J < I}} a_Jh_J \Bigr\|_{X(\mathbf{r})}
    + \frac{1}{2}\Bigl\| a_Ih_I - \sum_{\substack{J\in \mathcal{D}\\ J<I}} a_Jh_J \Bigr\|_{X(\mathbf{r})}.
  \end{equation*}
  The two summands on the right-hand side are equal because the two functions inside the norms are
  equimeasurable. Thus, we have
  \begin{equation*}
    \|a_Ih_I\|_{X(\mathbf{r})}
    \le \Bigl\| \sum_{\substack{J\in \mathcal{D}\\ J\le I}} a_Jh_J \Bigr\|_{X(\mathbf{r})}
    \le \|x\|_{X(\mathbf{r})}.
  \end{equation*}
  Consequently, we obtain that $\|h_I\|_{X_0(\mathbf{r})^{*}}\le |I|/\|h_I\|_{X(\mathbf{r})}$.  For
  the other inequality, note that
  \begin{equation*}
    \|h_I\|_{X_0(\mathbf{r})^{*}}
    \ge \Bigl\langle h_I, \frac{h_I}{\|h_I\|_{X(\mathbf{r})}} \Bigr\rangle
    = \frac{|I|}{\|h_I\|_{X(\mathbf{r})}}.\qedhere
  \end{equation*}
\end{proof}

The following lemma provides an upper bound on the norm of a Haar multiplier, in the spirit of the
theorem on Haar multipliers on $L^1$ by E.~M.~Semenov and S.~N.~Uksusov \cite{MR2975943} (see also
\cite{MR3397275} and \cite{MR4430957}).

\begin{lem}\label{lem:haar-multiplier}
  Let $X_0(\mathbf{r})\in \mathcal{HH}_0(\delta)$ and suppose that $(d_I)_{I\in \mathcal{D}}$ is a
  scalar sequence that satisfies
  \begin{equation*}
    \vvvert (d_I)_{I\in \mathcal{D}}\vvvert
    := |d_{[0,1)}| + 2 \sum_{I\in \mathcal{D}\setminus\{[0,1)\}} | d_I - d_{\pi(I)}|
    < \infty.
  \end{equation*}
  Then the Haar multiplier $D\colon X_0(\mathbf{r})\to X_0(\mathbf{r})$, defined as the linear
  extension of $D h_I = d_I h_I$, $I\in \mathcal{D}$, is bounded:
  \begin{equation}\label{eq:4}
    \|D\|
    \leq \vvvert (d_I)_{I\in \mathcal{D}}\vvvert.
  \end{equation}
  Moreover, if the Rademacher sequence $\mathbf{r}$ is independent, then we have
  \begin{equation}\label{eq:5}
    \|D\|
    \leq \sup_{I\in \mathcal{D}} |d_I|.
  \end{equation}
\end{lem}
\begin{rem}
  In contrast to~\cite{MR2975943}, our upper bound~\eqref{eq:4} does not only involve the largest
  variation of $(d_I)_{I\in \mathcal{D}}$ along branches of the dyadic tree, but the sum of
  \emph{all} differences between entries $d_I$ and their predecessors $d_{\pi(I)}$. This larger upper bound
  is sufficient for our purposes.
\end{rem}

\begin{proof}[Proof of \Cref{lem:haar-multiplier}]
  Firstly, under the condition that the sequence $\mathbf{r}$ is independent, estimate~\eqref{eq:5}
  follows immediately from the $1$-unconditionality of the Haar system in~$X_0(\mathbf{r})$.
  Secondly, considering that
  \begin{equation*}
    |d_I|
    \leq |d_{[0,1)}| + |d_I - d_{[0,1)}|
    \leq |d_{[0,1)}| + \sum_{J : I\subset J\subsetneq [0,1)} |d_J - d_{\pi(J)}|,
    \qquad I\in \mathcal{D},
  \end{equation*}
  we only have to show~\eqref{eq:4} in the case where $\mathbf{r}$ is a constant sequence, i.e.,
  $X_0(\mathbf{r}) = X_0$.  To this end, let $x = \sum_{I\in \mathcal{D}} a_I h_I\in H_0$ and
  observe that
  \begin{align*}
    \|Dx - d_{[0,1)}x\|_X
    &= \Bigl\| \sum_{I\in \mathcal{D}} (d_I - d_{[0,1)}) a_I h_I \Bigr\|_X
      = \Bigl\| \sum_{I\in \mathcal{D}} \sum_{J : I\subset J\subsetneq [0,1)} (d_J - d_{\pi(J)}) a_I h_I \Bigr\|_X\\
    &\leq \sum_{J\subsetneq [0,1)} |d_J - d_{\pi(J)}|\cdot \Bigl\| \sum_{I : I\subset J}  a_I h_I \Bigr\|_X.
  \end{align*}
  Now, for $n\in \mathbb{N}_0$, let $\mathbb{E}_n$ denote the conditional expectation with respect
  to the $\sigma$-algebra generated by $\mathcal{D}_n$.  Then, using
  \Cref{lem:conditional-expectation} and \Cref{pro:HS-1d}~\eqref{pro:HS-1d:v}, we obtain
  \begin{align*}
    \|Dx - d_{[0,1)}x\|_X
    &= \sum_{n=1}^{\infty} \sum_{J\in \mathcal{D}_n}
      |d_J - d_{\pi(J)}|\cdot \| \chi_J\cdot (I_X - \mathbb{E}_n) x\|_X\\
    &\leq 2 \sum_{J\subsetneq [0,1)} |d_J - d_{\pi(J)}|\cdot \|x\|_X,
  \end{align*}
  which shows~\eqref{eq:4}.
\end{proof}

\begin{lem}\label{lem:haar-multiplier-cond}
  Let $X_0(\mathbf{r})\in \mathcal{HH}_0(\delta)$, and let
  $(d_I)_{I\in \mathcal{D}}\in\{0,1\}^{\mathcal{D}}$ be a sequence with the property that for all
  $I\in \mathcal{D}$ with $d_I = 0$, we have $d_{I^+} = d_{I^-} = 0$.  Then the operator
  $D\colon X_0(\mathbf{r})\to X_0(\mathbf{r})$, defined as the continuous linear extension of
  $D h_I = d_I h_I$, $I\in \mathcal{D}$, is a contraction, i.e.,
  \begin{equation*}
    \|D x\|_{X(\mathbf{r})}
    \leq \|x\|_{X(\mathbf{r})},
    \qquad x\in X_0(\mathbf{r}).
  \end{equation*}
\end{lem}
\begin{proof}
  If the sequence $\mathbf{r}$ is independent, the result follows immediately from the
  $1$-uncondi\-tionality of the Haar system in $X_0(\mathbf{r})$.  So we only have to prove the lemma
  in the case where $\mathbf{r}$ is a constant sequence, i.e., $X_0(\mathbf{r}) = X_0$. Moreover, we
  may assume that $d_{[0,1)} = 1$.

  First, given $I\in \mathcal{D}$, we define
  $M(I) = \bigcap \{ J\in \mathcal{D} : J\supset I,\ d_J = 1 \}\in \mathcal{D}$.  Let $x\in H_0$ be
  given by $x = \sum_{I\in \mathcal{D}_{\le n}} a_I h_I$ for some natural number $n$.  Put
  $\mathcal{M} = \{ M(I) : I\in \mathcal{D}_n \}$ and observe that $\mathcal{M}$ is a partition of
  $[0,1)$ and that for every $J\in \mathcal{M}$, the function $Dx$ is constant on $J^+$ and on $J^-$
  since $d_I = 0$ for all $I\in \mathcal{D}_{\le n}$ with $I\subsetneq J$.  Now let $\mathcal{F}$ be
  the $\sigma$-algebra on~$[0,1)$ generated by the intervals $J^{\pm }$, $J\in \mathcal{M}$, and
  note that since
  \begin{equation*}
    x - Dx
    = \sum_{J\in \mathcal{M}} \sum_{\substack{I\in \mathcal{D}_{\le n}\\ I\subsetneq J}} a_Ih_I,
  \end{equation*}
  it follows that $Dx = \mathbb{E}^{\mathcal{F}}x$.  Thus, by \Cref{lem:conditional-expectation}, we
  have $\|Dx\|_X = \|\mathbb{E}^{\mathcal{F}}x\|_X\le \|x\|_X$.
\end{proof}
\begin{rem}
  Note that in the preceding lemma, $Dx$ may be interpreted as a stopped martingale with respect to
  the dyadic filtration, where the stopping time at $s\in [0,1)$ is determined by the index $I$ at
  which the sequence $(d_I : I\in \mathcal{D},\ s\in I)$ switches from~$1$ to~$0$.
\end{rem}

Next, we prove that for every Haar system Hardy space $X(\mathbf{r})\in \mathcal{HH}(\delta)$,
the closed subspace $X_0(\mathbf{r})$ is isomorphic to $X(\mathbf{r})$.

\begin{lem}\label{lem:shift-operator}
  Let $X_0(\mathbf{r})\in \mathcal{HH}_0(\delta)$ and put
  \begin{align*}
    E
    &= [h_I : I\in \mathcal{D},\ \inf I = 0],\\
    E_0
    &= [h_I : I\in \mathcal{D},\ \inf I = 0,\ \sup I \le \tfrac{1}{2}].
  \end{align*}
  Then the operator
  $W\colon E\to E_0$ defined by
  \begin{equation*}
    \sum_{\substack{I\in \mathcal{D}\\ \inf I = 0}} a_Ih_I \mapsto \sum_{\substack{I\in \mathcal{D}\\ \inf I = 0}} a_Ih_{I^+}
  \end{equation*}
  satisfies the estimates
  \begin{equation*}
    \frac{1}{2}\|x\|_{X(\mathbf{r})}\le \|W x\|_{X(\mathbf{r})}\le \|x\|_{X(\mathbf{r})},
    \qquad x\in X_0(\mathbf{r}).
  \end{equation*}
\end{lem}
\begin{proof}
  Let $x = \sum_{I\in \mathcal{D}} a_Ih_I\in E$ be a finite linear combination of the Haar system
  and observe that the two functions
  \begin{equation}\label{eq:177}
    \begin{aligned}
      s&\mapsto \int_0^1 \Bigl|\sum_{\substack{I \in \mathcal{D}\\ \inf I = 0}} r_I(u) a_I h_I(s)\Bigr| \dif u,\\
      s&\mapsto \int_0^1 \Bigl|\sum_{\substack{I \in \mathcal{D}\\ \inf I = 0}} r_I(u) a_I h_{I^+}(s)
         + \sum_{\substack{I \in \mathcal{D}\\ \inf I = 0}} r_I(u) a_I h_{I^+}(s - 1/2)\Bigr| \dif u
    \end{aligned}
  \end{equation}
  are equimeasurable.  Moreover, we note that the two terms of the second function in~\eqref{eq:177}
  are disjointly supported.  Hence, by \Cref{dfn:HS-1d} and \Cref{pro:HS-1d}, we obtain
  \begin{equation*}
    \|x\|_{X(\mathbf{r})}
    = \Bigl\| s\mapsto \int_0^1 \Bigl|\sum_{\substack{I \in \mathcal{D}\\ \inf I = 0}} r_I(u) a_I h_{I^+}(s)
    + \sum_{\substack{I \in \mathcal{D}\\ \inf I = 0}} r_I(u) a_I h_{I^+}(s - 1/2)\Bigr| \dif u \Bigr\|_{X}
    \geq \|W x\|_{X(\mathbf{r})}.
  \end{equation*}
  For the other inequality, note that the functions
  \begin{equation*}
    s\mapsto \int_0^1 \Bigl|\sum_{\substack{I \in \mathcal{D}\\ \inf I = 0}} r_I(u) a_I h_{I^+}(s)\Bigr| \dif u
    \qquad\text{and}\qquad
    s\mapsto \int_0^1 \Bigl|\sum_{\substack{I \in \mathcal{D}\\ \inf I = 0}} r_I(u) a_I h_{I^+}(s - 1/2)\Bigr| \dif u
  \end{equation*}
  are equimeasurable, and hence, by \Cref{dfn:HS-1d}, we obtain
  \begin{align*}
    \|W x\|_{X(\mathbf{r})}
    &\geq \frac{1}{2} \Bigl\|
      s\mapsto \int_0^1 \Bigl|\sum_{\substack{I \in \mathcal{D}\\ \inf I = 0}} r_I(u) a_I h_{I^+}(s)
      + \sum_{\substack{I \in \mathcal{D}\\ \inf I = 0}} r_I(u) a_I h_{I^+}(s - 1/2)\Bigr| \dif u
      \Bigr\|_X\\
    &= \frac{1}{2}\|x\|_{X(\mathbf{r})}.\qedhere
  \end{align*}
\end{proof}

\begin{pro}\label{pro:hsp:hyperplane}
  Let $X(\mathbf{r})\in\mathcal{HH}(\delta)$. Then the spaces $X_0(\mathbf{r})$ and $X(\mathbf{r})$
  are $162$-isomorphic to each other.
\end{pro}
\begin{proof}
  We are going to prove that $X_0(\mathbf{r})$ contains a complemented subspace that is isomorphic
  to its hyperplanes.  By \Cref{lem:haar-multiplier-cond}, the projection
  $P\colon X_0(\mathbf{r})\to X_0(\mathbf{r})$, given by
  \begin{equation*}
    \sum_{I\in \mathcal{D}} a_Ih_I
    \mapsto \sum_{\substack{I\in \mathcal{D}\\ \inf I = 0}} a_Ih_I,
  \end{equation*}
  is a well-defined contraction.  Put
  \begin{align*}
    E
    &= P(X_0(\mathbf{r})) = [h_I : I\in \mathcal{D},\ \inf I = 0],\\
    E_0
    &= [h_I : I\in \mathcal{D},\ \inf I = 0,\ \sup I \le \tfrac{1}{2}].
  \end{align*}
  Then $E$ is a $1$-complemented subspace of $X_0(\mathbf{r})$, and hence, $X_0(\mathbf{r})$ is
  $3$-isomorphic to $E\oplus F$ for some $F$.  We know from \Cref{lem:shift-operator} that $E$ is
  $2$-isomorphic to its hyperplane~$E_0$.  Moreover, it is clear that $X(\mathbf{r})$ is
  $3$-isomorphic to $X_0(\mathbf{r})\oplus \mathbb{R}$ and $E$ is $3$-isomorphic to
  $E_0\oplus\mathbb{R}$. Thus, we obtain
  \begin{equation}\label{eq:10}
    X(\mathbf{r})
    \overset{3}{\sim} X_0(\mathbf{r})\oplus \mathbb{R}
    \overset{3}{\sim} E\oplus F\oplus \mathbb{R}
    \overset{2}{\sim} E_0\oplus F\oplus\mathbb{R}
    \overset{3}{\sim} E\oplus F
    \overset{3}{\sim} X_0(\mathbf{r}),
  \end{equation}
  which completes the proof.
\end{proof}

\begin{rem}\label{rem:hsp:hyperplane}
  It follows from the proof of \Cref{pro:hsp:hyperplane} that an isomorphism
  $S\colon X(\mathbf{r})\to X_0(\mathbf{r})$ is given by the continuous linear extension of
  \begin{equation*}
    Sh_I =
    \begin{cases}
      h_{[0,1)}, & I = \varnothing,\\
      h_{I^+}, & I\in \mathcal{D},\ \inf I = 0,\\
      h_I, & I \in \mathcal{D},\ \inf I \ne 0.
    \end{cases}
  \end{equation*}
  In fact, a more detailed analysis of~\eqref{eq:10} shows that we always have $\|S\|\le 9$
  and $\|S^{-1}\|\le 18$.
\end{rem}

%%% Local Variables:
%%% mode: latex
%%% TeX-master: "main"
%%% End:

%auto-ignore

\section{Faithful Haar systems}
\label{sec:faithf-haar-systems}

We will now discuss \emph{faithful Haar systems}, a term which was coined in~\cite{MR4430957}.  A
faithful Haar system is a system of functions on $[0,1)$ which are blocks of the Haar system and
share many structural properties with the original Haar system. These and more generalized systems
were used extensively throughout the last decades.  In particular, we would like to highlight the
classical works of Gamlen-Gaudet~\cite{MR0328575}, Enflo-Maurey~\cite{MR0397386},
Alspach-Enflo-Odell~\cite{MR0500076} and Maurey~\cite{MR0586594}.  In order to ensure that the
orthogonal projection onto such a generalized system is bounded on $\mathrm{BMO}$,
P.~W.~Jones~\cite{MR0796906} found conditions which are nowadays referred to as \emph{Jones'
  compatibility conditions} (see also~\cite[p.~105]{MR2157745}).  For variants of Jones' conditions,
see e.g.~\cite{MR1283008,MR0955660,MR3819715}.

We will now introduce \emph{$(\varkappa_{I})_{I\in \mathcal{D}}$-faithful Haar systems}, which,
loosely speaking, allow for small gaps (in contrast to faithful Haar systems).

\begin{dfn}\label{dfn:almost-faithful}
  Let $\mathcal{B}_I$ be a finite subcollection of $\mathcal{D}$ for each $I\in\mathcal{D}$, and let
  $(\varepsilon_K)_{K\in \mathcal{D}}\in\{\pm 1\}^{\mathcal{D}}$ be a sequence of signs. Put
  $\tilde h_I = \sum_{K\in\mathcal{B}_I} \varepsilon_K h_K$, $I\in \mathcal{D}$. Moreover, let
  $(\varkappa_I)_{I\in \mathcal{D}}$ be a sequence of positive numbers with $0 < \varkappa_I\le 1$
  for all $I\in \mathcal{D}$. We say that $(\tilde h_I)_{I\in \mathcal{D}}$ is a
  \emph{$(\varkappa_I)_{I\in \mathcal{D}}$-faithful Haar system} if the following conditions are
  satisfied:
  \begin{enumerate}[(i)]
    \item\label{enu:dfn:almost-faithful:i} Each collection $\mathcal{B}_I$, $I\in \mathcal{D}$,
          consists of pairwise disjoint dyadic intervals, and we have
          $\mathcal{B}_I\cap \mathcal{B}_J = \emptyset$ for all $I\ne J\in \mathcal{D}$.
    \item\label{enu:dfn:almost-faithful:ii} For every $I\in\mathcal{D}$, we have
          $\mathcal{B}_{I^\pm}^* \subset \{\tilde h_I = \pm 1\}$ and
          $|\mathcal{B}_{I^{\pm }}^*| \geq \varkappa_I\cdot \frac{1}{2}|\mathcal{B}_I^{*}|$.
  \end{enumerate}
  If $\mathcal{B}_{[0,1)}^{*} = [0,1)$ and $\varkappa_I = 1$ for all $I\in \mathcal{D}$, then we
  simply say that $(\tilde h_I)_{I\in \mathcal{D}}$ is \emph{faithful}, and in this case, we will
  usually denote the system by $(\hat{h}_I)_{I\in \mathcal{D}}$.  If a
  $(\varkappa_I)_{I\in \mathcal{D}}$-faithful Haar system~$(\tilde{h}_I)_{I\in \mathcal{D}}$
  additionally satisfies $\mathcal{B}_I\subset \mathcal{D}_{n_I}$, $I\in \mathcal{D}$, for a
  strictly increasing sequence $(n_I)_{I\in \mathcal{D}}$ of non-negative integers, then we say that
  it is \emph{relative to the frequencies $(n_I)_{I\in \mathcal{D}}$}.
\end{dfn}

\begin{rem}\label{rem:almost-faithful:1}\label{rem:porous}
  We will summarize elementary yet important properties of faithful Haar systems
  $(\hat{h}_I)_{I\in \mathcal{D}}$.  Our first observation is that
  \begin{equation*}
    |\mathcal{B}_I^*| = |I|
    \quad \text{and}\quad
    \mathcal{B}_{I^\pm}^* = \{ \hat{h}_I = \pm 1\},
    \qquad I\in\mathcal{D}.
  \end{equation*}

  Any faithful Haar system $(\hat{h}_I)_{I\in \mathcal{D}}$ and the standard Haar system
  $(h_I)_{I\in \mathcal{D}}$ are distributionally equivalent, i.e., if $(a_I)_{I\in \mathcal{D}}$ is a sequence of scalars
  with $a_I\ne 0$ for at most finitely many $I\in \mathcal{D}$, then the functions
  $\sum_{I\in \mathcal{D}}a_I \hat{h}_I$ and $\sum_{I\in \mathcal{D}} a_Ih_I$ have the same distribution.  Moreover, for each
  $n\in \mathbb{N}_0$, the sets $\mathcal{B}_I^{*}$, $I\in \mathcal{D}_n$, form a partition of
  $[0,1)$, and we have the following equation relating the local and global properties of the system
  $(\hat{h}_I)_{I\in \mathcal{D}}$:
  \begin{equation}\label{eq:12}
    \frac{|K\cap \mathcal{B}_J^{*}|}{|\mathcal{B}_J^{*}|}
    = \frac{|K|}{|I|},
    \qquad K\in \mathcal{B}_I,\ J\subset I\in \mathcal{D}.
  \end{equation}
  A general $(\varkappa_I)_{I\in \mathcal{D}}$-faithful Haar system may violate
  equation~\eqref{eq:12}.  In some versions of Jones' compatibility conditions, this equation is
  replaced by an inequality.
\end{rem}

Next, we introduce some convenient notation for collections of dyadic intervals.

\begin{ntn}
  Let $\mathcal{A}\subset \mathcal{D}$.
  \begin{enumerate}[(i)]
    \item We set
          \begin{equation*}
            \mathcal{G}_0(\mathcal{A})
            = \{ I\in \mathcal{A} : I \text{ is maximal with respect to inclusion} \}.
          \end{equation*}
    \item For $n\in \mathbb{N}$, we recursively define the collections
          \begin{equation*}
            \mathcal{G}_n(\mathcal{A})
            = \mathcal{G}_0\Bigl( \mathcal{A}\setminus \bigcup_{k=0}^{n-1}\mathcal{G}_k(\mathcal{A}) \Bigr).
          \end{equation*}
    \item We say that \emph{$\mathcal{A}$ has finite generations} if $\mathcal{G}_n(\mathcal{A})$ is
          finite for every $n\in \mathbb{N}_0$.
    \item We put $\limsup(\mathcal{A}) = \bigcap_{n=0}^{\infty}\mathcal{G}_n^{*}(\mathcal{A})$,
          where $\mathcal{G}_n^{*}(\mathcal{A}) := \mathcal{G}_n(\mathcal{A})^{*}$ for all
          $n\in \mathbb{N}_0$.
  \end{enumerate}
\end{ntn}
Note that for every $n\in \mathbb{N}_0$, the elements of $\mathcal{G}_n(\mathcal{A})$ are pairwise disjoint.  Moreover, note that
if $n\ge 1$ and $I\in \mathcal{G}_n(\mathcal{A})$, then there exists a unique dyadic interval
$J\in \mathcal{G}_{n-1}(\mathcal{A})$ such that $I\subset J^+$ or $I\subset J^-$.  Hence, we have
$\mathcal{G}_n^{*}(\mathcal{A})\subset \mathcal{G}_{n-1}^{*}(\mathcal{A})$ for all $n\ge 1$. \Cref{fig:collections-Gn} shows the collections
$\mathcal{G}_0(\mathcal{A}),\mathcal{G}_1(\mathcal{A})$ and $\mathcal{G}_2(\mathcal{A})$ for a specific choice of $\mathcal{A}\subset \mathcal{D}$.

\begin{figure}[H]
  \centering \includegraphics{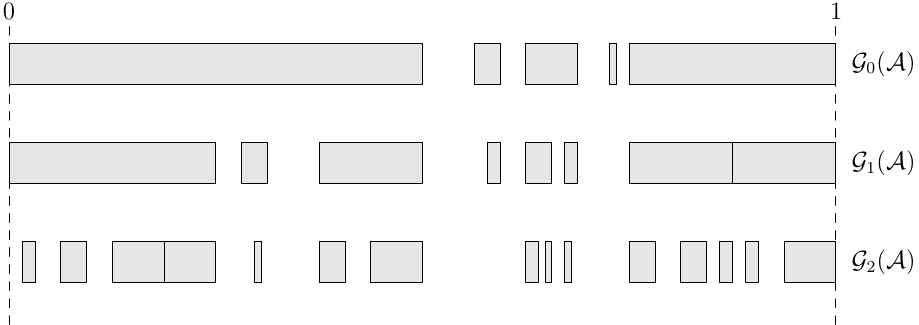}
  \caption{The collections $\mathcal{G}_n(\mathcal{A})$}
  \label{fig:collections-Gn}
\end{figure}

\begin{rem}\label{rem:faithful-hs-from-generations-and-signs}
  Suppose that $\hat{\mathcal{A}}\subset \mathcal{D}$ has finite generations and that
  $\mathcal{G}_n^{*}(\hat{\mathcal{A}}) = [0,1)$ for all $n\in \mathbb{N}_0$. Moreover, let
  $(\varepsilon_K)_{K\in \mathcal{D}}\in \{ \pm 1 \}^{\mathcal{D}}$ be a sequence of signs. Then we
  can construct a faithful Haar system $(\hat{h}_I)_{I\in \mathcal{D}}$ by putting
  \begin{equation*}
    \hat{h}_I = \sum_{K\in \hat{\mathcal{B}}_I} \varepsilon_Kh_K,
  \end{equation*}
  where $\hat{\mathcal{B}}_{[0,1)} = \mathcal{G}_0(\hat{\mathcal{A}})$ and
  \begin{equation*}
    \hat{\mathcal{B}}_{I^{\pm }}
    = \bigl\{ K\in \mathcal{G}_{n+1}(\hat{\mathcal{A}}) : K\subset \{ \hat{h}_I = \pm 1 \} \bigr\},
    \qquad I\in \mathcal{D}_n,\ n\in \mathbb{N}_0.
  \end{equation*}
\end{rem}

Next, we show that every $(\varkappa_I)_{I\in \mathcal{D}}$-faithful Haar system can be extended to
a faithful Haar system by adding additional Haar functions which ``fill the gaps''. This is
illustrated in \Cref{fig:faithful-and-almost-faithful}. Moreover, we prove that there exists a Haar
multiplier with norm~$1$ which maps the new system to the original one (see
\Cref{lem:almost-faithful:2}, below).
\begin{figure}[H]
  \centering \includegraphics[width=\textwidth]{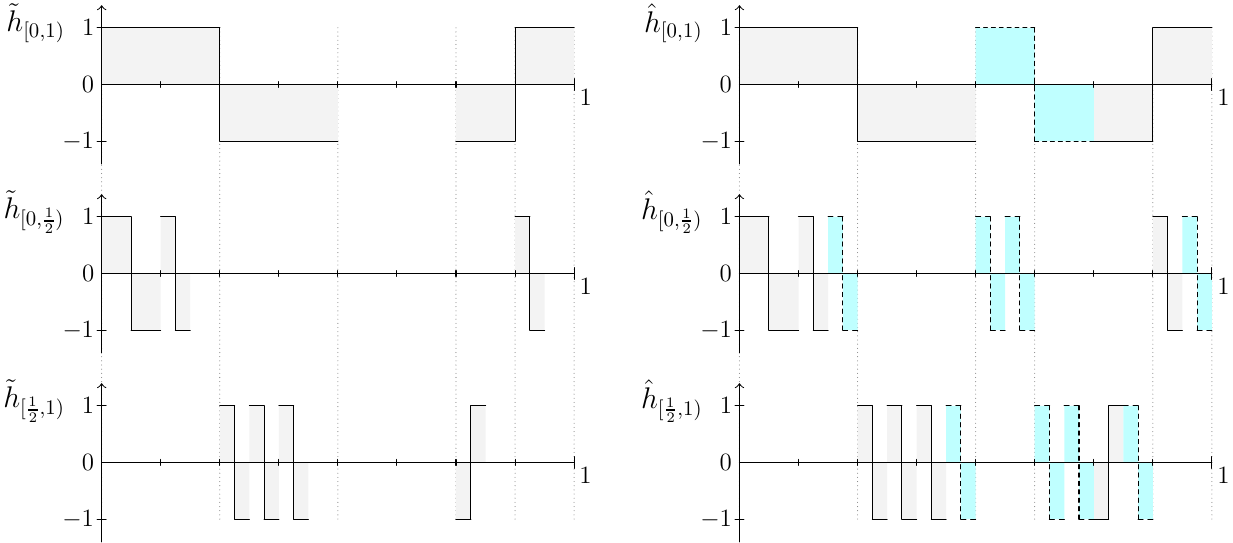}
  \caption{The first three functions of a $(\varkappa_I)_{I\in \mathcal{D}}$-faithful Haar
    system~$(\tilde{h}_I)_{I\in \mathcal{D}}$ which is extended to a faithful Haar
    system~$(\hat{h}_I)_{I\in \mathcal{D}}$ by adding the dashed Haar functions with the light blue
    shading}
  \label{fig:faithful-and-almost-faithful}
\end{figure}

\begin{lem}\label{lem:extend-finite-generations}
  Let $Y\in \mathcal{HH}_0(\delta)$ and suppose that $\mathcal{A}\subset \mathcal{D}$ has finite
  generations. Then there exists another collection $\hat{\mathcal{A}}\subset \mathcal{D}$ which has
  finite generations such that $\mathcal{G}_n^{*}(\hat{\mathcal{A}}) = [0,1)$ and
  $\mathcal{G}_n(\mathcal{A})\subset \mathcal{G}_n(\hat{\mathcal{A}})$ for all
  $n\in \mathbb{N}_0$. Moreover, there exists a bounded Haar multiplier $R\colon Y\to Y$ with
  $\|R\|\le 1$ such that for every $n\in \mathbb{N}_0$ and $K\in \mathcal{G}_n(\hat{\mathcal{A}})$,
  we have
  \begin{equation}\label{eq:29}
    Rh_K =
    \begin{cases}
      h_K, & K\in \mathcal{G}_n(\mathcal{A}),\\
      0, & K\in \mathcal{G}_n(\hat{\mathcal{A}})\setminus \mathcal{G}_n(\mathcal{A}).
    \end{cases}
  \end{equation}
\end{lem}

\begin{proof}
  Since $\mathcal{G}_0^{*}(\mathcal{A})$ is a finite union of dyadic intervals, there exists
  $n_0\in \mathbb{N}_0$ such that the complement $[0,1)\setminus \mathcal{G}_0^{*}(\mathcal{A})$ is
  a disjoint union of finitely many intervals from $\mathcal{D}_{n_0}$. By adding these intervals to
  $\mathcal{G}_0(\mathcal{A})$, we obtain a finite collection
  $\hat{\mathcal{G}}_0\subset \mathcal{D}$ with $\hat{\mathcal{G}}_0^{*} = [0,1)$. Next, since
  $\mathcal{G}_1^{*}(\mathcal{A})$ is a finite union of dyadic intervals, we can find
  $n_1\in \mathbb{N}_0$ such that $[0,1)\setminus \mathcal{G}_1^{*}(\mathcal{A})$ is a disjoint
  union of finitely many intervals in $\mathcal{D}_{n_1}$.  By adding these intervals to
  $\mathcal{G}_1(\mathcal{A})$, we obtain a collection $\hat{\mathcal{G}}_1$ with
  $\hat{\mathcal{G}}_1^{*} = [0,1)$, and by choosing $n_1$ sufficiently large, we can ensure that
  for each $K\in \hat{\mathcal{G}}_1$, there exists $L\in \hat{\mathcal{G}}_0$ such that
  $K\subset L^+$ or $L^-$.  By continuing this process, we obtain a sequence of collections
  $(\hat{\mathcal{G}}_n)_{n\in \mathbb{N}_0}$. Then the collection
  $\hat{\mathcal{A}} := \bigcup_{n=0}^{\infty}\hat{\mathcal{G}}_n$ has the desired properties since
  $\mathcal{G}_n(\hat{\mathcal{A}}) = \hat{\mathcal{G}}_n$ for all $n\in \mathbb{N}_0$.

  Now we define the sequence $(\rho_I)_{I\in \mathcal{D}}$ as
  \begin{equation*}
    \rho_I =
    \begin{cases}
      1, & \text{if } I\supset K \text{ for some } K\in \mathcal{A},\\
      0, & \text{else.}
    \end{cases}
  \end{equation*}
  Observe that if $\rho_I = 1$ for some $I\in \mathcal{D}$, then it follows that $\rho_J = 1$ for
  all $J\supset I$.  Hence, this sequence satisfies the conditions of
  \Cref{lem:haar-multiplier-cond}.  The corresponding Haar multiplier $R\colon Y\to Y$ defined by
  $Rh_I = \rho_Ih_I$, $I\in \mathcal{D}$, satisfies $\|R\|\le 1$ and $Rh_K = h_K$ for all
  $K\in \mathcal{A}$.  Moreover, if
  $L\in \mathcal{G}_n(\hat{\mathcal{A}})\setminus \mathcal{G}_n(\mathcal{A})$ for some
  $n\in \mathbb{N}_0$, then we have $Rh_L = 0$ because on the one hand, we cannot have $L\supset K$
  for any $K\in \mathcal{G}_m(\mathcal{A})$ with $m < n$, and on the other hand,
  $K\in \bigcup_{m=n}^{\infty}\mathcal{G}_m(\mathcal{A})$ implies that
  $K\subset \mathcal{G}_n^{*}(\mathcal{A})$, and so $K$ is disjoint from $L$.
\end{proof}

\begin{lem}\label{lem:almost-faithful:2}
  Let $Y\in \mathcal{HH}_0(\delta)$ and suppose that $(\tilde{h}_I)_{I\in \mathcal{D}}$ is a
  $(\varkappa_I)_{I\in \mathcal{D}}$-faithful Haar system for some sequence
  $(\varkappa_I)_{I\in \mathcal{D}}$ of scalars in $(0,1]$.  Then there exists a faithful Haar
  system $(\hat{h}_I)_{I\in \mathcal{D}}$ and a Haar multiplier $R\colon Y\to Y$ with
  $\|R\|\le 1$ such that $R\hat{h}_I = \tilde{h}_I$ for all $I\in \mathcal{D}$.
\end{lem}
\begin{proof}
  Write $\tilde{h}_I = \sum_{K\in \mathcal{B}_I} \varepsilon_Kh_K$, where
  $\mathcal{B}_I\subset \mathcal{D}$, $I\in \mathcal{D}$, and
  $(\varepsilon_K)_{K\in \mathcal{D}}\in \{ \pm 1 \}^{\mathcal{D}}$.  Then the collection
  $\mathcal{A} := \bigcup_{I\in \mathcal{D}}\mathcal{B}_I$ has finite generations.  By
  \Cref{lem:extend-finite-generations}, there exists another collection
  $\hat{\mathcal{A}}\subset \mathcal{D}$ with finite generations such that
  $\mathcal{G}_n^{*}(\hat{\mathcal{A}}) = [0,1)$ and
  $\mathcal{G}_n(\mathcal{A})\subset \mathcal{G}_n(\hat{\mathcal{A}})$ for all
  $n\in \mathbb{N}_0$. Moreover, there exists a bounded Haar multiplier $R\colon Y\to Y$ with
  $\|R\|\le 1$ such that equation~\eqref{eq:29} is satisfied for all $n\in \mathbb{N}_0$ and
  $K\in \mathcal{G}_n(\hat{\mathcal{A}})$.  Now let the faithful Haar system
  $(\hat{h}_I)_{I\in \mathcal{D}}$ and the associated collections
  $(\hat{\mathcal{B}}_I)_{I\in \mathcal{D}}$ be defined as in
  \Cref{rem:faithful-hs-from-generations-and-signs} using the collection $\hat{\mathcal{A}}$ and the
  signs $(\varepsilon_K)_{K\in \mathcal{D}}$.  Clearly, we have
  $\hat{\mathcal{B}}_I\cap \mathcal{G}_n(\mathcal{A}) = \mathcal{B}_I$ for all $I\in \mathcal{D}_n$,
  $n\in \mathbb{N}_0$. Thus, it follows from equation~\eqref{eq:29} that $R\hat{h}_I = \tilde{h}_I$
  for all $I\in \mathcal{D}$.
\end{proof}

%%% Local Variables:
%%% mode: latex
%%% TeX-master: "main"
%%% End:
 %auto-ignore

\section{Asymptotically curved Banach spaces}
\label{sec:curved}

In this section, we discuss asymptotically curved Banach spaces and uniformly asymptotically curved
sequences of Banach spaces (see \Cref{dfn:asympt-curved}).  We will need the following additional
concepts.

\begin{dfn}\label{dfn:upper-s-estimate}
  Let $(x_j)_{j=1}^{\infty}$ denote a sequence in a Banach space $E$ and let $1 < \tau < \infty$.
  We say that $(x_j)_{j=1}^{\infty}$ satisfies an \emph{upper $\tau$-estimate (in $E$) (with
    constant $C > 0$)} if
  \begin{equation*}
    \Bigl\|\sum_{j=1}^n x_j\Bigr\|_E
    \leq C \Bigl(\sum_{j=1}^n \|x_j\|_E^\tau\Bigr)^{1/\tau},
    \qquad n\in \mathbb{N}.
  \end{equation*}
  We say that $(x_j)_{j=1}^{\infty}$ satisfies an \emph{upper $\infty$-estimate (in $E$) (with
    constant $C > 0$)} if
  \begin{equation*}
    \Bigl\|\sum_{j=1}^n x_j\Bigr\|_E
    \leq C \max_{1\leq j\leq n} \|x_j\|_E,
    \qquad n\in \mathbb{N}.
  \end{equation*}
\end{dfn}

Next, we recall the notion of (Rademacher) type for a Banach space (see, e.g.,
\cite[Definition~1.e.12]{MR540367}).

\begin{dfn}\label{dfn:rademacher-type}
  Let $E$ be a Banach space, and let $1 < \tau \le 2$. We say that \emph{$E$ is of (Rademacher) type
    $\tau$} if there exists a constant $C > 0$ such that
  for every finite sequence of vectors $(x_j)_{j=1}^n$ in $E$, we have
  \begin{equation*}
    \int_0^1 \Bigl\| \sum_{j=1}^n r_j(t)x_j \Bigr\|_E\, \mathrm{d}t
    \le C \Bigl( \sum_{j=1}^n \|x_j\|_E^{\tau} \Bigr)^{1/\tau}.
  \end{equation*}
  If this holds, we say that $E$ is of (Rademacher) type $\tau$ \emph{with constant $C$}.
\end{dfn}

The proof of the following lemma is both elementary and straightforward, and therefore omitted.
\begin{lem}\label{lem:curved-sufficient}
  Let $(e_j)_{j=1}^{\infty}$ denote a Schauder basis of a Banach space $E$ and let $C > 0$.  Then
  the following statements are true:
  \begin{enumerate}[(i)]
  \item\label{enu:pro:curved-sufficient:ii} Suppose that $(e_j)_{j=1}^{\infty}$ is $C$-unconditional
    and $E$ has Rademacher type $\tau$ with constant~$C$ for some $1 < \tau \leq 2$, then every
    block basis $(x_j)_{j=1}^{\infty}$ of $(e_j)_{j=1}^{\infty}$ satisfies an upper
    $\tau$-estimate with constant $C^2$.
  \item\label{enu:pro:curved-sufficient:i} Suppose that every bounded block basis
    $(x_j)_{j=1}^{\infty}$ of $(e_j)_{j=1}^{\infty}$ satisfies an upper $\tau$-estimate for some
    $1 < \tau \leq \infty$, then $E$ is asymptotically curved (with respect to
    $(e_j)_{j=1}^{\infty}$).
  \end{enumerate}
\end{lem}

The following uniform version of \Cref{lem:curved-sufficient}~\eqref{enu:pro:curved-sufficient:i} is
taken from \cite{MR4299595}.
\begin{lem}\label{lem:uniformly-curved-sufficient}
  For each $k\in \mathbb{N}$, let $(e_{k,j})_{j=1}^{\infty}$ denote a Schauder basis of a Banach
  space $E_k$. Moreover, let $1 < \tau \leq \infty$ and $C > 0$, and suppose that for each
  $k\in \mathbb{N}$, every bounded block basis of $(e_{k,j})_{j=1}^{\infty}$ satisfies an upper
  $\tau$-estimate in $E_k$ with constant $C$.  Then $(E_k)_{k=1}^{\infty}$ is uniformly
  asymptotically curved with respect to the array $(e_{k,j})_{k,j=1}^{\infty}$.
\end{lem}

Another way to obtain a uniformly asymptotically curved sequence of Banach spaces is by repeatedly
taking the same asymptotically curved space, thus forming a constant sequence. This is proved in the
following lemma.

\begin{lem}\label{lem:curved-implies-uniformly-curved}
  Let $E$ be a Banach space with a Schauder basis $(e_j)_{j=1}^{\infty}$ and suppose that $E$ is
  asymptotically curved with respect to $(e_j)_{j=1}^{\infty}$. Put $e_{k,j} = e_j$ for all
  $k\in \mathbb{N}$. Then $(E,E,\dots)$ is uniformly asymptotically curved with respect to
  $(e_{k,j})_{k,j=1}^{\infty}$.
\end{lem}
\begin{proof}
  Let $(x_{k,j})_{k,j=1}^{\infty}$ be an array such that for every $k\in \mathbb{N}$,
  $(x_{k,j})_{j=1}^{\infty}$ is a block basis of $(e_j)_{j=1}^{\infty}$, and such that for some
  $C > 0$, we have
  \begin{equation*}
    \|x_{k,j}\|_E \le C, \qquad k,j\in \mathbb{N}.
  \end{equation*}
  We have to show that
  \begin{equation*}
    \lim_{n\to \infty} \sup_{k\in \mathbb{N}} \frac{1}{n}\Bigl\| \sum_{j=1}^n x_{k,j} \Bigr\|_E = 0.
  \end{equation*}
  Assume for a contradiction that there exist $\varepsilon > 0$ and sequences
  $(n_i)_{i=1}^{\infty}$ and $(k_i)_{i=1}^{\infty}$ of natural numbers such that
  $(n_i)_{i=1}^{\infty}$ is strictly increasing and
  \begin{equation}\label{eq:assume-for-contradiction}
    \frac{1}{n_i} \Bigl\| \sum_{j=1}^{n_i} x_{k_i,j} \Bigr\|_E \ge \varepsilon,
    \qquad i\in \mathbb{N}.
  \end{equation}
  Put $l_i = \max\operatorname{supp} x_{k_i, n_i}$ for all $i\in \mathbb{N}$. By passing to a
  subsequence, we may assume that $n_i\ge (1+4C/\varepsilon)l_{i-1}$ for all $i\ge 2$.  We will now
  construct a bounded block basis $(y_j)_{j=1}^{\infty}$ of $(e_j)_{j=1}^{\infty}$ and a strictly
  increasing sequence $(N_i)_{i=1}^{\infty}$ of natural numbers such that
  \begin{equation}\label{eq:block-basis-goal}
    \frac{1}{N_i} \Bigl\| \sum_{j=1}^{N_i} y_j \Bigr\|_E \ge \frac{\varepsilon}{2},
    \qquad i\in \mathbb{N},
  \end{equation}
  thus contradicting the hypothesis that $E$ is asymptotically curved with respect to
  $(e_j)_{j=1}^{\infty}$.

  Put $N_1 = n_1$ and
  \begin{equation*}
    y_1 = x_{k_1, 1}, \dots, y_{n_1} = x_{k_1,n_1}.
  \end{equation*}
  By~\eqref{eq:assume-for-contradiction}, inequality~\eqref{eq:block-basis-goal} holds for
  $i=1$. Now let $i\ge 2$ and assume that we have already chosen $N_1,\dots,N_{i-1}$ and picked
  $y_1,\dots,y_{N_{i-1}}\in \{ x_{k,j} : k,j\in \mathbb{N} \}$ such that $(y_1,\dots,y_{N_{i-1}})$
  is a finite block basis of $(e_j)_{j=1}^{\infty}$ and such that
  $y_{N_{i-1}} = x_{k_{i-1},n_{i-1}}$. Consider the vectors $x_{k_i,1},\dots,x_{k_i,n_i}$. If we
  skip the first $l_{i-1}$ of these vectors, then the supports of the remaining vectors are clearly
  subsets of $\{ l_{i-1}+1, l_{i-1}+2,\dots \}$. Thus, if we define
  $N_i = N_{i-1} + n_i - l_{i-1} > N_{i-1}$ and
  \begin{equation*}
    y_{N_{i-1}+1} = x_{k_i, l_{i-1}+1}, \dots, y_{N_i} = x_{k_i, n_i},
  \end{equation*}
  then $(y_1,\dots,y_{N_i})$ is a finite block basis of $(e_j)_{j=1}^{\infty}$.  Observe that
  $N_{i-1}\le \max \operatorname{supp}y_{N_{i-1}} = l_{i-1}$ and hence $N_i\le n_i$. Using
  these observations, exploiting that $\|y_j\|_E\le C$ for all $j$, and
  using~\eqref{eq:assume-for-contradiction}, we obtain
  \begin{align*}
    \Bigl\| \sum_{j=1}^{N_i} y_j \Bigr\|_E
    &\ge \Bigl\| \sum_{j=N_{i-1}+1}^{N_i} y_j \Bigr\|_E - CN_{i-1}
      = \Bigl\| \sum_{j=l_{i-1}+1}^{n_i} x_{k_i,j} \Bigr\|_E - CN_{i-1}\\
    &\ge \Bigl\| \sum_{j=1}^{n_i} x_{k_i,j} \Bigr\|_E - Cl_{i-1} - CN_{i-1}
      \ge \varepsilon n_i - 2\cdot Cl_{i-1}\\
    &\ge \frac{\varepsilon}{2}n_i + \frac{\varepsilon}{2}\cdot \frac{4C}{\varepsilon}l_{i-1} - 2Cl_{i-1}
      \ge \frac{\varepsilon}{2}N_i,
  \end{align*}
  which proves~\eqref{eq:block-basis-goal}.
\end{proof}

Next, we provide some conditions under which Haar system Hardy spaces and sequences of such spaces
are (uniformly) asymptotically curved.  Clearly, it does not make a difference whether we consider
Haar system Hardy spaces spaces $X(\mathbf{r})\in \mathcal{HH}(\delta)$ equipped with the Haar basis
$(h_I)_{I\in \mathcal{D}^+}$ or their closed subspaces $X_0(\mathbf{r})$ equipped with the Haar
basis $(h_I)_{I\in \mathcal{D}}$.

\begin{pro}\label{pro:hshs-curved-examples}
  Let $1\le p<\infty$, let $\mathbf{r}\in \mathcal{R}$, and put $X = L^p$ if $1\le p < \infty$ and
  $X = [h_I]_{I\in \mathcal{D}^+}\subset L^{\infty}$ if $p = \infty$. Then the space $X(\mathbf{r})$
  is asymptotically curved with respect to the Haar system $(h_I)_{I\in \mathcal{D}^+}$ if and only
  if
  \begin{equation*}
    1 < p < \infty
    \qquad\text{or}\qquad
    \text{$p = \infty$ and $\mathbf{r}$ is independent}.
  \end{equation*}
\end{pro}

\begin{proof}
  Considering a sequence of disjointly supported functions with norm~$1$ which are blocks of the
  Haar system shows that for $X_1 = L^1$ and $\mathbf{r}\in \mathcal{R}$, the space
  $X_1(\mathbf{r})\in \mathcal{H}\mathcal{H}(\delta)$ is not asymptotically curved with respect to
  the Haar basis.  Moreover, since any independent sequence of Rademacher functions in $L^{\infty}$
  is equivalent to the unit vector basis of $\ell^1$, we find that
  $X_2 = [h_I]_{I\in \mathcal{D}^+}\subset L^{\infty}$ is also not asymptotically curved.

  However, if $\mathbf{r}$ is an independent sequence, then $X_2(\mathbf{r})$ is in fact
  asymptotically curved with respect to the Haar basis according to
  \Cref{lem:curved-sufficient}~\eqref{enu:pro:curved-sufficient:i}, as every bounded block basis of
  $(h_I)_{I\in \mathcal{D}^+}$ satisfies an upper $2$-estimate (cf.~\cite[Lemma~4.2]{MR3819715}). To
  see this, let $(x_j)_{j=1}^{\infty}$ denote a bounded block basis of $(h_I)_{I\in \mathcal{D}^+}$
  and observe that for all $n\in \mathbb{N}$
  \begin{align*}
    \Bigl\| \sum_{j=1}^n x_j \Bigr\|_{X_2(\mathbf{r})}
    &= \sup_t \int_{0}^{1}
      \Bigl| \sum_{j=1}^n \sum_{K\in \mathcal{D}^+} r_K(u) \frac{\langle h_K, x_j \rangle}{|K|} h_K(t) \Bigr|
      \dif u\\
    &\leq \sup_t \biggl(
      \sum_{j=1}^n \sum_{K\in \mathcal{D}^+} \Bigl(\frac{\langle h_K, x_j \rangle}{|K|} h_K(t) \Bigr)^{2}
      \biggl)^{1/2}\\
    &\leq \biggl(
      \sum_{j=1}^n \sup_t \sum_{K\in \mathcal{D}^+} \Bigl(\frac{\langle h_K, x_j \rangle}{|K|} h_K(t) \Bigr)^{2}
      \biggr)^{1/2}.
  \end{align*}
  We conclude this argument by Khintchine's inequality, which yields
  \begin{equation*}
    \sup_t \bigg( \sum_{K\in \mathcal{D}^+} \Bigl(\frac{\langle h_K, x_I \rangle}{|K|} h_K(t) \Bigr)^{2} \biggr)^{1/2}
    \le C\|x_I\|_{X_2(\mathbf{r})}
  \end{equation*}
  for some absolute constant $C > 0$.

  Finally, let $1 < p < \infty$, $\mathbf{r}\in \mathcal{R}$, and put $X_3 = L^p$.  Recall that by
  \Cref{rem:classical-examples}, the identity operator is an isomorphism between $X_3(\mathbf{r})$
  and $X_3$.  Since $L^p$ has Rademacher type $\min(2,p)$, we record that by
  \Cref{lem:curved-sufficient}, $X_3(\mathbf{r})$ is asymptotically curved as well.
\end{proof}

\begin{rem}\label{rem:hshs-curved-general}
  If $X(\mathbf{r})\in \mathcal{HH}(\delta)$ is an arbitrary Haar system Hardy space, then one can use
  \Cref{lem:curved-sufficient} to check if $X(\mathbf{r})$ is asymptotically curved, either by
  verifying the condition in~\eqref{enu:pro:curved-sufficient:i} directly or by
  applying~\eqref{enu:pro:curved-sufficient:ii} if it is known that $X(\mathbf{r})$ has Rademacher
  type $\tau > 1$ and that the Haar system is an unconditional basis of $X(\mathbf{r})$. Recall that
  the Haar basis is always $1$-unconditional in~$X(\mathbf{r})$ if~$\mathbf{r}$ is an independent
  sequence of Rademacher functions. Moreover, according to~\cite[Theorem~2.c.6]{MR540367}, the Haar
  system is unconditional in a separable r.i.~function space $X$ on $[0,1)$ if and only if $X$ has
  non-trivial Boyd indices.
\end{rem}

\begin{rem}\label{rem:hshs-uniformly-curved-general}
  Finally, consider a sequence of Haar system Hardy spaces $(Z_k)_{k=1}^{\infty}$, and for each
  $k\in \mathbb{N}$, let $(h_{k,I})_{I\in \mathcal{D}^+}$ denote the Haar basis of $Z_k$.  Suppose
  that either of the following conditions is satisfied:
  \begin{itemize}
    \item There exist $1<\tau\le 2$ and a uniform
          constant $C > 0$ such that for every $k$, the Haar system is $C$-unconditional in $Z_k$
          and the space $Z_k$ has Rademacher type~$\tau$ with constant~$C$.
    \item We have $Z_k = Z_1$ for all $k\in \mathbb{N}$,
          and $Z_1$ is asymptotically curved with respect to the Haar system.
  \end{itemize}
  Then $(Z_k)_{k=1}^{\infty}$ is uniformly asymptotically curved with respect to
  $(h_{k,I})_{k\in \mathbb{N}, I\in \mathcal{D}^+}$.  This follows from \Cref{lem:curved-sufficient}
  together with \Cref{lem:uniformly-curved-sufficient} in the first case and from
  \Cref{lem:curved-implies-uniformly-curved} in the second case.
\end{rem}

%%% Local Variables:
%%% mode: latex
%%% TeX-master: "main"
%%% End:
 %auto-ignore

\section{Embeddings and projections on Haar system Hardy spaces}
\label{sec:embedd-proj-haar}

In this section, we will define the fundamental operators $A$ and $B$ associated with a
$(\varkappa_I)_{I\in \mathcal{D}}$-faithful Haar system, and we will prove that they are bounded if
the numbers $\varkappa_I$ are sufficiently small. In doing so, we will lay the foundation for
proving our main results.

\begin{pro}\label{pro:A-B-bounded}
  Let $Y\in \mathcal{HH}_0(\delta)$. Let $(\hat{h}_I)_{I\in \mathcal{D}}$ be a faithful Haar system
  and define the operators $\hat{A}, \hat{B}\colon Y\to Y$ by
  \begin{equation}\label{eq:B-A-def}
    \hat{B}x
    = \sum_{I\in \mathcal{D}} \frac{\langle h_I, x \rangle}{|I|}\hat{h}_I
    \qquad \text{and}\qquad
    \hat{A}x
    = \sum_{I\in \mathcal{D}} \frac{\langle \hat{h}_I, x \rangle}{|I|} h_I.
  \end{equation}
  Then we have $\hat{A}\hat{B} = I_Y$ and $\|\hat{A}\| = \|\hat{B}\| = 1$.
\end{pro}

Before proving \Cref{pro:A-B-bounded}, we state the following direct consequence.
\begin{cor}\label{cor:norm-in-dual-space}
  Let $\mathcal{B}\subset \mathcal{D}$ be a finite collection of pairwise disjoint dyadic intervals,
  and let $(\varepsilon_K)_{K\in \mathcal{D}}\in \{ \pm 1 \}^{\mathcal{D}}$ be a
  sequence of signs. Then we have
  \begin{equation*}
    \Bigl\| \sum_{K\in \mathcal{B}} \varepsilon_Kh_K \Bigr\|_{Y^{*}}\le 1.
  \end{equation*}
\end{cor}

\begin{proof}[Proof of \Cref{cor:norm-in-dual-space}]
  The proof follows either by an elementary direct computation or, alternatively, by exploiting the
  estimate for $\|\hat{A}\|$ in \Cref{pro:A-B-bounded} and using \Cref{lem:haar-multiplier-cond} (see
  also \Cref{lem:almost-faithful:2}).
\end{proof}

\begin{proof}[Proof of \Cref{pro:A-B-bounded}]
  Write $Y = X_0(\mathbf{r})$ for suitable $X\in \mathcal{H}(\delta)$ and
  $\mathbf{r}\in \mathcal{R}$. Suppose that our faithful Haar system
  $(\hat{h}_I)_{I\in \mathcal{D}}$ is given by
  \begin{equation*}
    \hat{h}_I
    = \sum_{K\in \mathcal{B}_I} \varepsilon_K h_K,
    \qquad I\in \mathcal{D},
  \end{equation*}
  where $\varepsilon = (\varepsilon_K)_{K\in \mathcal{D}}$ is a sequence of signs and
  $\mathcal{B}_I$ is a finite subset of $\mathcal{D}$ for each $I\in \mathcal{D}$.
  In order to prove
  $\|\hat{B}\| = 1$, we fix $x = \sum_{I\in \mathcal{D}}a_Ih_I\in H_0$, where $a_I\ne 0$ for at most finitely many
  $I\in \mathcal{D}$.  Then we have
  \begin{equation*}
    \hat{B}x
    = \sum_{I\in \mathcal{D}} a_I \hat{h}_I
    = \sum_{I\in \mathcal{D}} \sum_{K\in \mathcal{B}_I} a_I\varepsilon_K h_K
  \end{equation*}
  and hence
  \begin{equation}\label{eq:1}
    \|\hat{B}x\|_{X(\mathbf{r})}
    = \Bigl\| s\mapsto \int_0^1 \Bigl| \sum_{I\in \mathcal{D}}\sum_{K\in \mathcal{B}_I} r_K(u)a_I\varepsilon_K h_K(s) \Bigr|\, \mathrm{d}u \Bigr\|_X.
  \end{equation}
  Now for fixed $s \in [0,1)$, note that by the faithfulness of $(\hat{h}_I)_{I\in \mathcal{D}}$, for every
  $I\in \mathcal{D}$ there exists at most one interval $K\in \mathcal{B}_I$ with $s\in h_K$.  Thus, we may replace
  $r_K(u)$ by $r_I(u)$ in~\eqref{eq:1}, obtaining
  \begin{align*}
    \|\hat{B}x\|_{X(\mathbf{r})}
    &= \Bigl\| s\mapsto \int_0^1 \Bigl| \sum_{I\in \mathcal{D}} r_I(u)a_I \sum_{K\in \mathcal{B}_I} \varepsilon_Kh_K(s) \Bigr|\, \mathrm{d}u \Bigr\|_X\\
    &= \Bigl\| s\mapsto \int_0^1 \Bigl| \sum_{I\in \mathcal{D}} r_I(u)a_I \hat{h}_I(s) \Bigr|\, \mathrm{d}u \Bigr\|_X\\
    &= \Bigl\| s\mapsto \int_0^1 \Bigl| \sum_{I\in \mathcal{D}} r_I(u)a_I h_I(s) \Bigr|\, \mathrm{d}u \Bigr\|_X
      = \|x\|_{X(\mathbf{r})}.
  \end{align*}

  Next, we prove that $\|\hat{A}\| = 1$.  Let $x\in H_0$ be defined as above and let
  $N\in \mathbb{N}$ be sufficiently large such that for all
  $K\in \bigcup_{I\in \mathcal{D}\setminus \mathcal{D}_{<N}}\mathcal{B}_I$, we have $a_K = 0$
  (hence, we have $\langle \hat{h}_I, x \rangle = 0$ for all
  $I\in \mathcal{D}\setminus \mathcal{D}_{<N}$).  Let $\mathcal{F}$ denote the $\sigma$-algebra
  generated by the sets $\mathcal{B}_I^{*}$, $I\in \mathcal{D}_N$.  We will show that for every
  $K\in \mathcal{D}$, we have
  \begin{equation}\label{eq:2}
    \mathbb{E}^{\mathcal{F}} h_K =
    \begin{cases}
      0, & K\in \mathcal{D}\setminus \bigcup_{I\in \mathcal{D}_{<N}}\mathcal{B}_I,\\
      \varepsilon_K\frac{|K|}{|I|}\hat{h}_I, & K\in \mathcal{B}_I,\ I\in \mathcal{D}_{<N}.
    \end{cases}
  \end{equation}
  For every $I\in \mathcal{D}_N$, we have the following expansion, which is completely analogous to
  its well-known counterpart for the standard Haar system:
  \begin{equation*}
    \chi_{\mathcal{B}_I^{*}}
    = 2^{-N}\chi_{[0,1)} + \sum_{\substack{J\in \mathcal{D}\\ J \supsetneq I}} \frac{2^{-N}}{|J|}h_J(I)\hat{h}_J,
  \end{equation*}
  where $h_J(I)$ is the value that $h_J$ assumes on $I$.  Thus, for all
  $K\in \mathcal{D}\setminus \bigcup_{J\in \mathcal{D}_{<N}}\mathcal{B}_J$, we have
  $\mathbb{E}^{\mathcal{F}}h_K = 0$.  On the other hand, if $I\in \mathcal{D}_N$ and
  $K\in \mathcal{B}_{J_0}$ for some $J_0\in \mathcal{D}_{<N}$ with $J_0\supsetneq I$, then we have
  \begin{equation*}
    \langle \chi_{\mathcal{B}_I^{*}}, h_K \rangle
    = \sum_{\substack{J\in \mathcal{D}\\ J\supsetneq I}} \frac{|I|}{|J|}h_J(I) \langle \hat{h}_J, h_K \rangle
    = \frac{|I|}{|J_0|} h_{J_0}(I) \varepsilon_K |K|.
  \end{equation*}
  This implies that
  \begin{equation*}
    \mathbb{E}^{\mathcal{F}} h_K = \sum_{\substack{I\in \mathcal{D}_N\\ I\subset J_0}} \frac{1}{|I|} \langle \chi_{\mathcal{B}_I^{*}}, h_K \rangle \chi_{\mathcal{B}_I^{*}}
    = \varepsilon_K \frac{|K|}{|J_0|}  \sum_{\substack{I\in \mathcal{D}_N\\ I\subset J_0}} h_{J_0}(I)\chi_{\mathcal{B}_{I}^{*}}
    = \varepsilon_K \frac{|K|}{|J_0|} \hat{h}_{J_0},
  \end{equation*}
  which completes the proof of~\eqref{eq:2}.

  Now observe that
  \begin{equation*}
    \mathbb{E}^{\mathcal{F}}x = \sum_{I\in \mathcal{D}_{<N}} \sum_{K\in \mathcal{B}_I} a_K\varepsilon_K \frac{|K|}{|I|}\hat{h}_I
    = \sum_{I\in \mathcal{D}_{<N}} \frac{\langle \hat{h}_I, x \rangle}{|I|} \hat{h}_I
    = \hat{B}\hat{A}x.
  \end{equation*}
  We define $\mathcal{C} = \mathcal{D}\setminus \bigcup_{I\in \mathcal{D}_{<N}}\mathcal{B}_I$ and
  split the function $x$ into two parts accordingly:
  \begin{align*}
    \|x\|_{X(\mathbf{r})}
    &= \Bigl\| s\mapsto \int_0^1 \Bigl| \sum_{I\in \mathcal{D}_{<N}} \sum_{K\in \mathcal{B}_I} r_K(u)a_Kh_K(s)
      + \sum_{K\in \mathcal{C}} r_K(u)a_Kh_K(s)\Bigr|\, \mathrm{d}u \Bigr\|_X\\
    &= \Bigl\| s\mapsto \int_0^1\int_0^1 \Bigl| \sum_{I\in \mathcal{D}_{<N}} r_I(u)\sum_{K\in \mathcal{B}_I} a_Kh_K(s)
      + \sum_{K\in \mathcal{C}} r_K(v)a_Kh_K(s)\Bigr|\, \mathrm{d}u\, \mathrm{d}v \Bigr\|_X,
  \end{align*}
  where we again replaced $r_K(u)$ by $r_I(u)$ in the first sum.  Now we use
  \Cref{lem:conditional-expectation}, Jensen's inequality for conditional expectations and
  \Cref{pro:HS-1d}~\eqref{pro:HS-1d:v} to obtain
  \begin{align*}
    \|x\|_{X(\mathbf{r})}
    &\ge \Bigl\| s\mapsto \int_0^1\int_0^1 \Bigl| \sum_{I\in \mathcal{D}_{<N}} r_I(u)\sum_{K\in \mathcal{B}_I} a_K (\mathbb{E}^{\mathcal{F}}h_K)(s)
      + \sum_{K\in \mathcal{C}} r_K(v)a_K(\mathbb{E}^{\mathcal{F}}h_K)(s)\Bigr|\, \mathrm{d}u\, \mathrm{d}v \Bigr\|_X.
  \end{align*}
  According to~\eqref{eq:2}, we have
  $\mathbb{E}^{\mathcal{F}}h_K = \varepsilon_K\frac{|K|}{|I|}\hat{h}_I$ in the first sum and
  $\mathbb{E}^{\mathcal{F}}h_K = 0$ in the second sum. Hence, it follows that
  \begin{align*}
    \|x\|_{X(\mathbf{r})}
    &\ge \Bigl\| s\mapsto \int_0^1 \Bigl| \sum_{I\in \mathcal{D}_{<N}}r_I(u) \Bigl( \sum_{K\in \mathcal{B}_I} \varepsilon_Ka_K \frac{|K|}{|I|} \Bigr) \hat{h}_I(s) \Bigr|\, \mathrm{d}u \Bigr\|_X\\
    &= \Bigl\| s\mapsto \int_0^1 \Bigl| \sum_{I\in \mathcal{D}_{<N}} r_I(u) \frac{\langle \hat{h}_I,x \rangle}{|I|} \sum_{K\in \mathcal{B}_I} \varepsilon_Kh_K(s) \Bigr|\, \mathrm{d}u \Bigr\|_X.
  \end{align*}
  Now we swap back $r_I(u)$ with $r_K(u)$ and obtain
  \begin{align*}
    \|x\|_{X(\mathbf{r})}
    \ge \Bigl\| s\mapsto \int_0^1 \Bigl| \sum_{I\in \mathcal{D}_{<N}} \frac{\langle \hat{h}_I, x \rangle}{|I|} \sum_{K\in \mathcal{B}_I} r_K(u)\varepsilon_K h_K(s) \Bigr|\, \mathrm{d}u \Bigr\|
    = \|\hat{B}\hat{A}x\|_{X(\mathbf{r})} = \|\hat{A}x\|_{X(\mathbf{r})}.
  \end{align*}

  By extending both operators continuously from $H_0$ to $Y$, we obtain
  $\hat{A},\hat{B}\colon Y\to Y$ as defined in~\eqref{eq:B-A-def}, where both series converge in
  norm.
\end{proof}

\begin{thm}\label{thm:projection-1d}
  Let $Y\in \mathcal{HH}_0(\delta)$. Suppose that $(\tilde{h}_I)_{I\in \mathcal{D}}$ is a $(\varkappa_I)_{I\in \mathcal{D}}$-faithful
  Haar system for some family $(\kappa_I)_{I\in \mathcal{D}}$ of scalars in $(0,1]$ that satisfies
  \begin{equation*}
    \sigma := \sum_{I\in \mathcal{D}} (1 - \varkappa_I) < 1.
  \end{equation*}
  Define the operators $A,B\colon Y\to Y$ by
  \begin{equation*}
    B x
    = \sum_{I\in \mathcal{D}} \frac{\langle h_I, x\rangle}{\|h_I\|_2^2} \tilde{h}_I
    \qquad\text{and}\qquad
    A x
    = \sum_{I\in \mathcal{D}} \frac{\langle \tilde{h}_I, x\rangle}{\|\tilde{h}_I\|_2^2} h_I,\qquad x\in Y.
  \end{equation*}
  Then we have $AB = I_Y$, and the operators $A$ and $B$ satisfy
  \begin{equation*}%\label{eq:thm:projection-1d:estimates}
    \begin{aligned}
      \|B\| = 1
      \qquad \text{and}\qquad
      \|A\| \leq \frac{1}{\mu}\cdot \frac{1 + 3\sigma}{1 - \sigma},
    \end{aligned}
  \end{equation*}
  where $\mu = |\{ \tilde{h}_{[0,1)} \ne 0 \}|$.
\end{thm}

\begin{proof}%{Proof of \Cref{thm:projection-1d}}
  Using \Cref{lem:almost-faithful:2} and \Cref{lem:haar-multiplier}, we can find a faithful Haar
  system $(\hat{h}_I)_{I\in \mathcal{D}}$ and a Haar multiplier
  $R\colon Y\to Y$ with $\|R\|\leq 1$ such that
  $\tilde{h}_I = R \hat{h}_{I}$, $I\in \mathcal{D}$.  We know from \Cref{pro:A-B-bounded} that the
  operators $\hat{B},\hat{A}\colon Y\to Y$ given by
  \begin{equation*}
    \hat{B}x
    = \sum_{I\in \mathcal{D}} \frac{\langle h_I, x \rangle}{|I|} \hat{h}_I
    \qquad \text{and} \qquad
    \hat{A}x
    = \sum_{I\in \mathcal{D}} \frac{\langle \hat{h}_I, x \rangle}{|I|}h_I
  \end{equation*}
  satisfy $\hat{A}\hat{B} = I_Y$ and $\|\hat{A}\| = \|\hat{B}\| = 1$.  For
  $I\in \mathcal{D}$, let $\mathcal{B}_I\subset \mathcal{D}$ denote the Haar support of
  $\tilde{h}_I$, and define $M\colon Y\to Y$ as the linear extension of
  $Mh_I = m_Ih_I$, where $m_I = |I|/|\mathcal{B}_I^{*}|$, $I\in \mathcal{D}$.  We will show that the
  operator $M$ is bounded.  To this end, let $n\in \mathbb{N}_0$ and $I\in \mathcal{D}$ and observe
  that repeatedly exploiting \Cref{dfn:almost-faithful}~\eqref{enu:dfn:almost-faithful:ii} yields
  \begin{equation*}
    |\mathcal{B}_I^{*}|
    \ge \Bigl( \prod_{J\in \mathcal{D} : J\supsetneq I} \varkappa_J \Bigr)\frac{|\mathcal{B}_{[0,1)}^{*}|}{2^n}
    \ge \Bigl( 1 - \sum_{J\in \mathcal{D}} (1 - \varkappa_J) \Bigr) \frac{|\mathcal{B}_{[0,1)}^{*}|}{2^n}
    = (1-\sigma)\mu|I|.
  \end{equation*}
  Moreover, we clearly have $|\mathcal{B}_I^{*}| \le \mu |I|$ for all $I\in \mathcal{D}$.  Hence,
  \begin{equation}\label{eq:I-BI}
    1 - \sigma
    \le \frac{|\mathcal{B}_I^{*}|}{\mu |I|}
    \le 1,
    \qquad I\in \mathcal{D}.
  \end{equation}
  Note that $m_{[0,1)} = 1/\mu$.  Now let $I\in \mathcal{D}\setminus \{ [0,1) \}$ and consider
  \begin{equation*}
    |m_I - m_{\pi(I)}|
    = \biggl| \frac{|I|}{|\mathcal{B}_I^{*}|} - \frac{2|I|}{|\mathcal{B}_{\pi(I)}^{*}|} \biggr|.
  \end{equation*}
  Using the inequalities $|\mathcal{B}_I^{*}|\le \frac{1}{2}|\mathcal{B}_{\pi(I)}^{*}|$ and
  $|\mathcal{B}_{\pi(I)}^{*}|\le \frac{2}{\varkappa_{\pi(I)}}|\mathcal{B}_I^{*}|$, we obtain
  \begin{equation*}
    0
    \le \frac{|I|}{|\mathcal{B}_I^{*}|} - \frac{2|I|}{|\mathcal{B}_{\pi(I)}^{*}|}
    \le (1 - \varkappa_{\pi(I)}) \frac{|I|}{|\mathcal{B}_I^{*}|}.
  \end{equation*}
  Together with~\eqref{eq:I-BI}, this yields
  \begin{equation*}
    |m_I - m_{\pi(I)}|
    \le \frac{1 - \varkappa_{\pi(I)}}{\mu(1 - \sigma)},
    \qquad I\in \mathcal{D}\setminus \{ [0,1) \}.
  \end{equation*}
  Thus, by \Cref{lem:haar-multiplier}, we have
  \begin{align*}
    \|M\|
    &\le \frac{1}{\mu}\Bigl( 1 + \frac{2}{1 - \sigma} \sum_{I\in \mathcal{D}\setminus \{ [0,1) \}} (1 - \varkappa_{\pi(I)}) \Bigr)\\
    &= \frac{1}{\mu}\Bigl( 1 + \frac{4\sigma}{1 - \sigma} \Bigr)
      = \frac{1}{\mu}\cdot \frac{1 + 3\sigma}{1 - \sigma}
      < \infty.
  \end{align*}
  Finally, observe that $B = R\hat{B}$ and $A = M\hat{A}R$ because we have for all $x\in H$:
  \begin{align*}
    M\hat{A}Rx
    &= M\Bigl( \sum_{I\in \mathcal{D}} \frac{\langle \hat{h}_I, Rx \rangle}{|I|}h_I \Bigr)
      = \sum_{I\in \mathcal{D}}\frac{\langle R\hat{h}_I, x \rangle}{|\mathcal{B}_I^{*}|}h_I
      = \sum_{I\in \mathcal{D}}\frac{\langle \tilde{h}_I, x \rangle}{|\mathcal{B}_I^{*}|}h_I
      = Ax.
  \end{align*}
  Thus, we conclude that $\|B\|\le \|R\|\|\hat{B}\|\le 1$ and
  $\|A\|\le \|M\|\|\hat{A}\|\|R\|\le \frac{1}{\mu}\cdot \frac{1 + 3\sigma}{1 - \sigma}$, as claimed.
\end{proof}

%%% Local Variables:
%%% mode: latex
%%% TeX-master: "main"
%%% End:
 %auto-ignore

\section{Stabilization of Haar multipliers}
\label{sec:stabilization}

The key ingredient for proving our main results is the observation that any bounded Haar multiplier
on a Haar system Hardy space~$Y$ can be reduced to a \emph{stable} Haar multiplier, i.e., a Haar
multiplier whose entries have very small variation (see \Cref{pro:stabilization}).  Such a stable
Haar multiplier is in turn close to $cI_Y$ for some scalar value~$c$.  The stabilization is achieved
by utilizing randomized faithful Haar systems.

Let $Y\in \mathcal{HH}_0(\delta)$, and let $D\colon Y\to Y$ denote a bounded Haar multiplier. Then
its entries $(d_I)_{I\in \mathcal{D}}$ are defined as
\begin{equation*}
  d_I = \frac{\langle h_I, Dh_I \rangle}{|I|},\qquad I\in \mathcal{D}.
\end{equation*}

In the following, we denote the product measure on $\{\pm 1\}^{\mathcal{D}}$ of the normalized
uniform measure on $\{\pm 1\}$ by~$\mathbb{P}$.  Moreover, for fixed
$\varepsilon = (\varepsilon_J)_{J\in \mathcal{D}}\in\{\pm 1\}^{\mathcal{D}}$ and
$n\in \mathbb{N}_0$, we define
\begin{equation*}
  \mathbb{P}_n^\varepsilon(\cdot) = \mathbb{P}(\cdot\mid \{(\theta_J)_{J\in \mathcal{D}} : \theta_J = \varepsilon_J,\ J\in \mathcal{D}\setminus \mathcal{D}_n\})
\end{equation*}
and denote the corresponding conditional expectation and variance by $\mathbb{E}_n^{\varepsilon}$
and $\mathbb{V}_n^{\varepsilon}$.

\begin{dfn}
Let $n\in \mathbb{N}_0$, let $\Gamma$ be a subset of $[0,1)$, and let
$\varepsilon = (\varepsilon_K)_{K\in \mathcal{D}}\in \{ \pm 1 \}^{\mathcal{D}}$ be a sequence of signs. Then we define
\begin{equation*}
  r_n^{\Gamma}
  = \sum_{\substack{K\in \mathcal{D}_n\\K\subset \Gamma}} h_K
  \qquad \text{and} \qquad
  r_n^{\Gamma}(\varepsilon)
  = \sum_{\substack{K\in \mathcal{D}_n\\K\subset \Gamma}} \varepsilon_K h_K.
\end{equation*}
\end{dfn}

\begin{rem}\label{rem:randomized-faithful-haar-system}
  Let $(n_I)_{I\in \mathcal{D}}$ be a fixed strictly increasing sequence of non-negative
  integers. Then, given any sequence of signs $\theta\in \{ \pm 1 \}^{\mathcal{D}}$, we can
  construct a faithful Haar system $(\hat{h}_I(\theta))_{I\in \mathcal{D}}$ relative to the
  frequencies $(n_I)_{I\in \mathcal{D}}$ by putting
  \begin{equation*}
    \hat{h}_I(\theta)
    = r_{n_I}^{\Gamma_I(\theta)}(\theta)
    = \sum_{\substack{K\in \mathcal{D}_{n_I}\\K\subset \Gamma_I(\theta)}}\theta_K h_K,
    \qquad I\in \mathcal{D},
  \end{equation*}
  where the sets $\Gamma_I(\theta)$ are defined as
  \begin{equation*}
    \Gamma_{[0,1)}(\theta) = [0,1)
    \qquad\text{and}\qquad
    \Gamma_{I^{\pm}}(\theta) = \{\hat{h}_I(\theta) = \pm 1\},
    \quad I\in \mathcal{D}.
  \end{equation*}
  Note that for every $I\in \mathcal{D}$, the set $\Gamma_I(\theta)$ only depends on the signs
  $(\theta_K : K\in \mathcal{D}_{<n_I})$, whereas $\hat{h}_I(\theta)$ also depends on
  $(\theta_K : K\in \mathcal{D}_{n_I})$.
\end{rem}

The next result is our main probabilistic lemma, which will be essential for proving our
stabilization result. We use the technique from \cite[Lemma~5.3]{MR4430957}---for convenience, we
provide a detailed proof.

\begin{lem}\label{lem:probabilistic}
  Let $Y\in \mathcal{HH}_0(\delta)$. Given a bounded Haar multiplier $D\colon Y\to Y$ and a strictly
  increasing sequence of non-negative integers $(n_I)_{I\in \mathcal{D}}$, we define the random
  variables
  \begin{equation*}
    X_I(\theta)
    = \langle r_{n_I}^{\Gamma_I(\theta)}, Dr_{n_I}^{\Gamma_I(\theta)} \rangle,
    \qquad I\in \mathcal{D},\ \theta\in\{\pm 1\}^{\mathcal{D}},
  \end{equation*}
  where the sets $\Gamma_I(\theta)$ are defined as in \Cref{rem:randomized-faithful-haar-system}
  with respect to $(n_I)_{I\in \mathcal{D}}$.  Then for every
  $\varepsilon\in \{\pm 1\}^{\mathcal{D}}$ and $I\in \mathcal{D}$, we have
  \begin{equation*}
    \mathbb{E}_{n_I}^{\varepsilon} X_{I^{\pm}}
    = \frac{1}{2} \langle r_{n_{I^{\pm}}}^{\Gamma_I(\varepsilon)}, Dr_{n_{I^{\pm}}}^{\Gamma_I(\varepsilon)} \rangle
    \qquad\text{and}\qquad
    \mathbb{V}_{n_I}^{\varepsilon} X_{I^{\pm}}
    \leq \frac{1}{4}2^{-n_I}\|D\|^2 |I|.
  \end{equation*}
\end{lem}

\begin{proof}
  Let $(d_K)_{K\in \mathcal{D}}$ denote the entries of the Haar multiplier $D$, and fix
  $I\in \mathcal{D}$ and $\theta\in \{ \pm 1 \}^{\mathcal{D}}$. Then we have
  \begin{equation*}
    X_{I^{\pm}}(\theta)
    = \sum_{\substack{K\in \mathcal{D}_{n_{I^{\pm}}}\\ K\subset \Gamma_{I^{\pm }}(\theta)}}
    \sum_{\substack{L\in \mathcal{D}_{n_{I^{\pm}}}\\ L\subset \Gamma_{I^{\pm }}(\theta)}}
    \theta_K\theta_L\langle h_K, Dh_L \rangle
    = \sum_{\substack{K\in \mathcal{D}_{I^{\pm}}\\ K\subset \Gamma_{I^{\pm }}(\theta)}} d_K|K|
    = \sum_{\substack{J\in \mathcal{D}_{n_I}\\ J\subset \Gamma_I(\theta)}}
    \sum_{\substack{K\in \mathcal{D}_{n_{I^{\pm}}}\\ K\subset J^{\pm \theta_J}}} d_K|K|.
  \end{equation*}
  Taking the conditional expectation yields
  \begin{equation*}
    \mathbb{E}_{n_I}^{\varepsilon} X_{I^{\pm}}
    = \sum_{\substack{J \in \mathcal{D}_{n_I}\\ J\subset \Gamma_I(\varepsilon)}} \mathbb{E}_{n_I}^{\varepsilon}
    \Bigl( \sum_{\substack{K\in \mathcal{D}_{n_{I^{\pm}}}\\ K\subset J^{\pm \theta_J}}} d_K|K| \Bigr)
    = \frac{1}{2} \sum_{\substack{K\in \mathcal{D}_{n_{I^{\pm}}}\\ K\subset \Gamma_I(\varepsilon)}} d_K|K|
    = \frac{1}{2} \langle r_{n_{I^{\pm}}}^{\Gamma_I(\varepsilon)}, Dr_{n_{I^{\pm}}}^{\Gamma_I(\varepsilon)} \rangle,
  \end{equation*}
  which proves the first identity.

  Next, we calculate the variance. Exploiting the independence of
  $(\theta_J : J\in \mathcal{D}_{n_I},\, J\subset \Gamma_I(\varepsilon))$, we obtain
  \begin{equation*}
    \mathbb{V}_{n_I}^{\varepsilon} X_{I^{\pm}}
    = \sum_{\substack{J\in \mathcal{D}_{n_I}\\ J\subset \Gamma_I(\varepsilon)}} \mathbb{V}_{n_I}^{\varepsilon} \Bigl(
    \sum_{\substack{K\in \mathcal{D}_{n_{I^{\pm}}}\\ K\subset J^{\pm \theta_J}}} d_K|K|
    \Bigr)
    = \sum_{\substack{J\in \mathcal{D}_{n_I}\\ J\subset \Gamma_I(\varepsilon)}}\frac{1}{4} \Bigl(
    \sum_{\substack{K\in \mathcal{D}_{n_{I^{\pm}}}\\ K\subset J^+}} d_K|K| - \sum_{\substack{K\in \mathcal{D}_{n_{I^{\pm}}}\\ K\subset J^-}} d_K|K|
    \Bigr)^2.
  \end{equation*}
  By \Cref{lem:product-of-norms}, we
  have
  \begin{equation*}
    |d_K||K|
    = \langle h_K, Dh_K \rangle
    \le \|h_K\|_{Y^{*}}\|h_K\|_Y\|D\|
    = \|D\| |K|
  \end{equation*}
  for all $K\in \mathcal{D}$. Hence, we obtain
  \begin{equation*}
    \mathbb{V}_{n_I}^{\varepsilon} X_{I^{\pm}}
    \le \sum_{\substack{J\in \mathcal{D}_{n_I}\\ J\subset \Gamma_I(\varepsilon)}} \frac{1}{4}
    \Bigl( \sum_{\substack{K\in \mathcal{D}_{n_{I^{\pm}}}\\ K\subset J}} \|D\||K| \Bigr)^2\\
    \leq \frac{1}{4}\sum_{\substack{J\in \mathcal{D}_{n_I}\\ J\subset \Gamma_I(\varepsilon)}} |J|^2 \|D\|^2
    = \frac{1}{4} 2^{-n_I}\|D\|^2 |\Gamma_I(\varepsilon)|.\qedhere
  \end{equation*}
\end{proof}

\begin{pro}\label{pro:stabilization}
  Let $Y\in \mathcal{HH}_0(\delta)$, let $D\colon Y\to Y$ be a bounded Haar multiplier with entries
  $(d_I)_{I\in \mathcal{D}}$, and suppose that $c\in \Lambda(D)$. Given a sequence of positive real
  numbers $(\eta_I)_{I\in \mathcal{D}}$, there exists another bounded Haar multiplier
  $D^{\mathrm{stab}}\colon Y\to Y$ such that $D^{\mathrm{stab}}$ projectionally factors through $D$
  with constant~$1$ and error~$0$ and such that its entries
  $(d_I^{\mathrm{stab}})_{I\in \mathcal{D}}$ satisfy
  \begin{equation*}
    |d_{[0,1)}^{\mathrm{stab}} - c|
    \le \eta_{[0,1)}
    \qquad \text{and} \qquad
    |d_{I^{\pm}}^{\mathrm{stab}} - d_I^{\mathrm{stab}}|
    \le \eta_I, \qquad I\in \mathcal{D}.
  \end{equation*}
  Moreover, if we additionally assume that there exists a $\delta > 0$ such that $d_I\ge \delta$ for
  all $I\in \mathcal{D}$, then we also have $d_I^{\mathrm{stab}} \ge \delta$ for all
  $I\in \mathcal{D}$.
\end{pro}
\begin{proof}
  In the following, we will first select a strictly increasing sequence of non-negative integers
  $(n_I)_{I\in \mathcal{D}}$ and then construct a faithful Haar system
  $(\hat{h}_I)_{I\in \mathcal{D}}$ relative to the frequencies $(n_I)_{I\in \mathcal{D}}$ by
  choosing a sequence of signs $\varepsilon\in \{ \pm 1 \}^{\mathcal{D}}$ and putting
  $\hat{h}_I = \hat{h}_I(\varepsilon)$, $I\in \mathcal{D}$, as defined in
  \Cref{rem:randomized-faithful-haar-system}.

  Before beginning the construction, we make the following observation. Given a subset
  $\mathcal{A}\subset \mathcal{D}$, let $\mathbb{F}(\mathcal{A})$ denote the countable set of all
  finite unions of intervals in $\mathcal{A}$. For any $\Gamma\in \mathbb{F}(\mathcal{D})$, the
  sequence $(\langle r_n^{\Gamma}, Dr_n^{\Gamma} \rangle)_{n\in \mathbb{N}}$ is bounded because
  $\|r_n^{\Gamma}\|_Y\le 1$ and, by \Cref{cor:norm-in-dual-space},
  $\|r_n^{\Gamma}\|_{Y^{*}}\le 1$. Thus, by Cantor's diagonalization argument, we can find an
  infinite subset $\mathcal{N}\subset \mathbb{N}$ such that for each
  $\Gamma\in \mathbb{F}(\mathcal{D})$, the sequence
  $(\langle r_n^{\Gamma}, Dr_n^{\Gamma} \rangle)_{n\in \mathcal{N}}$ converges to some real number
  $\alpha_{\Gamma}$, and since $c\in \Lambda(D)$, we can ensure that $\alpha_{[0,1)} = c$.

  We now begin our inductive construction. First, we can find $n_{[0,1)}\in \mathcal{N}$ such that
  \begin{equation}\label{eq:ineq-norm-squared-eta-0}
    |\langle r_n, D r_n\rangle - c|
    < \eta_{[0,1)}
    \quad\text{and}\quad
    2^{-n}\|D\|^2
    < \eta_{[0,1)}^2,
    \qquad n\in \mathcal{N},\ n\geq n_{[0,1)}.
  \end{equation}
  This completes the initial step.

  Now let $I\in \mathcal{D}\setminus \{ [0,1) \}$ and suppose that we have already constructed the
  numbers $n_J$, $J < I$, such that $(n_J)_{J < I}$ is strictly increasing. By the definition of
  $\mathcal{N}$, we can find $n_I\in \mathcal{N}$ such that $n_I> n_J$ for all $J < I$ and
  \begin{equation}\label{eq:rnGamma-DrnGamma-alpha}
    |\langle r_n^{\Gamma}, Dr_n^{\Gamma} \rangle - \alpha_{\Gamma}|
    < \eta_I,
    \qquad \Gamma \in \mathbb{F}(\mathcal{D}_{n_{\pi(I)} + 1}),\ n\in \mathcal{N},\ n\geq n_I
  \end{equation}
  as well as
  \begin{equation}\label{eq:ineq-norm-squared-eta-k}
    2^{-n}\|D\|^2
    < \eta_I^2,
    \qquad n\geq n_I.
  \end{equation}

  Next, we construct the sequence of signs $\varepsilon = (\varepsilon_K)_{K\in \mathcal{D}}$. Let
  $I\in \mathcal{D}$ and suppose that we have already chosen the signs
  $(\varepsilon_K : K\in \mathcal{D}_{<n_I})$ (or no signs yet, if $I = [0,1)$). Now let
  $\varepsilon^I\in \{ \pm 1 \}^{\mathcal{D}}$ be any sequence of signs such that
  $\varepsilon^I_K = \varepsilon_K$ for all $K\in \mathcal{D}_{<n_I}$ and consider
  the random variables
  \begin{equation*}
    X_{I^{\pm }}(\theta)
    = \langle r_{n_{I^{\pm}}}^{\Gamma_{I^{\pm }}(\theta)}, Dr_{n_{I^{\pm}}}^{\Gamma_{I^{\pm }}(\theta)} \rangle,
    \qquad \theta\in\{\pm 1\}^{\mathcal{D}}.
  \end{equation*}
  Using \Cref{lem:probabilistic} together with~\eqref{eq:ineq-norm-squared-eta-0}
  and~\eqref{eq:ineq-norm-squared-eta-k} yields
  \begin{equation*}
    \mathbb{E}_{n_I}^{\varepsilon^I} X_{I^{\pm}}
    = \frac{1}{2} \langle r_{n_{I^{\pm}}}^{\Gamma_I(\varepsilon^I)}, Dr_{n_{I^{\pm}}}^{\Gamma_I(\varepsilon^I)} \rangle
    \qquad\text{and}\qquad
    \mathbb{V}_{n_I}^{\varepsilon^I} X_{I^{\pm}}
    < \frac{1}{4}\eta_I^2 |I|.
  \end{equation*}
  By Chebyshev's inequality, we conclude that
  \begin{equation*}
    \mathbb{P}_{n_I}^{\varepsilon^I}(
    |X_{I^+} - \mathbb{E}_{n_I}^{\varepsilon^I} X_{I^+}| \geq \eta_I
    \text{ or }
    |X_{I^-} - \mathbb{E}_{n_I}^{\varepsilon^I} X_{I^-}| \geq \eta_I
    )
  \end{equation*}
  is bounded from above by
  \begin{equation*}
    \frac{1}{\eta_I^2}\mathbb{V}_{n_I}^{\varepsilon^I} X_{I^+}
    + \frac{1}{\eta_I^2}\mathbb{V}_{n_I}^{\varepsilon^I} X_{I^-}
    \le \frac{1}{\eta_I^2}\cdot 2\cdot \frac{1}{4}\eta_I^2 |I|
    \le \frac{1}{2}
    < 1.
  \end{equation*}
  Thus, we can choose signs $(\varepsilon_K : K\in \mathcal{D}_{n_I})$ such that for any
  $\theta\in \{ \pm 1 \}^{\mathcal{D}}$ with $\theta_K = \varepsilon_K$, $K\in \mathcal{D}_{\le n_I}$, we have
  \begin{equation}\label{eq:rnkGamma-difference-one-half-predecessor}
    \Bigl|
    \langle r_{n_{I^{\pm}}}^{\Gamma_{I^{\pm}}(\theta)}, Dr_{n_{I^{\pm}}}^{\Gamma_{I^{\pm}}(\theta)} \rangle
    - \frac{1}{2}\langle r_{n_{I^{\pm}}}^{\Gamma_I(\theta)}, Dr_{n_{I^{\pm}}}^{\Gamma_I(\theta)} \rangle
    \Bigr|
    < \eta_I.
  \end{equation}
  The signs $(\varepsilon_K : K\in \mathcal{D}_{<n_J}\setminus \mathcal{D}_{\le n_I})$, where
  $J = \iota^{-1}(\iota(I) + 1)$, can be chosen arbitrarily: We put $\varepsilon_K = 1$ for
  $K\in \mathcal{D}_{<n_J}\setminus \mathcal{D}_{\le n_I}$. This concludes the construction of the
  signs.

  Now we proceed to analyze the properties of our construction. For $I\in \mathcal{D}$, we write
  $\Gamma_I = \Gamma_I(\varepsilon)$, where $\varepsilon$ denotes the sequence of signs that we have
  just chosen. Observe that we have $n_I, n_{I^{\pm}}\in \mathcal{N}$ and $n_I < n_{I^{\pm}}$
  for all $I\in \mathcal{D}$, and moreover, $\Gamma_I\in \mathbb{F}(\mathcal{D}_{n_{\pi(I)} + 1})$
  for all $I\in \mathcal{D}\setminus \{ [0,1) \}$. Combining~\eqref{eq:ineq-norm-squared-eta-0}
  with~\eqref{eq:rnGamma-DrnGamma-alpha} yields
  \begin{equation*}
    | \langle r_{n_{I^{\pm}}}^{\Gamma_I}, Dr_{n_{I^{\pm}}}^{\Gamma_I} \rangle - \langle r_{n_I}^{\Gamma_I}, Dr_{n_I}^{\Gamma_I} \rangle |
    \le |
    \langle r_{n_{I^{\pm}}}^{\Gamma_I}, Dr_{n_{I^{\pm}}}^{\Gamma_I} \rangle
    - \alpha_{\Gamma_I}| + |\alpha_{\Gamma_I} - \langle r_{n_I}^{\Gamma_I}, Dr_{n_I}^{\Gamma_I} \rangle |
    < 2\eta_I
  \end{equation*}
  for all $I\in \mathcal{D}$. Together with~\eqref{eq:rnkGamma-difference-one-half-predecessor},
  applied to $\theta = \varepsilon$, this implies that
  \begin{equation*}
    \Bigl|
    \langle r_{n_{I^{\pm}}}^{\Gamma_{I^{\pm }}}, Dr_{n_{I^{\pm}}}^{\Gamma_{I^{\pm }}} \rangle
    - \frac{1}{2}\langle r_{n_I}^{\Gamma_I}, Dr_{n_I}^{\Gamma_I} \rangle
    \Bigr|
    < 2\eta_I,
    \qquad I\in \mathcal{D}.
  \end{equation*}
  Now observe that
  \begin{align*}
    \langle r_{n_I}^{\Gamma_I}(\varepsilon), Dr_{n_I}^{\Gamma_I}(\varepsilon) \rangle
    &= \sum_{\substack{K\in \mathcal{D}_{n_I}\\ K\subset \Gamma_I}} \langle h_K, Dh_K \rangle
    = \langle r_{n_I}^{\Gamma_I}, Dr_{n_I}^{\Gamma_I} \rangle,
    \qquad I\in \mathcal{D}.
  \end{align*}
  Thus, the faithful Haar system $(\hat{h}_I)_{I\in \mathcal{D}}$ defined by
  $\hat{h}_I = \hat{h}_I(\varepsilon) = r_{n_I}^{\Gamma_I}(\varepsilon)$, $I\in \mathcal{D}$,
  satisfies
  \begin{align}
    |\langle \hat{h}_{[0,1)}, D \hat{h}_{[0,1)} \rangle - c|
    &< \eta_{[0,1)},\label{eq:27}\\
    \Bigl| \langle \hat{h}_{I^{\pm}}, D \hat{h}_{I^{\pm}} \rangle - \frac{1}{2}\langle \hat{h}_I, D \hat{h}_I \rangle \Bigr|
    &< 2\eta_I,
      \qquad I\in \mathcal{D}.\label{eq:28}
  \end{align}

  Now we use the operators $\hat{A},\hat{B}\colon Y\to Y$, defined as in \Cref{pro:A-B-bounded} with
  respect to our newly constructed faithful Haar system $(\hat{h}_I)_{I\in \mathcal{D}}$, and we
  define $D^{\mathrm{stab}} = \hat{A}D\hat{B}$. We have $\hat{A}\hat{B} = I_Y$ and
  $\|\hat{A}\| = \|\hat{B}\| = 1$. Since $D$ is diagonal with respect to $(h_I)_{I\in \mathcal{D}}$
  and the functions $(\hat{h}_I)_{I\in \mathcal{D}}$ have pairwise disjoint Haar supports, it
  follows that $\langle \hat{h}_I, D \hat{h}_J \rangle = 0$ for all $I\ne J$. Thus, we have
  \begin{equation*}
    D^{\mathrm{stab}}h_I
    = \hat{A}D\hat{B}h_I
    = \sum_{J\in \mathcal{D}} \frac{\langle \hat{h}_J, D \hat{h}_I \rangle}{|J|} h_J
    = \frac{\langle \hat{h}_I, D \hat{h}_I \rangle}{|I|} h_I,
    \quad I\in \mathcal{D},
  \end{equation*}
  and so $D^{\mathrm{stab}}$ is also diagonal with respect to
  $(h_I)_{I\in \mathcal{D}}$. By~\eqref{eq:27} and~\eqref{eq:28}, the entries
  $d_I^{\mathrm{stab}} = \langle \hat{h}_I, D\hat{h}_I \rangle/|I|$, $I\in \mathcal{D}$, satisfy
  \begin{equation*}
    |d_{[0,1)}^{\mathrm{stab}} - c| < \eta_{[0,1)}
    \qquad\text{and}\qquad
    |d_{I^{\pm }}^{\mathrm{stab}} - d_I^{\mathrm{stab}}| < 4\eta_I/|I|, \quad I\in \mathcal{D}.
  \end{equation*}
  The desired conclusion is obtained by replacing $\eta_I$ by $\eta_I|I|/4$ for each
  $I\in \mathcal{D}$ before the construction begins.

  The additional statement follows since the inequalities $\langle h_K, Dh_K \rangle\ge \delta |K|$,
  $K\in \mathcal{D}$, imply that
  \begin{equation*}
    \langle \hat{h}_I, D \hat{h}_I \rangle
    = \sum_{\substack{K\in \mathcal{D}_{n_I}\\ K\subset \Gamma_I}} \langle h_K, Dh_K \rangle
    \ge \delta \sum_{\substack{K\in \mathcal{D}_{n_I}\\ K\subset \Gamma_I}} |K|
    = \delta |I|,\qquad I\in \mathcal{D}.\qedhere
  \end{equation*}
\end{proof}

%%% Local Variables:
%%% TeX-master: "main"
%%% End:

%auto-ignore

This stabilization result enables us to prove the main result \Cref{thm:main-result:3}.

\begin{myproof}[Proof of \Cref{thm:main-result:3}]
  \begin{proofstep-noskip}[Proof of~\eqref{thm:main-result:3:i}]
    Fix $\eta > 0$ and let $(\eta_I)_{I\in \mathcal{D}}$ be a sequence of positive real
    numbers with
    \begin{equation*}
      \sum_{I\in \mathcal{D}} \eta_I \le \frac{\eta}{8}.
    \end{equation*}
    By applying \Cref{pro:stabilization} to $D$ and $(\eta_I)_{I\in \mathcal{D}}$, we obtain a
    bounded Haar multiplier $D^{\mathrm{stab}}\colon Y\to Y$ and operators
    $\hat{A}, \hat{B}\colon Y\to Y$ such that $\hat{A}\hat{B} = I_Y$, $\|\hat{A}\|\|\hat{B}\|\le 1$
    and $D^{\mathrm{stab}} = \hat{A}D\hat{B}$, and such that the entries
    $(d_I^{\mathrm{stab}})_{I\in \mathcal{D}}$ of $D^{\mathrm{stab}}$ satisfy
    \begin{equation}\label{eq:16}
      |d_{[0,1)}^{\mathrm{stab}} - c| \le \frac{\eta}{2}
      \qquad \text{and}\qquad
      |d_I^{\mathrm{stab}} - d_{\pi(I)}^{\mathrm{stab}}| \le \eta_{\pi(I)},\quad I\in \mathcal{D}\setminus \{ [0,1) \}.
    \end{equation}
    We will show that $\|cI_Y - D^{\mathrm{stab}}\| \le \eta$.  Note that
    $d_{[0,1)}^{\mathrm{stab}}I_Y - D^{\mathrm{stab}}$ is a Haar multiplier with entries
    $(d_{[0,1)}^{\mathrm{stab}} - d_I^{\mathrm{stab}})_{I\in \mathcal{D}}$, which satisfy
    \begin{equation*}
      \vvvert (d_{[0,1)}^{\mathrm{stab}} - d_I^{\mathrm{stab}})_{I\in \mathcal{D}}\vvvert
      \le 2 \sum_{I\in \mathcal{D}\setminus\{[0,1)\}} | d_I^{\mathrm{stab}} - d_{\pi(I)}^{\mathrm{stab}}|
      \le 4 \sum_{I\in \mathcal{D}} \eta_I
      \le \frac{\eta}{2}.
    \end{equation*}
    Thus, by \Cref{lem:haar-multiplier}, we have
    $\|d_{[0,1)}^{\mathrm{stab}}I_Y - D^{\mathrm{stab}}\| \le \eta/2$, and hence, using~\eqref{eq:16},
    \begin{equation*}
      \|c I_Y - \hat{A}D\hat{B}\| \le \eta.
    \end{equation*}
  \end{proofstep-noskip}

  \begin{proofstep-noskip}[Proof of~\eqref{thm:main-result:3:i}--\eqref{thm:main-result:3:iii}]
    The remaining statements follow immediately from \eqref{thm:main-result:3:i} combined with
    \Cref{rem:factors}, \Cref{rem:factors:iso} and \Cref{rem:characteristic}.\qedhere
  \end{proofstep-noskip}
\end{myproof}

%%% Local Variables:
%%% TeX-master: "main"
%%% End:
 %auto-ignore

\section{Strategically supporting systems and strategically reproducible bases}
\label{sec:stratsupp-stratrep}

In this section, we provide an overview of the framework of strategically supporting systems and
strategically reproducible Schauder bases introduced in~\cite{MR4477014,MR4145794}. We explain how
this framework can be used to diagonalize operators and to reduce operators with large diagonal to
operators with large \emph{positive} or large \emph{negative} diagonal. Moreover, we extend the
definitions so that they can be utilized in the context of the primary factorization property. We
need the following definitions.

\begin{dfn}
Suppose that $(e_j)_{j=1}^\infty$ is a Schauder basis for a Banach space $E$ and let
          $(e_j^*)_{j=1}^\infty$ denote the biorthogonal functionals.  Then we say that
          $(e_j^{*})_{j=1}^{\infty}$ is a \emph{weak* Schauder basis}.  In this case, for any
          $e^{*}\in E^{*}$, we have the following unique expansion:
          \begin{equation*}
            e^{*} = \text{w*-}\sum_{j=1}^\infty  \langle e^{*}, e_j \rangle e_j^{*},
          \end{equation*}
          where the above series converges in the weak* topology.  From now on, we will always
          indicate weak* convergence as above.
\end{dfn}

\begin{dfn} Let $E$ and $F$ be Banach spaces. Let $(x_j)_{j=1}^{\infty}$, $(y_j)_{j=1}^{\infty}$,
  $(x_j^{*})_{j=1}^{\infty}$ and $(y_j^{*})_{j=1}^{\infty}$ be sequences in $E$, $F$, $E^{*}$ and
  $F^{*}$, respectively, and let $C > 0$.
  \begin{itemize}
    \item We say that $(x_j)_{j=1}^{\infty}$ and $(y_j)_{j=1}^{\infty}$ are \emph{impartially
          $C$-equivalent} if for any sequence of scalars $(a_j)_{j=1}^{\infty}$ with $a_j\ne 0$ for
          at most finitely many~$j$, we have
          \begin{equation*}
            \frac{1}{\sqrt{C}} \Bigl\| \sum_{j=1}^{\infty}a_jy_j \Bigr\|_F
            \le \Bigl\| \sum_{j=1}^{\infty}a_jx_j \Bigr\|_E
            \le \sqrt{C}\Bigl\| \sum_{j=1}^{\infty} a_jy_j \Bigr\|_F.
          \end{equation*}
    \item We say that $(x_j)_{j=1}^{\infty}$ is \emph{$C$-norm-dominated} by $(y_j)_{j=1}^{\infty}$
          if whenever $\sum_{j=1}^{\infty} a_j y_j$ converges in norm for a scalar sequence
          $(a_j)_{j=1}^{\infty}$, then $\sum_{j=1}^{\infty} a_j x_j$ converges in norm and
          \begin{equation*}
            \Bigl\|\sum_{j=1}^{\infty} a_j x_j\Bigr\|_E
            \leq C \Bigl\|\sum_{j=1}^{\infty} a_j y_j\Bigr\|_F.
          \end{equation*}
    \item We say that $(x_j^{*})_{j=1}^{\infty}$ is \emph{$C$-weak*-dominated} by
          $(y_j^{*})_{j=1}^{\infty}$ if whenever $\text{w*-}\sum_{j=1}^{\infty} a_j y_j^{*}$
          converges weak* for a scalar sequence $(a_j)_{j=1}^{\infty}$, then
          $\text{w*-}\sum_{j=1}^{\infty} a_j x_j^{*}$ converges weak* and
          \begin{equation*}
            \Bigl\|\text{w*-}\sum_{j=1}^\infty a_j x_j^{*}\Bigr\|_{E^{*}}
            \le C \Bigl\|\text{w*-}\sum_{j=1}^\infty a_j y_j^{*}\Bigr\|_{F^{*}}.
          \end{equation*}
  \end{itemize}
\end{dfn}

In~\cite{MR4477014}, T.~Kania and the first named author introduced strategically supporting systems
in dual pairs of Banach spaces.  We will now provide the definition for the special case of a Banach
space $E$ and its dual $E^{*}$ on the one hand, while also expanding the concept to accommodate
\emph{projectional factors} on the other hand.
\begin{dfn}\label{dfn:strat-supp}
  Let $(e_j)_{j=1}^\infty$ denote a Schauder basis for the Banach space $E$ and let
  $(e_j^{*})_{j=1}^\infty$ denote the biorthogonal functionals.  We say that
  \emph{$((e_j,e_j^{*}))_{j=1}^\infty$ is $C$-strategically supporting (in $E\times E^{*}$)} if for
  all $\eta > 0$ and all partitions $N_1,N_2$ of $\mathbb{N}$ there exists $i\in\{1,2\}$ and
  \begin{align*}
    \exists&\ \text{finite}\ E_1\subset N_i\;\exists(\lambda_j^1)_j, (\mu_j^1)_j\in\mathbb{R}^{E_1}\;
             \forall (\varepsilon_j^1)_j\in\{\pm 1\}^{E_1}\\
    \exists&\ \text{finite}\   E_2\subset N_i\;\exists(\lambda_j^2)_j, (\mu_j^2)_j\in\mathbb{R}^{E_2}\;
             \forall (\varepsilon_j^2)_j\in\{\pm 1\}^{E_2}\\
           &\vdots\\
    \exists&\ \text{finite}\  E_k\subset N_i\;\exists(\lambda_j^k)_j, (\mu_j^k)_j\in\mathbb{R}^{E_k}\;
             \forall (\varepsilon_j^k)_j\in\{\pm 1\}^{E_k}\\
           &\vdots
  \end{align*}
  such that if we define
  \begin{equation}\label{eq:6}
    x_k
    = \sum_{j\in E_k}\varepsilon_j^k\lambda_j^ke_j
    \quad\text{and}\quad
    x_k^{*}
    = \sum_{j\in E_k}\varepsilon_j^k\mu_j^ke_j^{*},
    \qquad k\in\mathbb{N},
  \end{equation}
  we have that
  \begin{enumerate}[(i)]
    \item\label{enu:dfn:strat-supp:i} $(x_k)_{k=1}^\infty$ is $\sqrt{C+\eta}$-norm-dominated by
          $(e_k)_{k=1}^\infty$;
    \item\label{enu:dfn:strat-supp:ii} $(x_k^{*})_{k=1}^\infty$ is $\sqrt{C+\eta}$-weak*-dominated
          by
          $(e_k^{*})_{k=1}^\infty$,\\
          and $\sum_{k=1}^{\infty} \langle x_k^{*}, x \rangle e_k$ converges in norm for all
          $x\in E$;
    \item\label{enu:dfn:strat-supp:iii}
          $1 \leq \langle x_k^{*}, x_k \rangle = \sum_{j\in E_k} \lambda_j^k \mu_j^k \leq 1 + \eta$,
          $k\in\mathbb{N}$;
    \item\label{enu:dfn:strat-supp:iv} $\lambda_j^k\mu_j^k\geq 0$, $k\in\mathbb{N}$, $j\in E^k$.
  \end{enumerate}
  If additionally $(x_k^{*})_{k=1}^{\infty}$ is biorthogonal to $(x_k)_{k=1}^{\infty}$, then we say
  that $((e_j,e_j^{*}))_{j=1}^\infty$ is $C$-strategically supporting (in $E\times E^{*}$)
  \emph{with projectional factors.}
\end{dfn}

\begin{rem}
  Note that \Cref{dfn:strat-supp} contains a phrase with infinitely many quantifiers. For a formal
  definition of this kind of notation, we refer to~\cite[Section~2]{MR1926866}. Moreover, this
  condition can be understood to mean that one of two player has a winning strategy in an
  \emph{infinite game}.  In fact, we will define \emph{strategically reproducible} Schauder bases in
  terms of such an infinite two-player game in \Cref{dfn:stratrep}, which should be interpreted in
  the same manner (see~\cite{MR1926866,MR0054922,MR0403976}).
\end{rem}

The following proposition states that if a system consisting of a Schauder basis and its
biorthogonal functionals is strategically supporting, then any bounded linear operator with large
diagonal can be reduced to an operator with large positive or large negative diagonal.

\begin{pro}\label{pro:positive-reduction}
  Let $(e_j)_{j=1}^\infty$ denote a Schauder basis for a Banach space $E$, and let
  $(e_j^{*})_{j=1}^\infty$ denote the biorthogonal functionals.  Suppose that
  $((e_j,e_j^{*}))_{j=1}^\infty$ is $C$-strategically supporting in $E\times E^{*}$ (with
  projectional factors) and let $T\colon E\to E$ denote a bounded linear operator which has
  $\delta$-large diagonal with respect to $(e_j)_{j=1}^{\infty}$.  Then the following statements are true:
  \begin{itemize}
    \item For every $\gamma > 0$, there exists a bounded linear operator $S\colon E\to E$ with
          either $\delta$-large positive or $\delta$-large negative diagonal with respect to
          $(e_j)_{j=1}^{\infty}$ such that $S$ (projectionally) factors through $T$ with
          constant~$C + \gamma$ and error~$0$.
    \item If $T$ is diagonal with respect to $(e_j)_{j=1}^{\infty}$, then in the case of
          projectional factors, $S$ can also be chosen to be diagonal with respect to
          $(e_j)_{j=1}^{\infty}$.
  \end{itemize}
\end{pro}

\begin{proof}
  Suppose that for some $\delta > 0$ the bounded linear operator $T\colon Y\to Y$ has $\delta$-large
  diagonal with respect to $(e_j)_{j=1}^{\infty}$, i.e.,
  $|\langle e_j^{*}, Te_j \rangle| \ge \delta$ for all $j\in \mathbb{N}$, and let $\gamma > 0$.
  Define $N_1 = \{ j \in \mathbb{N} : \langle e_j^{*}, Te_j \rangle \ge \delta\}$, and
  $N_2 = \{ j \in \mathbb{N} : \langle e_j^{*}, Te_j \rangle \leq -\delta\}$. Then, since $T$ has
  $\delta$-large diagonal, $N_1$, $N_2$ is a partition of $\mathbb{N}$.

  Exploiting that $((e_j,e_j^{*}))_{j=1}^\infty$ is $C$-strategically supporting, we first obtain an
  index $i\in \{1,2\}$, determining that our systems $(x_k)_{k=1}^{\infty}$,
  $(x_k^{*})_{k=1}^{\infty}$ will be supported in $N_i$.  From now on, we will assume that $i=1$ and
  note that the case $i=2$ is obtained by replacing~$T$ with~$-T$.  Next, we obtain a finite set
  $E_1\subset N_i$ and $(\lambda_j^1)_j, (\mu_j^1)_j\in\mathbb{R}^{E_1}$ and we are now free to pick
  the signs $(\varepsilon_j^1)_j\in\{\pm 1\}^{E_1}$.  To this end, let $(\theta_j^1 : j\in E_1)$ be
  an independent family of random variables taking the values $\pm 1$, each with probability $1/2$,
  and let $\mathbb{E}$ denote the expectation.  Then, since $E_1\subset N_1$,
  \Cref{dfn:strat-supp}~\eqref{enu:dfn:strat-supp:iii} and~\eqref{enu:dfn:strat-supp:iv} yield
  \begin{equation*}
    \mathbb{E} \Bigl\langle \sum_{j\in E_1} \theta_j^1 \mu_j^1 e_j^{*}, T\Bigl( \sum_{j\in E_1} \theta_j^1 \lambda_j^1 e_j \Bigr) \Bigr\rangle
    = \sum_{j\in E_1} \mu_j^1\lambda_j^1 \langle e_j^{*}, T e_j \rangle
    \geq \delta \sum_{j\in E_1} \mu_j^1\lambda_j^1
    \geq \delta,
  \end{equation*}
  and hence, we can find signs $(\varepsilon_j^1)_j\in\{\pm 1\}^{E_1}$ such that
  \begin{equation*}
    x_1
    = \sum_{j\in E_1}\varepsilon_j^1\lambda_j^1e_j
    \quad\text{and}\quad
    x_1^{*}
    = \sum_{j\in E_1}\varepsilon_j^1\mu_j^1e_j^{*}
    \qquad\text{satisfy}\qquad
    \langle x_1^{*}, x_1 \rangle
    \geq \delta.
  \end{equation*}
  Proceeding in this manner, we obtain sequences $(x_k)_{k=1}^{\infty}$ and
  $(x_k^{*})_{k=1}^{\infty}$, defined according to~\eqref{eq:6}, which satisfy
  \Cref{dfn:strat-supp}~\eqref{enu:dfn:strat-supp:i}--\eqref{enu:dfn:strat-supp:iv} as well as the
  additional condition
  \begin{equation}\label{eq:24}
    \langle x_k^{*}, T x_k \rangle
    \geq \delta,
    \qquad k\in \mathbb{N}.
  \end{equation}
  If we assume that $((e_j,e_j^{*}))_{j=1}^\infty$ is $C$-strategically supporting in
  $E\times E^{*}$ \emph{with projectional factors}, then we get additionally that
  $(x_k^{*})_{k=1}^{\infty}$ is biorthogonal to $(x_k)_{k=1}^{\infty}$.

  We will now define the canonical operators $A,B\colon E\to E$ given by
  \begin{equation}\label{eq:26}
    Bx
    = \sum_{k=1}^{\infty} \langle e_k^{*}, x \rangle x_k
    \qquad\text{and}\qquad
    Ax
    = \sum_{k=1}^{\infty} \langle x_k^{*}, x \rangle e_k,
    \qquad x\in E,
  \end{equation}
  and note that they are well-defined since the above series converge in norm.  Exploiting that
  $(e_k^{*})_{k=1}^{\infty}$ is a weak* Schauder basis, together with~\eqref{enu:dfn:strat-supp:ii},
  we have for all $x^{*}\in E^{*}$ and $x\in E$:
  \begin{align*}
    |\langle x^{*}, Ax \rangle|
    &= \Bigl|\sum_{k=1}^{\infty} \langle x_k^{*}, x \rangle \langle x^{*}, e_k \rangle\Bigr|
      = \Bigl|\Bigl\langle \text{w*-}\sum_{k=1}^{\infty} \langle x^{*}, e_k \rangle x_k^{*}, x \Bigr\rangle\Bigr|\\
    &\leq \Bigl\|\text{w*-}\sum_{k=1}^{\infty} \langle x^{*}, e_k \rangle x_k^{*}\Bigr\|_{E^{*}} \|x\|_E
      \leq \sqrt{C+\gamma} \|x^{*}\|_{E^{*}} \|x\|_E.
  \end{align*}
  The above estimate together with~\eqref{enu:dfn:strat-supp:i} yields
  \begin{equation}\label{eq:25}
    \|A\|, \|B\|
    \leq \sqrt{C+\gamma}.
  \end{equation}

  We now define the bounded linear operator $S\colon E\to E$ by putting $S = ATB$.  Firstly, note
  that by~\eqref{eq:25}, $S$ factors through $T$ with constant $C+\gamma$ and error $0$.  Secondly,
  using~\eqref{eq:26} and~\eqref{eq:24}, we also obtain that $S$ has $\delta$-large positive
  diagonal with respect to $(e_j)_{j=1}^{\infty}$:
  \begin{equation}\label{eq:11}
    \langle e_j^{*}, S e_j \rangle
    = \langle e_j^{*}, ATx_j \rangle
    = \langle x_j^{*}, Tx_j \rangle
    \geq \delta,
    \qquad j\in \mathbb{N}.
  \end{equation}
  If we assume that $((e_j,e_j^{*}))_{j=1}^\infty$ is $C$-strategically supporting in
  $E\times E^{*}$ \emph{with projectional factors}, then $(x_k)_{k=1}^{\infty}$ and
  $(x_k^{*})_{k=1}^{\infty}$ are biorthogonal. Hence, we obtain that $AB = I_E$, i.e., $S$
  projectionally factors through $T$ with constant $C+\gamma$ and error $0$. Moreover, in this case,
  if $T$ is diagonal with respect to $(e_j)_{j=1}^{\infty}$, then~\eqref{eq:11} implies that $S$ is
  also diagonal with respect to $(e_j)_{j=1}^{\infty}$.
\end{proof}

The notion of a \emph{strategically reproducible} Schauder basis was introduced
in~\cite{MR4145794}. Before presenting the definition, we establish the following terminology.

\begin{dfn}
  Let $E$ be a Banach space.  By $\operatorname{cof}(E)$, we denote the set of cofinite-dimensional
  closed subspaces of $E$, while $\operatorname{cof}_{w^{*}}(E^{*})$ denotes the set of
  cofinite-dimensional weak* closed subspaces of $E^{*}$.
\end{dfn}
\begin{rem}\label{rem:cof}
  In fact, we will always use the following characterization of $\operatorname{cof}(E)$ and
  $\operatorname{cof}_{w^{*}}(E^{*})$: A subspace $W\subset E$ is in $\operatorname{cof}(E)$ if and
  only if there are there are $N\in \mathbb{N}_0$ and $x_1^{*},\dots,x_N^{*}\in E^{*}$ such that
  \begin{equation*}
    W
    = \{ x_1^{*}, \dots, x_N^{*} \}_{\perp}.
  \end{equation*}
  Similarly, a subspace $G\subset E^{*}$ is in $\operatorname{cof}_{w^{*}}(E^{*})$ if and only if
  there are $M\in \mathbb{N}_0$ and $x_1,\dots, x_M\in E$ such that
  \begin{equation*}
    G = \{ x_1,\dots,x_M \}^{\perp}.
  \end{equation*}
  Details can be found in \cite[Lemma~5.1.7]{speckhofer:2022}.
\end{rem}

\begin{dfn}\label{dfn:stratrep}
  Let $E$ be a Banach space with a normalized Schauder basis $(e_j)_{j=1}^{\infty}$ and associated
  biorthogonal functionals $(e_j^{*})_{j=1}^{\infty}$, and fix positive constants $C\ge 1$ and
  $\eta > 0$.

  Consider the following two-player game between player~(I) and player~(II):

  Before the first turn player~(I) is allowed to choose a partition of $\mathbb{N} = N_1\cup N_2$.
  For $k\in \mathbb{N}$, turn~$k$ is played out in three steps.

  \begin{enumerate}[Step 1:]
    \item Player~(I) chooses $\eta_k > 0$, $W_k\in \operatorname{cof}(E)$, and
          $G_k\in \operatorname{cof}_{w^{*}}(E^{*})$.
    \item Player~(II) chooses $i_k\in \{ 1,2 \}$, a finite subset $E_k$ of $N_{i_k}$ and sequences
          of non-negative real numbers $(\lambda_j^k)_{j\in E_k}$, $(\mu_j^k)_{j\in E_k}$ satisfying
          \begin{equation*}
            1 - \eta < \sum_{j\in E_k} \lambda_j^k\mu_j^k < 1 + \eta.
          \end{equation*}
    \item Player~(I) chooses $(\varepsilon_j^k)_{j\in E_k}$ in $\{ \pm 1 \}^{E_k}$.
  \end{enumerate}

  We say that player~(II) has a winning strategy in the game
  $\operatorname{Rep}_{(E, (e_j))}(C, \eta)$ if he can force the following properties on the result:

  For all $k\in \mathbb{N}$, we set
  \begin{equation}\label{eq:18}
    x_k = \sum_{j\in E_k}\varepsilon_j^k\lambda_j^ke_j
    \qquad \text{and} \qquad
    x_k^{*} = \sum_{j\in E_k}\varepsilon_j^k\mu_j^ke_j^{*}
  \end{equation}
  and demand:
  \begin{enumerate}[(i)]
    \item\label{enu:dfn:stratrep:i} the sequences $(x_k)_{k=1}^{\infty}$ and $(e_k)_{k=1}^{\infty}$
          are impartially $(C + \eta)$-equivalent,
    \item\label{enu:dfn:stratrep:ii} the sequences $(x_k^{*})_{k=1}^{\infty}$ and
          $(e_k^{*})_{k=1}^{\infty}$ are impartially $(C + \eta)$-equivalent,
    \item\label{enu:dfn:stratrep:iii} for all $k\in \mathbb{N}$, we have
          $\operatorname{dist}(x_k, W_k) < \eta_k$, and
    \item\label{enu:dfn:stratrep:iv} for all $k\in \mathbb{N}$, we have
          $\operatorname{dist}(x_k^{*}, G_k) < \eta_k$.
  \end{enumerate}
  We say that $(e_j)_{j=1}^{\infty}$ is \emph{$C$-strategically reproducible in $E$} if for every
  $\eta > 0$, player~(II) has a winning strategy in the game
  $\operatorname{Rep}_{(E, (e_j))}(C, \eta)$.  If player~(II) can additionally guarantee that
  $(x_k^{*})_{k=1}^{\infty}$ is biorthogonal to $(x_k)_{k=1}^{\infty}$, then we say that
  $(e_j)_{j=1}^{\infty}$ is $C$-strategically reproducible in $E$ \emph{with projectional factors}.
\end{dfn}

Before proceeding, we make the following observations.
\begin{rem}\label{rem:stratrep}
  Let $(e_j)_{j=1}^{\infty}$ denote a Schauder basis of a Banach space $E$ with basis constant
  $\lambda\geq 1$.  In the proof of \cite[Theorem~3.12]{MR4145794}, it is noted that if
  $(e_j)_{j=1}^{\infty}$ is $C$-strategically reproducible in $E$, then player~(I) can always force
  the following outcome on the winning strategy of player~(II): For any $\eta > 0$, in addition
  to~\eqref{enu:dfn:stratrep:i}--\eqref{enu:dfn:stratrep:iv} in \Cref{dfn:stratrep}, the sequence
  $(x_k^{*})_{k=1}^{\infty}$ is summably close to some block sequence
  $(\tilde{x}_k^{*})_{k=1}^{\infty}$ of $(e_j^{*})_{j=1}^{\infty}$, i.e.,
  \begin{equation}\label{eq:15}
    \sum_{k=1}^{\infty} \|x_k^{*} - \tilde{x}_k^{*}\|_{X^{*}}
    < \infty.
  \end{equation}
  In this case, the operator $A\colon E\to E$ given by
  \begin{equation*}
    Ax
    = \sum_{k=1}^{\infty}  \langle x_k^{*}, x \rangle e_k,
    \qquad x\in E
  \end{equation*}
  is well-defined and satisfies
  \begin{equation*}
    \|A\| \leq \lambda\sqrt{C + \eta}.
  \end{equation*}
  For more details, we refer to Lemma~5.2.1 and the proof of Theorem~5.2.3
  in~\cite{speckhofer:2022}.

  On the other hand, observe that even without~\eqref{eq:15}, condition \eqref{enu:dfn:stratrep:i}
  implies that the operator $B\colon E\to E$ given by
  \begin{equation*}
    Bx
    = \sum_{k=1}^{\infty} \langle e_k^{*}, x \rangle x_k,
    \qquad x\in E
  \end{equation*}
  is well-defined and satisfies
  \begin{equation*}
    \|B\| \leq \sqrt{C + \eta}.
  \end{equation*}

  Finally, note that if $(x_k)_{k=1}^{\infty}$ and $(x_k^{*})_{k=1}^{\infty}$ were constructed
  according to a winning strategy with projectional factors, i.e., $(x_k^{*})_{k=1}^{\infty}$ is
  biorthogonal to $(x_k)_{k=1}^{\infty}$, then we additionally have $AB = I_E$.
\end{rem}

The next proposition states that if a Schauder basis of a Banach space is strategically
reproducible, then any bounded linear operator can be diagonalized with respect to the given basis,
and moreover, it is possible to preserve a large (positive) diagonal when diagonalizing. This result
is is implicitly contained in the proof of~\cite[Theorem~3.12]{MR4145794}. Nonetheless, we state
this diagonalzation step as a separate result since we will need it in the proof of
\Cref{thm:stratrep-primary-fact-prop}.

\begin{pro}\label{pro:stratrep-diag}
  Let $E$ be a Banach space with a normalized Schauder basis $(e_j)_{j=1}^{\infty}$ with basis
  constant $\lambda$.  Let $C\geq 1$ and suppose that $(e_j)_{j=1}^{\infty}$ is $C$-strategically
  reproducible in $E$ (with projectional factors).  Moreover, let $T\colon E\to E$ denote a bounded
  linear operator and let $\eta > 0$.  Then the following assertions are true:
  \begin{enumerate}[(i)]
    \item\label{pro:stratrep-diag:i} There exists a bounded linear operator $D\colon E\to E$ which
          is diagonal with respect to $(e_j)_{j=1}^{\infty}$ such that $D$ (projectionally) factors
          through $T$ with constant $\lambda(C + \eta)$ and error $\eta$.
    \item\label{pro:stratrep-diag:ii} If we additionally assume that $T$ has $\delta$-large
          (positive) diagonal with respect to $(e_j)_{j=1}^{\infty}$ for some $\delta > 0$, then the
          above diagonal operator $D$ can be chosen so that $D$ also has $\delta$-large (positive)
          diagonal with respect to $(e_j)_{j=1}^{\infty}$.
  \end{enumerate}
\end{pro}
\begin{proof}
  \Cref{pro:stratrep-diag}~\eqref{pro:stratrep-diag:i} follows from the proof of
  \cite[Theorem~3.12]{MR4145794} together with the discussion in \Cref{rem:stratrep}.  More
  precisely, at the beginning of the proof of \cite[Theorem~3.12]{MR4145794}, we replace
  player~(I)'s choice of $N_1$ and $N_2$ by, say, $N_1 = \mathbb{N}$ and $N_2 = \emptyset$, since
  in~\eqref{pro:stratrep-diag:i} we do not assume that $T$ has large diagonal. Moreover, in the last
  step of the $n$th turn, instead of using a probabilistic argument to pick the signs
  $(\varepsilon^{(n)}_i)_{i\in E_n}$, player~(I) simply chooses $\varepsilon^{(n)}_i = 1$ for all
  $i\in E_n$. Then we follow the proof of \cite[Theorem~3.12]{MR4145794} until the conditions of
  \cite[Lemma~3.14]{MR4145794} are checked, and we conclude by \cite[Lemma~3.14]{MR4145794} and
  \Cref{rem:stratrep}.

  In order to prove \Cref{pro:stratrep-diag}~\eqref{pro:stratrep-diag:ii}, we repeat the argument
  from \cite[Theorem~3.12]{MR4145794}, and thereby we directly obtain a diagonal operator $D$ which
  has $(1-\eta)\delta$-large (positive) diagonal.  By rescaling with the factor $(1-\eta)^{-1}$, we
  obtain the result as claimed in~\eqref{pro:stratrep-diag:ii}.
\end{proof}

\begin{rem}\label{rem:stratrep-lp-sum}
  Fix $\lambda, C_r \ge 1$ and suppose that for every $k\in \mathbb{N}$, $E_k$ is a Banach space
  with a normalized Schauder basis $(e_{k,j})_{j=1}^{\infty}$ which is $C_r$-strategically
  reproducible in $E_k$ and whose basis constant is bounded by
  $\lambda$. Then by \cite[Proposition~7.5]{MR4145794}, the sequence $(e_{k,j})_{j=1}^{\infty}$,
  enumerated as $(\tilde{e}_m)_{m=1}^{\infty}$ according to \Cref{rem:enumeration-of-basis}, is a
  $C_r$-strategically reproducible Schauder basis of $\ell^p((E_k)_{k=1}^{\infty})$ whose basis
  constant is bounded by $\lambda$.  By inspecting the proof of \cite[Proposition~7.5]{MR4145794},
  it is clear that $(\tilde{e}_m)_{m=1}^{\infty}$ is even $C_r$-strategically reproducible
  \emph{with projectional factors} if the same is true for all bases $(e_{k,j})_{j=1}^{\infty}$,
  $k\in \mathbb{N}$. Thus, \Cref{pro:stratrep-diag} can be used to diagonalize any bounded linear
  operator $T\colon \ell^p((E_k)_{k=1}^{\infty})\to \ell^p((E_k)_{k=1}^{\infty})$.

  For $p = \infty$, however, we cannot use \Cref{pro:stratrep-diag} to diagonalize an operator on
  $\ell^{\infty}((E_k)_{k=1}^{\infty})$. Instead, we have to extract an analogous result from the
  proof of \cite[Theorem~3.9]{MR4299595}. First, we give the definition of a diagonal operator on
  $\ell^{\infty}((E_k)_{k=1}^{\infty})$.
\end{rem}

\begin{dfn}
  For each $k\in \mathbb{N}$, let $E_k$ be a Banach space with a Schauder basis
  $(e_{k,j})_{j=1}^{\infty}$. Denote $Z = \ell^{\infty}((E_k)_{k=1}^{\infty})$. A bounded linear
  operator $D\colon Z\to Z$ is called \emph{diagonal (with respect to $(e_{k,j})_{k,j=1}^{\infty}$)}
  if it is of the form $D(x_k)_{k=1}^{\infty} = (D_kx_k)_{k=1}^{\infty}$, $(x_k)_{k=1}^{\infty}\in Z$,
  where for each $k\in \mathbb{N}$, $D_k\colon E_k\to E_k$ is a bounded diagonal operator with
  respect to $(e_{k,j})_{j=1}^{\infty}$.
\end{dfn}

\begin{pro}\label{pro:stratrep-diag-ell-infty}
  For each $k\in \mathbb{N}$, let $E_k$ be a Banach space with a normalized Schauder basis
  $(e_{k,j})_{j=1}^{\infty}$, and denote $Z = \ell^{\infty}((E_k)_{k=1}^{\infty})$. Assume that
  $(E_k)_{k=1}^{\infty}$ is uniformly asymptotically curved with respect to
  $(e_{k,j})_{j,k=1}^{\infty}$ and that there are $C_r,\lambda\ge 1$ such that
  \begin{itemize}
    \item for every $k\in \mathbb{N}$, the basis constant of $(e_{k,j})_{j=1}^{\infty}$ in $E_k$ is
          at most $\lambda$;
    \item for every $k\in \mathbb{N}$, the basis $(e_{k,j})_{j=1}^{\infty}$ is $C_r$-strategically
          reproducible in $E_k$ (with projectional factors).
  \end{itemize}
  Then for every bounded linear operator $T\colon Z\to Z$ and every $\eta > 0$, there exists a
  bounded linear operator $D\colon Z\to Z$ that is diagonal with respect to
  $(e_{k,j})_{k,j=1}^{\infty}$ such that $D$ (projectionally) factors through $T$ with constant
  $\lambda(C_r + \eta)$ and error $\eta$.
\end{pro}

\begin{proof}
  Similarly to \Cref{pro:stratrep-diag}, this result is contained in the proof of
  \cite[Theorem~3.9]{MR4299595}. At the beginning we again replace the choice of $N_1$ and $N_2$ by
  $N_1 = \mathbb{N}$ and $N_2 = \emptyset$, and in \emph{Turn~$n$, Step~3}, player~(I) chooses the
  signs $\varepsilon_i^{(n)} = 1$, $i\in E_n$. In the end, from
  \cite[Proposition~4.5~(b)]{MR4299595}, we obtain a bounded linear operator $D\colon Z\to Z$ that
  is diagonal with respect to $(e_{k,j})_{k,j=1}^{\infty}$ and bounded linear operators
  $A,B\colon Z\to Z$ such that
  \begin{equation*}
    \|D - ATB\| \le 2\lambda \sqrt{C_r + \eta}(3 + \|T\|)\eta
  \end{equation*}
  and
  \begin{equation*}
    \|A\|\le \lambda\sqrt{C_r + \eta}, \qquad \|B\|\le \sqrt{C_r + \eta}
  \end{equation*}
  (see \cite[Remark~4.4]{MR4299595}). Moreover, for all $(z_k)_{k=1}^{\infty}\in Z$, we have
  $A(z_k)_{k=1}^{\infty} = (A_kz_k)_{k=1}^{\infty}$ and
  $B(z_k)_{k=1}^{\infty} = (B_kz_k)_{k=1}^{\infty}$, where for each $k\in \mathbb{N}$,
  $A_k,B_k\colon E_k\to E_k$ are bounded linear operators of the same form as $A,B\colon E\to E$ in
  \Cref{rem:stratrep}. In particular, if all bases are $C_r$-strategically reproducible \emph{with
    projectional factors}, then we have $A_kB_k = I_{E_k}$ for all $k\in \mathbb{N}$ and thus
  $AB = I_Z$.
\end{proof}

%%% Local Variables:
%%% mode: latex
%%% TeX-master: "main"
%%% End:
 %auto-ignore

\section{Strategic properties of the Haar system}

Next, we are going to prove that for every Haar system Hardy space $Y\in \mathcal{HH}_0(\delta)$, the
system $((h_I/\|h_I\|_Y, h_I^{*}))_{I\in \mathcal{D}}$ is $2$-strategically supporting in
$Y\times Y^{*}$, and under the assumption that $(r_n)_{n=0}^{\infty}$ is weakly null, the
normalized Haar basis of $Y$ is $2$-strategically reproducible. The following lemma will be used in
the proofs of both properties.

\begin{lem}\label{lem:impartially-equivalent}
  Let $Y\in \mathcal{HH}_0(\delta)$ and let $\eta > 0$.  Then there exist $\sigma_0, \rho_0 > 0$ such that for all
  $0<\sigma\le \sigma_0$ and $0<\rho\le \rho_0$, the following holds:

  Let $(\tilde{h}_I)_{I\in \mathcal{D}}$ be a $(\varkappa_I)_{I\in \mathcal{D}}$-faithful Haar system for some family
  $(\varkappa_I)_{I\in \mathcal{D}}$ of real numbers in $(0,1]$ that satisfies
  \begin{equation*}
    \sum_{I\in \mathcal{D}} (1 - \varkappa_I)
    \le \sigma.
  \end{equation*}
  Moreover, for every $I\in \mathcal{D}$, let $\mathcal{B}_I\subset \mathcal{D}$ denote the Haar support of
  $\tilde{h}_I$, and suppose that $|\mathcal{B}_{[0,1)}^{*}| \ge \frac{1}{2} - \rho$.  Put
  \begin{equation}\label{eq:17}
    x_I
    = \sqrt{2}\frac{\tilde{h}_I}{\|h_I\|_Y}
    \qquad \text{and}\qquad
    x_I^{*}
    = \frac{1}{\sqrt{2}}\frac{|I|}{|\mathcal{B}_I^{*}|}\cdot \frac{\tilde{h}_I}{\|h_I\|_{Y^{*}}}
  \end{equation}
  for every $I\in \mathcal{D}$.  Then the following assertions are true:
  \begin{enumerate}[(i)]
    \item\label{enu:lem:impartially-equivalent:i} The sequence $(x_I)_{I\in \mathcal{D}}$ is
    impartially $(2 + \eta)$-equivalent to $(h_I/\|h_I\|_Y)_{I\in \mathcal{D}}$ in $Y$;\\
          in particular, $(x_I)_{I\in \mathcal{D}}$ is $\sqrt{2+\eta}$-norm-dominated by
          $(h_I/\|h_I\|_Y)_{I\in \mathcal{D}}$.
    \item\label{enu:lem:impartially-equivalent:ii} The sequence $(x_I^{*})_{I\in \mathcal{D}}$ is
          impartially $(2 + \eta)$-equivalent to $(h_I/\|h_I\|_{Y^{*}})_{I\in \mathcal{D}}$ in
          $Y^{*}$.
    \item\label{enu:lem:impartially-equivalent:iii} The sequence $(x_I^{*})_{I\in \mathcal{D}}$ is
    $\sqrt{2 + \eta}$-weak*-dominated by $(h_I/\|h_I\|_{Y^{*}})_{I\in \mathcal{D}}$,\\
          and for every $x\in Y$, the series
          $\sum_{I\in \mathcal{D}} \langle x_I^{*}, x \rangle h_I/\|h_I\|_Y$ converges in norm.
  \end{enumerate}
\end{lem}
\begin{proof}
  If $\sigma < 1$, then by \Cref{thm:projection-1d}, the operators $A,B\colon Y\to Y$ defined by
  \begin{equation*}
    B x
    = \sum_{I\in \mathcal{D}} \frac{\langle h_I, x\rangle}{\|h_I\|_2^2} \tilde{h}_I
    \qquad\text{and}\qquad
    A x
    = \sum_{I\in \mathcal{D}} \frac{\langle \tilde{h}_I, x\rangle}{\|\tilde{h}_I\|_2^2} h_I
  \end{equation*}
  satisfy
  \begin{equation*}
    \begin{aligned}
      \|B\| = 1
      \qquad \text{and}\qquad
      \|A\| \leq \frac{1}{\mu}\cdot \frac{1 + 3\sigma}{1 - \sigma},
    \end{aligned}
  \end{equation*}
  where $\mu = |\mathcal{B}_{[0,1)}^{*}| \ge \frac{1}{2} - \rho$.  In particular, by choosing
  $\sigma_0$ and $\rho_0$ small enough, we can ensure that
  \begin{equation}\label{eq:19}
    \|B\|
    = 1
    \qquad\text{and}\qquad
    \|A\|
    \le \sqrt{2(2+\eta)}.
  \end{equation}
  Moreover, we know that $AB = I_Y$ and hence $B^{*}A^{*} = I_{Y^{*}}$, as well as
  \begin{equation}\label{eq:20}
    B h_I
    = \tilde{h}_I,
    \quad
    A \tilde{h}_I
    = h_I
    \qquad\text{and}\qquad
    B^{*} \tilde{h}_I
    = \frac{|\mathcal{B}_I^{*}|}{|I|} h_I,
    \quad
    A^{*} h_I
    = \frac{|I|}{|\mathcal{B}_I^{*}|} \tilde{h}_I
  \end{equation}
  for all $I\in \mathcal{D}$. Combining our estimates~\eqref{eq:19} with the identities
  in~\eqref{eq:17} and~\eqref{eq:20} yields~\eqref{enu:lem:impartially-equivalent:i}
  and~\eqref{enu:lem:impartially-equivalent:ii}.

  Next, if $(a_I)_{I\in \mathcal{D}}$ is a scalar sequence and $x^{*}\in Y^{*}$ is such that
  \begin{equation}\label{eq:21}
    x^{*}
    = \text{w*-}\sum_{I\in \mathcal{D}} a_I \frac{h_I}{\|h_I\|_{Y^{*}}},
  \end{equation}
  then, since $A^{*}\colon Y^{*}\to Y^{*}$ is weak*-to-weak* continuous, it follows that
  \begin{equation}\label{eq:22}
    A^{*}x^{*}
    = \sqrt{2}\cdot \text{w*-}\sum_{I\in \mathcal{D}} a_I x_I^{*}
  \end{equation}
  is weak* convergent, and thus, combining~\eqref{eq:21} and~\eqref{eq:22} with~\eqref{eq:19} yields
  \begin{equation*}
    \sqrt{2}\cdot \Bigl\|\text{w*-}\sum_{I\in \mathcal{D}} a_I x_I^{*}\Bigr\|_{Y^{*}}
    \le \|A^{*}\|\|x^{*}\|_{Y^{*}}
    = \|A\| \Bigl\|\text{w*-}\sum_{I\in \mathcal{D}} a_I \frac{h_I}{\|h_I\|_{Y^{*}}}\Bigr\|_{Y^{*}}.
  \end{equation*}

  Finally, for every $x\in Y$, using~\eqref{eq:17} and~\eqref{eq:20}, we see that the series
  \begin{equation*}
    \sum_{I\in \mathcal{D}} \langle x_I^{*}, x \rangle \frac{h_I}{\|h_I\|_Y}
    = \frac{1}{\sqrt{2}}\sum_{I\in \mathcal{D}}\Bigl\langle \frac{h_I}{\|h_I\|_{Y^{*}}}, Ax \Bigr\rangle \frac{h_I}{\|h_I\|_Y}
    = \frac{1}{\sqrt{2}}Ax
  \end{equation*}
  converges in norm.
\end{proof}

In the following, we will always identify the index $k\in \mathbb{N}$ from \Cref{dfn:strat-supp}
with a dyadic interval $I$ via the bijection $\iota|_{\mathcal{D}}\colon \mathcal{D}\to \mathbb{N}$.  Thus, the
partition $\mathbb{N} = N_1\cup N_2$ will correspond to a partition
$\mathcal{D} = \mathcal{A}_1\cup \mathcal{A}_2$.  Moreover, instead of finite subsets $E_k$ of $N_1$
or $N_2$, we will construct finite subsets $\mathcal{B}_I$ of $\mathcal{A}_1$ or $\mathcal{A}_2$,
and the real numbers $\lambda_j^k$, $\mu_j^k$ and signs $\varepsilon_j^k$, $j\in E_k$, will be
denoted as $\lambda^I_L$, $\mu^I_L$ and $\varepsilon^I_L$, $L\in \mathcal{B}_I$.
\begin{thm}\label{thm:strat-supp}
  Let $(h_I/\|h_I\|_{Y})_{I\in \mathcal{D}}$ denote the normalized Haar basis of
  $Y\in \mathcal{HH}_0(\delta)$, and let $(h_I^{*})_{I\in\mathcal{D}}$ denote the biorthogonal
  functionals.  Then the system $((h_I/\|h_I\|_Y,h_I^{*}))_{I\in \mathcal{D}}$ is $2$-strategically
  supporting in $Y\times Y^{*}$ with projectional factors.
\end{thm}

\begin{proof}
  Let $\eta > 0$ and let $\mathcal{A}_1$, $\mathcal{A}_2$ be a partition of $\mathcal{D}$.  Thus,
  $\limsup(\mathcal{A}_1)\cup \limsup(\mathcal{A}_2) = [0,1)$, and hence, we have either
  \begin{equation*}
    |\limsup(\mathcal{A}_1)|
    \ge \frac{1}{2}
    \qquad\text{or}\qquad
    |\limsup(\mathcal{A}_2)|
    \ge \frac{1}{2}.
  \end{equation*}
  If the former inequality is true, we define $i=1$, and we put $i=2$ otherwise.  In any case, we
  obtain
  \begin{equation*}
    |\limsup(\mathcal{A}_i)|
    \ge \frac{1}{2}.
  \end{equation*}
  By~\cite[Lemma~4.4]{MR4145794}, we can find a subset $\mathcal{A}\subset \mathcal{A}_i$ such that
  \begin{itemize}
    \item $\mathcal{G}_n(\mathcal{A})$ is finite and
          $\mathcal{G}_n(\mathcal{A})\subset \mathcal{G}_n(\mathcal{A}_i)$ for all
          $n\in \mathbb{N}_0$,
    \item $|\limsup(\mathcal{A})| \ge \frac{1}{2} - \rho$,
  \end{itemize}
  where $\rho > 0$ is a small number, to be determined later.  We will write
  $S = \limsup(\mathcal{A})$.  Moreover, fix $0 < \sigma < 1$ to be determined later, and let
  $(\varkappa_I)_{I\in \mathcal{D}}$ be a sequence of real numbers in $(0,1)$ such that
  \begin{equation*}
    \sum_{I\in \mathcal{D}} (1 - \varkappa_I) \le \sigma.
  \end{equation*}
  Define $\varkappa_I' = 1 - \varkappa_I$ for all $I\in \mathcal{D}$.

  Now let $I\in \mathcal{D}$ and assume that we have already picked strictly increasing frequencies
  $(n_J)_{J < I}$ and chosen collections $\mathcal{B}_J\subset \mathcal{G}_{n_J}(\mathcal{A})$ and
  positive real numbers
  $(\lambda_L^J)_{L\in \mathcal{B}_J}, (\mu_L^J)_{L\in \mathcal{B}_J}\in \mathbb{R}^{\mathcal{B}_J}$
  for all $J < I$.  Moreover, suppose that the sequences of signs
  $(\varepsilon_L^J)_{L\in \mathcal{B}_J}\in\{\pm 1\}^{\mathcal{B}_{J}}$, $J < I$, are given, thus
  determining the newly constructed systems $(x_J)_{J < I}$ and $(x_J^{*})_{J < I}$ in~\eqref{eq:6},
  i.e.,
  \begin{equation*}
    x_J
    = \sum_{L\in \mathcal{B}_J} \varepsilon^J_L \lambda^J_L \frac{h_L}{\|h_L\|_{Y}}
    \qquad\text{and}\qquad
    x_J^{*}
    = \sum_{L\in \mathcal{B}_J} \varepsilon^J_L \mu^J_L h_L^{*}.
  \end{equation*}
  In addition, we define the auxiliary $L^{\infty}$-normalized Haar system $(\tilde{h}_J)_{J < I}$
  by putting
  \begin{equation*}
    \tilde{h}_J
    = \sum_{L\in \mathcal{B}_J} \varepsilon^J_L h_L,\qquad J\in \mathcal{D},\ J < I.
  \end{equation*}

  We now describe the construction of $\mathcal{B}_I$ and
  $(\lambda_L^I)_{L\in \mathcal{B}_I}, (\mu_L^I)_{L\in \mathcal{B}_I}\in \mathbb{R}^{\mathcal{B}_I}$,
  which in turn determine $x_I$ and $x_I^{*}$.  Firstly, we pick $\mathcal{B}_I\subset \mathcal{A}$
  as follows:
  \begin{align*}
    \mathcal{B}_I
    &= \mathcal{G}_{n_I}(\mathcal{A}),
    && \text{if } I = [0,1),\\
    \mathcal{B}_I
    &= \bigl\{ L\in \mathcal{G}_{n_I}(\mathcal{A}) : L\subset \{ \tilde{h}_{\pi(I)} = \pm 1 \} \bigr\},
    && \text{if } I = J^{\pm } \text{ for some } J\in \mathcal{D},
  \end{align*}
  where $n_I\in \mathbb{N}_0$ is chosen sufficiently large such that $n_I > n_J$ for all $J < I$ and
  such that
  \begin{equation}\label{eq:faithful-inequality-kappa}
    |\mathcal{B}_I^{*}\cap S| > (1 - \varkappa_I'/2) |\mathcal{B}_I^{*}|.
  \end{equation}
  Secondly, we define the non-negative real numbers
  \begin{equation*}
    \lambda^I_L
    = \sqrt{2}\cdot \frac{\|h_L\|_{Y}}{\|h_I\|_{Y}}
    \qquad\text{and}\qquad
    \mu^I_L
    = \frac{1}{\sqrt{2}} \frac{|I|}{|\mathcal{B}_I^{*}|}\cdot
    \frac{\|h_L\|_{Y^{*}}}{\|h_I\|_{Y^{*}}},
    \qquad L\in \mathcal{B}_I.
  \end{equation*}
  Now we fix $(\varepsilon^I_L)_{L\in \mathcal{B}_I}\in \{\pm 1\}^{\mathcal{B}_I}$ and observe that
  \begin{equation*}
    x_I
    = \sum_{L\in \mathcal{B}_I} \varepsilon^I_L \lambda^I_L \frac{h_L}{\|h_L\|_{Y}}
    \qquad\text{and}\qquad
    x_I^{*}
    = \sum_{L\in \mathcal{B}_I} \varepsilon^I_L \mu^I_L h_L^{*}
  \end{equation*}
  is in accordance with~\eqref{eq:6}. Recall that by \Cref{lem:product-of-norms}, we have
  $\|h_K\|_{Y}\|h_K\|_{Y^{*}} = |K|$, $K\in \mathcal{D}$, and thus, we obtain
  \begin{equation*}
    \langle x_I^{*}, x_I \rangle
    = \sum_{L\in \mathcal{B}_I} \lambda^I_L \mu^I_L
    = 1,
  \end{equation*}
  which shows condition~\eqref{enu:dfn:strat-supp:iii} in \Cref{dfn:strat-supp}.

  We record the following identities relating $\tilde{h}_I$ with $x_I$ and $x_I^{*}$:
  \begin{equation*}
    x_I
    = \frac{\sqrt{2}}{\|h_I\|_{Y}} \tilde{h}_I
    \qquad\text{and}\qquad
    x_I^{*}
    = \frac{1}{\sqrt{2}}\frac{|I|}{|\mathcal{B}_I^{*}|}\cdot
    \frac{\tilde{h}_I}{\|h_I\|_{Y^{*}}}.
  \end{equation*}

  We will now discuss essential properties of our construction.  First, let
  $I\in \mathcal{D}\setminus \{ [0,1) \}$. Exploiting the nestedness of the dyadic intervals, one
  can show that $\mathcal{B}_I^{*}\cap S = \{ \tilde{h}_{\pi(I)} = \pm 1 \}\cap S$, and thus,
  by~\eqref{eq:faithful-inequality-kappa} and \cite[Lemma~4.5~(i)]{MR4145794}, we have
  \begin{equation*}
    |\mathcal{B}_I^{*}|
    \ge |\mathcal{B}_I^{*}\cap S|
    = |\{ \tilde{h}_{\pi(I)} = \pm 1 \}\cap S|
    \ge (1 - \varkappa_{\pi(I)}')\frac{|\mathcal{B}_{\pi(I)}^{*}|}{2}
    = \varkappa_{\pi(I)} \frac{|\mathcal{B}_{\pi(I)}^{*}|}{2}.
  \end{equation*}
  Moreover, by the definition of $\mathcal{B}_{I^{\pm }}$, we have
  $\operatorname{supp}(\tilde{h}_{I^{\pm }})\subset \{ \tilde{h}_I = \pm 1 \}$ for all
  $I\in \mathcal{D}$. Thus, it follows that $(\tilde{h}_I)_{I\in \mathcal{D}}$ is a
  $(\varkappa_I)_{I\in \mathcal{D}}$-faithful Haar system. If $\rho$ and $\sigma$ are chosen
  sufficiently small, then by \Cref{lem:impartially-equivalent}, $(x_I)_{I\in \mathcal{D}}$ is
  $\sqrt{2 + \eta}$-norm-dominated by $(h_I/\|h_I\|_Y)_{I\in \mathcal{D}}$ and
  $(x_I^{*})_{I\in \mathcal{D}}$ is $\sqrt{2 + \eta}$-weak*-dominated by
  $(h_I^{*})_{I\in \mathcal{D}}$, and moreover,
  $\sum_{I\in \mathcal{D}} \langle x_I^{*}, x \rangle h_I/\|h_I\|_Y$ converges in norm for all
  $x\in Y$. Finally, $(x_I^{*})_{I\in \mathcal{D}}$ is clearly biorthogonal to
  $(x_I)_{I\in \mathcal{D}}$.
\end{proof}

\begin{lem}\label{lem:rademacher-with-signs-weakly-null}
  Let $Y\in \mathcal{HH}_0(\delta)$ and assume that the sequence of Rademacher functions
  $(r_n)_{n=0}^{\infty}$ is weakly null in $Y$.  Moreover, suppose that
  $\mathcal{A}\subset \mathcal{D}$ has finite generations. For every sequence of signs
  $\varepsilon = (\varepsilon_K)_{K\in \mathcal{D}}\in \{ \pm 1 \}^{\mathcal{D}}$ and
  every $n\in \mathbb{N}_0$, put
  \begin{equation*}
    \tilde{r}_n(\varepsilon) = \sum_{K\in \mathcal{G}_n(\mathcal{A})} \varepsilon_Kh_K.
  \end{equation*}
  Then for every $y\in Y$ and $y^{*}\in Y^{*}$, we have
  \begin{equation}\label{eq:30}
    \lim_{n\to \infty} \sup_{\varepsilon\in \{ \pm 1 \}^{\mathcal{D}}} |\langle \tilde{r}_n(\varepsilon), y \rangle| = 0
    \qquad \text{and}\qquad
    \lim_{n\to \infty} \sup_{\varepsilon\in \{ \pm 1 \}^{\mathcal{D}}} |\langle y^{*}, \tilde{r}_n(\varepsilon) \rangle| = 0.
  \end{equation}
\end{lem}
\begin{proof}
  Let $\eta > 0$. Since $H_0$ is dense in $Y$, we can find $z\in H_0$ such that
  $\|y - z\|_Y\le \eta$. Then for all sufficiently large $n\in \mathbb{N}_0$ and all
  $\varepsilon\in \{ \pm 1 \}^{\mathcal{D}}$, we have
  $\langle \tilde{r}_n(\varepsilon), z \rangle = 0$ and hence
  \begin{equation*}
    |\langle \tilde{r}_n(\varepsilon), y \rangle| = |\langle \tilde{r}_n(\varepsilon), y - z \rangle|\le \eta
  \end{equation*}
  because $\|\tilde{r}_n(\varepsilon)\|_{Y^{*}}\le 1$ by \Cref{cor:norm-in-dual-space}.

  In order to prove the second equality in~\eqref{eq:30}, we first fix $y^{*}\in Y^{*}$ and
  $\varepsilon\in \{ \pm 1 \}^{\mathcal{D}}$.  By \Cref{lem:extend-finite-generations},
  there exists another collection $\hat{\mathcal{A}}\subset \mathcal{D}$ with
  finite generations such that $\mathcal{G}_n^{*}(\hat{\mathcal{A}}) = [0,1)$ and
  $\mathcal{G}_n(\mathcal{A})\subset \mathcal{G}_n(\hat{\mathcal{A}})$ for all
  $n\in \mathbb{N}_0$, and there exists a bounded Haar multiplier $R\colon Y\to Y$ with
  $\|R\|\le 1$ such that for every $n\in \mathbb{N}_0$ and $K\in \mathcal{G}_n(\hat{\mathcal{A}})$,
  we have
  \begin{equation*}
    Rh_K =
    \begin{cases}
      h_K, & K\in \mathcal{G}_n(\mathcal{A}),\\
      0, & K\in \mathcal{G}_n(\hat{\mathcal{A}})\setminus \mathcal{G}_n(\mathcal{A}).
    \end{cases}
  \end{equation*}

  Now let $(\hat{h}_I)_{I\in \mathcal{D}}$ be the faithful Haar system defined as in
  \Cref{rem:faithful-hs-from-generations-and-signs} with respect to the collection
  $\hat{\mathcal{A}}$ and the sequence of signs $\varepsilon$.  Since
  $(\hat{h}_I)_{I\in \mathcal{D}}$ is faithful, we know from \Cref{pro:A-B-bounded} that the
  operator $\hat{B}\colon Y\to Y$ defined as the linear extension of $\hat{B}h_I = \hat{h}_I$,
  $I\in \mathcal{D}$, is bounded.  Now we can write
  \begin{equation*}
    \tilde{r}_n(\varepsilon)
    = R\Bigl( \sum_{K\in \mathcal{G}_n(\hat{\mathcal{A}})} \varepsilon_Kh_K \Bigr)
    = R\Bigl( \sum_{I\in \mathcal{D}_n} \hat{h}_I \Bigr)
    = R \hat{B} r_n,
    \qquad n\in \mathbb{N}_0,
  \end{equation*}
  and since $(r_n)_{n=0}^{\infty}$ is weakly null in $Y$, it follows that
  \begin{equation}\label{eq:32}
    \lim_{n\to \infty}|\langle y^{*}, \tilde{r}_n(\varepsilon) \rangle| = 0.
  \end{equation}
  Clearly, there exists a sequence $\varepsilon\in \{ \pm 1 \}^{\mathcal{D}}$ that satisfies
  $|\langle y^{*}, \tilde{r}_n(\varepsilon) \rangle| = \sup_{\theta\in \{ \pm 1 \}^{\mathcal{D}}}|\langle y^{*}, \tilde{r}_n(\theta) \rangle|$
  for all $n\in \mathbb{N}_0$. Hence, by applying~\eqref{eq:32} to this sequence, we obtain the
  second equality in~\eqref{eq:30}.
\end{proof}

\begin{thm}\label{thm:stratrep}
  Let $Y\in \mathcal{HH}_0(\delta)$ and suppose that the sequence of Rademacher functions
  $(r_n)_{n=0}^{\infty}$ is weakly null in $Y$.  Then the normalized Haar basis
  $(h_I/\|h_I\|_Y)_{I\in \mathcal{D}}$ of $Y$ is $2$-strategically reproducible in $Y$ with
  projectional factors.
\end{thm}

\begin{proof}
  Fix $\eta > 0$.  In the following, we will describe a winning strategy for player~(II) in the game
  $\operatorname{Rep}_{(Y, (h_I / \|h_I\|_{Y}))}(2, \eta)$.  Before the game starts, player~(I)
  chooses a partition $\mathcal{D} = \mathcal{A}_1 \cup \mathcal{A}_2$.  Thus,
  $\limsup(\mathcal{A}_1)\cup \limsup(\mathcal{A}_2) = [0,1)$, and hence, we have
  \begin{equation*}
    |\limsup(\mathcal{A}_1)|\ge \frac{1}{2}\qquad \text{or} \qquad |\limsup(\mathcal{A}_2)|\ge \frac{1}{2}.
  \end{equation*}
  We may assume without loss of generality that $|\limsup(\mathcal{A}_1)| \ge \frac{1}{2}$. By
  \cite[Lemma~4.4]{MR4145794}, we can find a subset $\mathcal{A}\subset \mathcal{A}_1$ such that
  \begin{itemize}
    \item $\mathcal{G}_n(\mathcal{A})$ is finite and
          $\mathcal{G}_n(\mathcal{A})\subset \mathcal{G}_n(\mathcal{A}_1)$ for all
          $n\in \mathbb{N}_0$,
    \item $|\limsup(\mathcal{A})| \ge \frac{1}{2} - \rho$,
  \end{itemize}
  where $\rho > 0$ is a small number, to be determined later.  We will write
  $S = \limsup(\mathcal{A})$.  Moreover, fix $0 < \sigma < 1$ to be determined later, and let
  $(\varkappa_I)_{I\in \mathcal{D}}$ be a family of real numbers in $(0,1)$ such that
  \begin{equation*}
    \sum_{I\in \mathcal{D}} (1 - \varkappa_I)
    \le \sigma.
  \end{equation*}
  Define $\varkappa_I' = 1 - \varkappa_I$ for all $I\in \mathcal{D}$.

  Now consider turn $I\in \mathcal{D}$. Let $(n_J)_{J < I}$ denote a strictly increasing sequence of
  natural numbers and suppose that for each $J < I$, player~(II) has already chosen a finite set
  $\mathcal{B}_J\subset \mathcal{G}_{n_J}(\mathcal{A})$ and non-negative real numbers
  $(\lambda^J_L)_{L\in \mathcal{B}_J}$, $(\mu^J_L)_{L\in \mathcal{B}_J}$. Moreover, suppose that
  player~(I) has chosen sequences of signs
  $(\varepsilon^J_L)_{L\in \mathcal{B}_J}\in \{ \pm 1 \}^{\mathcal{B}_J}$, $J < I$, thus determining
  the systems $(x_J)_{J < I}$ and $(x_J^{*})_{J < I}$ defined in~\eqref{eq:18}, i.e.,
  \begin{equation*}
    x_J
    = \sum_{L\in \mathcal{B}_J} \varepsilon^J_L \lambda^J_L \frac{h_L}{\|h_L\|_{Y}}
    \qquad\text{and}\qquad
    x_J^{*}
    = \sum_{L\in \mathcal{B}_J} \varepsilon^J_L \mu^J_L \frac{h_L}{\|h_L\|_{Y^{*}}}.
  \end{equation*}
  In addition, we define the auxiliary $L^{\infty}$-normalized Haar system $(\tilde{h}_J)_{J < I}$
  by putting
  \begin{equation*}
    \tilde{h}_J = \sum_{L\in \mathcal{B}_J} \varepsilon^J_L h_L,\qquad J\in \mathcal{D},\ J < I.
  \end{equation*}

  At the beginning of turn $I$, player~(I) chooses $\eta_I > 0$ and spaces
  $W_I\in \operatorname{cof}(Y)$ and $G_I\in \operatorname{cof}_{w^{*}}(Y^{*})$.  Then player~(II)
  may choose a subset $\mathcal{B}_I$ of either $\mathcal{A}_1$ or $\mathcal{A}_2$ and finite
  sequences of non-negative real numbers
  $(\lambda^I_L)_{L\in \mathcal{B}_I}, (\mu^I_L)_{L\in \mathcal{B}_I}\in \mathbb{R}^{\mathcal{B}_I}$.
  In this winning strategy, player~(II) picks $\mathcal{B}_I\subset \mathcal{A}$ as follows:
  \begin{align*}
    \mathcal{B}_I
    &= \mathcal{G}_{n_I}(\mathcal{A}),
    && \text{if } I = [0,1),\\
    \mathcal{B}_I
    &= \bigl\{ L\in \mathcal{G}_{n_I}(\mathcal{A}) : L\subset \{ \tilde{h}_{\pi(I)} = \pm 1 \} \bigr\},
    && \text{if } I = J^{\pm } \text{ for some } J\in \mathcal{D},
  \end{align*}
  where $n_I\in \mathbb{N}_0$ is chosen sufficiently large such that $n_I > n_J$ for all $J < I$ and
  such that the following two conditions are satisfied:
  \begin{enumerate}[(i)]
    \item\label{it:faithful-inequality-kappa}
          $|\mathcal{B}_I^{*}\cap S| > (1 - \varkappa_I'/2) |\mathcal{B}_I^{*}|$.
    \item\label{it:dist-cofinite-dimensional-small} For every
          $\varepsilon\in \{ \pm 1 \}^{\mathcal{D}}$, the following two inequalities hold for
          $h = \sum_{L\in \mathcal{B}_I} \varepsilon_Lh_L$:
          \begin{equation}\label{eq:dist-subspaces-small}
            \operatorname{dist}_{Y}(h, W_I)
            \le \tilde{\eta}_I
            \qquad \text{and}\qquad
            \operatorname{dist}_{Y^{*}}(h, G_I)
            \le \tilde{\eta}_I,
          \end{equation}
          where
          \begin{equation*}
            \tilde{\eta}_{[0,1)}
            = \eta_{[0,1)}\cdot \frac{1}{\sqrt{2}} |S|
            \qquad\text{and}\qquad
            \tilde{\eta}_I
            = \eta_I\cdot \frac{|I|}{\sqrt{2}}\min \biggl(
            1, \varkappa_{\pi(I)} \frac{|\mathcal{B}_{\pi(I)}^{*}|}{2|I|}
            \biggr),
          \end{equation*}
          for all $I\in \mathcal{D}\setminus \{ [0,1) \}$.
  \end{enumerate}
  We have to show that such a number $n_I$ exists.  According to \cite[Lemma~4.5~(ii)]{MR4145794},
  condition~\eqref{it:faithful-inequality-kappa} is satisfied for all sufficiently large $n_I$.  In
  order to prove that~\eqref{it:dist-cofinite-dimensional-small} is also satisfied for all large
  enough $n_I$, first recall that by \Cref{rem:cof}, there are $N,M\in \mathbb{N}_0$ and
  $y_1^{*},\dots,y_N^{*}\in Y^{*}$ as well as $y_1,\dots,y_M\in Y$ such that
  \begin{equation*}
    W_I
    = \{ y_1^{*}, \dots, y_N^{*} \}_{\perp} \qquad \text{and} \qquad G_I = \{ y_1,\dots,y_M \}^{\perp}.
  \end{equation*}
  Given $h\in H_0$, the inequalities~\eqref{eq:dist-subspaces-small} are certainly satisfied if
  $|\langle y_j^{*}, h \rangle|$ is sufficiently small for all $j\in \{1,\dots,N\}$ and
  $|\langle h, y_j \rangle|$ is sufficiently small for all $j\in \{1,\dots,M\}$. Thus, it suffices
  to prove the following assertion: Put $\Gamma = \{ \tilde{h}_{\pi(I)} = \pm 1 \}$ if $I = J^{\pm}$
  for some $J\in \mathcal{D}$ (and $\Gamma = [0,1)$ if $I = [0,1)$) and, in addition, for
  $n \in \mathbb{N}_0$, put
  $\tilde{\mathcal{G}}_n = \{ L\in \mathcal{G}_{n_{\pi(I)} + 1 + n}(\mathcal{A}) : L\subset \Gamma \}$
  and
  \begin{equation*}
    \tilde{r}_n(\varepsilon)
    = \sum_{K\in \tilde{\mathcal{G}}_n} \varepsilon_Kh_K,
    \qquad \varepsilon\in \{ \pm 1 \}^{\mathcal{D}}.
  \end{equation*}
  Then for every $y^{*}\in Y^{*}$ and $y\in Y$, we claim that
  \begin{equation*}
    \lim_{n\to \infty} \sup_{\varepsilon\in \{ \pm 1 \}^{\mathcal{D}}} |\langle y^{*}, \tilde{r}_n(\varepsilon) \rangle|
    = 0
    \qquad \text{and}\qquad
    \lim_{n\to \infty} \sup_{\varepsilon\in \{ \pm 1 \}^{\mathcal{D}}} |\langle \tilde{r}_n(\varepsilon), y \rangle|
    = 0.
  \end{equation*}
  But this follows by applying \Cref{lem:rademacher-with-signs-weakly-null} to
  $\tilde{\mathcal{A}} := \bigcup_{n=0}^{\infty}\tilde{\mathcal{G}}_n$ since $(r_n)_{n=0}^{\infty}$
  is weakly null in $Y$.  Thus,~\eqref{it:faithful-inequality-kappa}
  and~\eqref{it:dist-cofinite-dimensional-small} are satisfied if $n_I$ is chosen sufficiently
  large.

  After the set $\mathcal{B}_I$ is selected, player~(II) chooses the non-negative real numbers
  \begin{equation*}
    \lambda^I_L
    = \sqrt{2}\cdot \frac{\|h_L\|_{Y}}{\|h_I\|_{Y}}
    \qquad \text{and}\qquad
    \mu^I_L
    = \frac{1}{\sqrt{2}} \frac{|I|}{|\mathcal{B}_I^{*}|}
    \cdot \frac{\|h_L\|_{Y^{*}}}{\|h_I\|_{Y^{*}}},
    \qquad L\in \mathcal{B}_I.
  \end{equation*}
  Recall that by \Cref{lem:product-of-norms}, we have $\|h_K\|_{Y}\|h_K\|_{Y^{*}} = |K|$ for all
  $K\in \mathcal{D}$, so we obtain
  \begin{equation*}
    \sum_{L\in \mathcal{B}_I} \lambda^I_L \mu^I_L
    = 1.
  \end{equation*}
  Next, player~(I) chooses signs
  $(\varepsilon^{I}_L)_{L\in \mathcal{B}_I}\in \{ \pm 1 \}^{\mathcal{B}_I}$. Observe that
  \begin{equation*}
    x_I
    = \sum_{L\in \mathcal{B}_I} \varepsilon^I_L \lambda^I_L \frac{h_L}{\|h_L\|_{Y}}
    \qquad\text{and}\qquad
    x_I^{*}
    = \sum_{L\in \mathcal{B}_I} \varepsilon^I_L \mu^I_L h_L^{*}
  \end{equation*}
  is in accordance with~\eqref{eq:18}. Then the next turn begins.

  Now we assume that the game is completed, and we record the following identities relating
  $\tilde{h}_I$ to $x_I$ and $x_I^{*}$:
  \begin{equation*}
    x_I
    = \frac{\sqrt{2}}{\|h_I\|_{Y}} \tilde{h}_I
    \qquad\text{and}\qquad
    x_I^{*}
    = \frac{1}{\sqrt{2}}\frac{|I|}{|\mathcal{B}_I^{*}|}\cdot
    \frac{\tilde{h}_I}{\|h_I\|_{Y^{*}}},\qquad I\in \mathcal{D}.
  \end{equation*}
  Like in the proof of \Cref{thm:strat-supp}, it follows from the above
  condition~\eqref{it:faithful-inequality-kappa} that
  \begin{equation}\label{eq:measure-of-BIstar}
    |\mathcal{B}_I^{*}|
    \ge \varkappa_{\pi(I)}\frac{|\mathcal{B}_{\pi(I)}^{*}|}{2},
    \qquad I\in \mathcal{D}\setminus \{ [0,1) \},
  \end{equation}
  and since $\operatorname{supp}(\tilde{h}_{I^{\pm }})\subset \{ \tilde{h}_I = \pm 1 \}$ for all
  $I\in \mathcal{D}$, this implies that $(\tilde{h}_I)_{I\in \mathcal{D}}$ is a
  $(\varkappa_I)_{I\in \mathcal{D}}$-faithful Haar system.  If $\rho$ and $\sigma$ are chosen
  sufficiently small, then by \Cref{lem:impartially-equivalent}, $(x_I)_{I\in \mathcal{D}}$ is
  impartially $(2 + \eta)$-equivalent to $(h_I/\|h_I\|_Y)_{I\in \mathcal{D}}$ in $Y$ and
  $(x_I^{*})_{I\in \mathcal{D}}$ is impartially $(2 + \eta)$-equivalent to
  $(h_I/\|h_I\|_{Y^{*}})_{I\in \mathcal{D}} = (h_I^{*})_{I\in \mathcal{D}}$ in $Y^{*}$.
  Furthermore, the inequalities~\eqref{eq:dist-subspaces-small}, \Cref{lem:product-of-norms} and
  \Cref{pro:HS-1d}~\eqref{pro:HS-1d:i} imply that $\operatorname{dist}_{Y}(x_I, W_I)\le \eta_I$ for
  all $I\in \mathcal{D}$, and together with $|S|\le |\mathcal{B}_{[0,1)}^{*}|$
  and~\eqref{eq:measure-of-BIstar}, we also have
  $\operatorname{dist}_{Y^{*}}(x_I^{*}, G_I) \le \eta_I$ for all $I\in \mathcal{D}$.  Finally, it is
  obvious that $(x_I^{*})_{I\in \mathcal{D}}$ is biorthogonal to $(x_I)_{I\in \mathcal{D}}$.
\end{proof}

%%% Local Variables:
%%% mode: latex
%%% TeX-master: "main"
%%% End:

%auto-ignore

\section{Proofs of \Cref{thm:main-result:A} and \Cref{thm:main-result:B}}
\label{sec:proofs}

Our proof of~\Cref{thm:main-result:A} is based on the following result, which adds to the framework
of strategically reproducible bases developed in~\cite{MR4145794}.  More precisely, the result
transfers \cite[Theorem~3.12 and Theorem~7.6]{MR4145794} to the setting of the primary factorization
property: It allows us to reduce the primary factorization property to the primary \emph{diagonal}
factorization property with respect to a Schauder bases which is strategically reproducible with
projectional factors.

\begin{thm}\label{thm:stratrep-primary-fact-prop}
  Let $E$ be a Banach space with a normalized Schauder basis $(e_j)_{j=1}^{\infty}$ with basis
  constant $\lambda\ge 1$. Let $C_r, C \ge 1$ and suppose that $(e_j)_{j=1}^{\infty}$ is
  $C_r$-strategically reproducible in $E$ with projectional factors and that $E$ has the $C$-primary
  diagonal factorization property with respect to $(e_j)_{j=1}^{\infty}$.  Let $Z$ be one of the
  following spaces:
  \begin{enumerate}[(i)]
    \item\label{thm:stratrep-primary-fact-prop:i} $Z = E$
    \item\label{thm:stratrep-primary-fact-prop:ii} $Z = \ell^p(E)$ for some $1\le p<\infty$
    \item\label{thm:stratrep-primary-fact-prop:iii} $Z = \ell^{\infty}(E)$ if $E$ is asymptotically
          curved with respect to $(e_j)_{j=1}^{\infty}$.
  \end{enumerate}
  Then $Z$ has the $\lambda C_rC$-primary factorization property, and hence, $\mathcal{M}_Z$ is the
  unique maximal ideal of $\mathcal{B}(Z)$. In particular, the spaces
  in~\eqref{thm:stratrep-primary-fact-prop:ii} and~\eqref{thm:stratrep-primary-fact-prop:iii} are
  primary.
\end{thm}

\begin{myproof}
  \begin{proofcase-noskip}[Case~\eqref{thm:stratrep-primary-fact-prop:i}]
    Let $T\colon E\to E$ be a bounded linear operator, and let $\eta > 0$. By
    \Cref{pro:stratrep-diag}~\eqref{pro:stratrep-diag:i}, there exists a bounded linear operator
    $D\colon E\to E$, which is diagonal with respect to $(e_j)_{j=1}^{\infty}$, such that $D$
    projectionally factors through $T$ with constant $\lambda(C_r + \eta)$ and error $\eta$.  Recall
    that by \Cref{rem:factors:proj}, we also have that $I_E - D$ projectionally factors through
    $I_E - T$ with constant $\lambda(C_r + \eta)$ and error $\eta$. Moreover, by the hypothesis, we
    know that the identity $I_E$ factors either through $D$ or through $I_E - D$ with constant
    $C + \eta$ (and error $0$). Thus, using \Cref{rem:factors} and \Cref{rem:factors:iso}, we
    conclude that for sufficiently small $\eta$, the identity $I_E$ either factors through $T$ or
    through $I_E - T$ with constant
    \begin{equation*}
      \frac{\lambda(C_r + \eta)(C + \eta)}{1 - (C + \eta)\eta},
    \end{equation*}
    and this converges to $\lambda C_rC$ as $\eta \to 0$.
  \end{proofcase-noskip}

  \begin{proofcase-noskip}[Case~\eqref{thm:stratrep-primary-fact-prop:ii}]
    For each $k\in \mathbb{N}$, let $(e_{k,j})_{j=1}^{\infty}$ be a copy of $(e_j)_{j=1}^{\infty}$
    in the $k$th component of $\ell^p(E)$. We know from~\Cref{rem:stratrep-lp-sum} that the sequence
    $(e_{k,j})_{j,k=1}^{\infty}$, enumerated as $(\tilde{e}_m)_{m=1}^{\infty}$ according to
    \Cref{rem:enumeration-of-basis}, is a Schauder basis of $\ell^p(E)$ whose basis constant is
    bounded by $\lambda$, and it is $C_r$-strategically reproducible in $\ell^p(E)$ with
    projectional factors.  Thus, by~\eqref{thm:stratrep-primary-fact-prop:i}, it suffices to show
    that $\ell^p(E)$ has the $C$-primary diagonal factorization property with respect to
    $(\tilde{e}_m)_{m=1}^{\infty}$.

    Let $D\colon \ell^p(E)\to \ell^p(E)$ be a bounded linear operator that is diagonal with respect
    to $(\tilde{e}_m)_{m=1}^{\infty}$.  Then for each $k\in \mathbb{N}$, by restricting $D$ to the
    $k$th component of $\ell^p(E)$, we obtain a bounded linear operator $D_k\colon E\to E$ that is
    diagonal with respect to $(e_j)_{j=1}^{\infty}$. By the hypothesis, the sets
    \begin{align*}
      \mathcal{N}_1 &= \{ k\in \mathbb{N} : I_E\text{ factors through } D_k \text{ with constant }C^+ \},\\
      \mathcal{N}_2 &= \{ k\in \mathbb{N} : I_E\text{ factors through } I_E - D_k \text{ with constant }C^+ \}
    \end{align*}
    satisfy $\mathcal{N}_1\cup \mathcal{N}_2 = \mathbb{N}$, so at least one of these sets is
    infinite. We assume without loss of generality that $\mathcal{N}_1$ is infinite. For each
    $k\in \mathcal{N}_1$, there exist bounded linear operators $A_k, B_k\colon E\to E$ such that
    \begin{equation*}
      I_E = A_kD_kB_k\qquad \text{and} \qquad \|A_k\|\|B_k\|\le C + \eta.
    \end{equation*}
    By rescaling, we may assume that $\|A_k\|\le \sqrt{C + \eta}$ and
    $\|B_k\|\le \sqrt{C + \eta}$. For $k\in \mathbb{N}\setminus \mathcal{N}_1$, we put
    $A_k = B_k = 0$. Now write $\mathcal{N}_1 = \{ n_1 < n_2 < n_3 < \dots \}$ and define the
    operators $Q,R\colon \ell^p(E)\to \ell^p(E)$ by
    $Q(x_k)_{k=1}^{\infty} = (x_{n_k})_{k=1}^{\infty}$ as well as
    $R(x_k)_{k=1}^{\infty} = (y_k)_{k=1}^{\infty}$ where $y_{n_k} = x_k$ for all $k\in \mathbb{N}$
    and $y_k = 0$ if $k\notin \mathcal{N}_1$. Note that $\|Q\|, \|R\|\le 1$. Finally, define
    $A,B\colon \ell^p(E)\to \ell^p(E)$ by
    \begin{equation*}
      A(x_k)_{k=1}^{\infty} = (A_kx_k)_{k=1}^{\infty}\qquad \text{and}
      \qquad B(x_k)_{k=1}^{\infty} = (B_kx_k)_{k=1}^{\infty}
    \end{equation*}
    and observe that $\|A\|,\|B\|\le \sqrt{C + \eta}$ and $I_{\ell^p(E)} = QADBR$, so
    $I_{\ell^p(E)}$ factors through $D$ with constant $C + \eta$ (and error $0$).
  \end{proofcase-noskip}
  \begin{proofcase-noskip}[Case~\eqref{thm:stratrep-primary-fact-prop:iii}]
    Let $T\colon \ell^{\infty}(E)\to \ell^{\infty}(E)$ be a bounded linear operator, and let
    $\eta > 0$.  Like in the proof for case~\eqref{thm:stratrep-primary-fact-prop:i}, we first
    diagonalize the operator $T$. Here, however, instead of using
    \Cref{pro:stratrep-diag}~\eqref{pro:stratrep-diag:i}, we apply
    \Cref{lem:curved-implies-uniformly-curved} and then \Cref{pro:stratrep-diag-ell-infty} to obtain
    a bounded linear operator $D\colon \ell^\infty(E)\to \ell^{\infty}(E)$ which is diagonal with
    respect to $(e_{k,j})_{j,k=1}^{\infty}$ such that $D$ projectionally factors through $T$ with
    constant $\lambda(C_r + \eta)$ and error $\eta$. Next, like in the proof
    of~\eqref{thm:stratrep-primary-fact-prop:ii}, we see that $\ell^{\infty}(E)$ has the $C$-primary
    diagonal factorization property with respect to $(e_{k,j})_{j,k=1}^{\infty}$. The rest of the
    proof is the same as for~\eqref{thm:stratrep-primary-fact-prop:i}.\qedhere
  \end{proofcase-noskip}
\end{myproof}

\begin{proof}[Proof of \Cref{thm:main-result:A}]
  Since the sequence $(r_n)_{n=1}^{\infty}$ is weakly null in~$Y$, we know from \Cref{thm:stratrep}
  that the normalized Haar basis $(h_I/\|h_I\|_Y)_{I\in \mathcal{D}}$ of~$Y$ is $2$-strategically
  reproducible in~$Y$ with projectional factors. Moreover,
  \Cref{thm:main-result:3}~\eqref{thm:main-result:3:iii} implies that~$Y$ has the $2$-primary
  diagonal factorization property with respect to the Haar basis. Thus, the statements in
  \Cref{thm:main-result:A} follow from \Cref{thm:stratrep-primary-fact-prop}.
\end{proof}

\begin{proof}[Proof of \Cref{thm:main-result:B}]
  First, we show that the Haar basis of~$Y$ (or, analogously, of any space~$Y_k$, $k\in \mathbb{N}$)
  has the $2/\delta$-diagonal factorization property. Let $\delta > 0$, and let $D\colon Y\to Y$ be
  a bounded Haar multiplier with $\delta$-large diagonal. Let $\gamma > 0$. Since by
  \Cref{thm:strat-supp}, the system $((h_I/\|h_I\|_Y, h_I^{*}))_{I\in \mathcal{D}}$ is strategically
  supporting in $Y\times Y^{*}$, we can apply \Cref{pro:positive-reduction}, which yields a bounded
  Haar multiplier $\tilde{D}\colon Y\to Y$ with either $\delta$-large \emph{positive} or
  $\delta$-large \emph{negative} diagonal such that~$\tilde{D}$ factors through~$D$ with constant
  $2 + \gamma$ (and error~$0$). Hence, we have $|c|\ge \delta$ for all
  $c\in \Lambda(\tilde{D})$. Now \Cref{thm:main-result:3}~\eqref{thm:main-result:3:ii} implies that
  the identity~$I_Y$ factors through $\tilde{D}$ with constant $(1/\delta)^+$ (and error~$0$). Thus,
  we have proved that~$I_Y$ factors through~$D$ with constant $(2/\delta)^+$.
  \begin{proofcase-noskip}[Case~\eqref{thm:main-result:B:i}]
    Since the sequence $(r_n)_{n=1}^{\infty}$ is weakly null in~$Y$, $k\in \mathbb{N}$, we know from
    \Cref{thm:stratrep} that the normalized Haar basis $2$-strategically reproducible in~$Y$. Hence,
    if $\delta > 0$ and $T\colon Y\to Y$ is a bounded linear operator with $\delta$-large diagonal
    with respect to the Haar basis, then by \Cref{pro:stratrep-diag}, for every $\eta > 0$, there
    exists a bounded Haar multiplier $D\colon Y\to Y$ with $\delta$-large diagonal such that $D$
    factors through $T$ with constant $2 + \eta$ and error $\eta$. Thus, the
    $4/\delta$-factorization property of the Haar basis of~$Y$ follows from $2/\delta$-diagonal
    factorization property together with \Cref{rem:factors} and
    \Cref{rem:factors:iso}. Alternatively, it also fallows from the strategic reproducibility and
    the diagonal factorization property using \cite[Theorem~3.12]{MR4145794}.
  \end{proofcase-noskip}
  \begin{proofcase-noskip}[Case~\eqref{thm:main-result:B:ii}]
    We know that the Haar basis is $2$-strategically reproducible and has the $2/\delta$-diagonal
    factorization property in each space $Y_k$, $k\in \mathbb{N}$. By \cite[Lemma~7.3]{MR4145794},
    it follows that $(h_{k,I})_{k\in \mathbb{N},I\in \mathcal{D}}$ has the $2/\delta$-diagonal
    factorization property in $\ell^p((Y_k)_{k=1}^{\infty})$ for $1\le p<\infty$. Thus, the
    $4/\delta$-factorization property follows like in Case~\eqref{thm:main-result:B:i}, again using
    \Cref{pro:stratrep-diag} for diagonalization (see \Cref{rem:stratrep-lp-sum}). Alternatively, it
    follows from \cite[Theorem~7.6]{MR4145794}.
  \end{proofcase-noskip}
  \begin{proofcase-noskip}[Case~\eqref{thm:main-result:B:iii}] In this case, the
    $4/\delta$-factorization property follows from~\cite[Theorem~3.9]{MR4299595}.\qedhere
  \end{proofcase-noskip}
\end{proof}

%%% Local Variables:
%%% mode: latex
%%% TeX-master: "main"
%%% End:

\noindent\textbf{Acknowledgments.}
The authors would like to thank Paul~F.X.~Müller and Thomas Schlumprecht for helpful discussions.
This article is part of the second author's PhD~thesis, which is being prepared at the Institute of
Analysis, Johannes Kepler University Linz.

\bibliographystyle{abbrv}%
\bibliographystyle{plain}%
\bibliography{bibliography}%

\end{document}